%% file: ordering.tex
\documentclass[tikz, sepfignums, defaultenums]{nmd/article}
\usepackage{enumitem}
\setlist[enumerate]{parsep=0pt, itemsep=0.8ex, topsep=0.8ex}

\usepackage{tikz-cd}
\usetikzlibrary{arrows}
\usetikzlibrary{shapes}
\usetikzlibrary{intersections}
\tikzset{
  commutative diagrams/.cd,
  arrow style=tikz,
  diagrams={>=stealth', line width=0.7pt},
}
\usepackage{placeins}
\usepackage{titletoc}
\contentsmargin{1.5em}
\dottedcontents{section}[2.5em]{\normalsize}{2em}{1pc}

\ifnmd
\fi

\title{Orderability and Dehn filling}

\author{Marc Culler}
\givenname{Marc}
\surname{Culler}
\address{ Dept.~of Mathematics, Statistics, and 
      Computer Science, M/C 249 \\
          University of Illinois at Chicago\\
          851 S. Morgan St. \\ 
          Chicago, IL 60607-7045, USA
}
\email{culler@math.uic.edu}
\urladdr{http://math.uic.edu/~culler}

\author{Nathan M. Dunfield}
\givenname{Nathan}
\surname{Dunfield}
\address{ Dept.~of Mathematics, MC-382 \\
          University of Illinois \\
          1409 W. Green St. \\
          Urbana, IL 61801, USA
}
\email{nathan@dunfield.info}
\urladdr{http://dunfield.info}

\arxivreference{}
\arxivpassword{}

\newcommand{\G}{G}
\newcommand{\GG}{{G^\star}}
\newcommand{\SLcharvar}[1]{X\left(#1\right)}

\newcommand{\augcharvar}[1]{\widehat{X}\left(#1\right)}
\newcommand{\augrepvar}[1]{\widehat{R}\left(#1\right)}
\newcommand{\auginc}{\hat\inc\mkern2mu}

\newcommand{\PSLRcharvar}[1]{\mathit{X}_{G}\left(#1\right)}

\newcommand{\RG}[1]{\mathit{R}_{G}\left(#1\right)}
\newcommand{\XG}[1]{\mathit{X}_{G}\left(#1\right)}
\newcommand{\RGPE}[1]{\mathit{PE}_{G}\left(#1\right)}
\newcommand{\RGlift}[1]{\mathit{PE}^{\mathit{lift}}_{G}\left(#1\right)}
\newcommand{\RGtil}[1]{\mathit{R}_{\Gtil}\left(#1\right)}
\newcommand{\RGtilPE}[1]{\mathit{PE}_{\Gtil}\left(#1\right)}
\newcommand{\GC}{G_\C}
\newcommand{\slrhoplus}{\mathfrak{sl}_2(\C)_{\rho^+}}

\newcommand{\sslash}{/\mkern-3mu/}

\newcommand{\inc}{\iota}
\newcommand{\projsp}{\P}
\newcommand{\Pone}{\P^1(\C)}
\newcommand{\Cunits}{\C^\times}
\newcommand{\Runits}{\R^\times}

\newcommand{\Csquare}{\Cunits\times\Cunits}
\DeclareMathOperator{\trans}{trans}

\newcommand{\TEL}[1]{\mathit{EL}_\Gtil\left(#1\right)}
\newcommand{\SymTEL}[1]{D_\infty\left(#1\right)}
\newcommand{\TELquo}[1]{\mathit{BL}_\Gtil\left(#1\right)}
\newcommand{\rhohyp}{\rho_{\mathit{hyp}}}
\newcommand{\euler}[1]{\mathit{Euler}\left(#1\right)}
\newcommand{\Honefree}{H_1(M; \Z)_{\mathrm{free}}}
\newcommand{\Honefreedef}{%
\rightquom{H_1(M; \Z)}{(\mathrm{torsion})}{1pt}{\big}}
\newcommand{\twisted}[3]{#1^#2\left(#3; \  \slrhoplus\right)}
\newcommand{\abstrans}[1]{\abs{\trans\left(#1\right)}}

\newcommand{\cFtil}{\widetilde{\cF}}

\newcommand{\psiF}{{\psi|_{\pi_1(F)}}}
\newcommand{\rhoF}{{\rho|_{\pi_1(F)}}}

\definecolor{locus}{HTML}{8B008B}
\definecolor{simplealex}{HTML}{48D1CC}
\definecolor{galoisgeom}{HTML}{ADF802}
\definecolor{otherparabolic}{gray}{0.2}
\definecolor{axesgray}{gray}{0.5}

\begin{document}
\begin{abstract} 
  Motivated by conjectures relating group orderability, Floer
  homology, and taut foliations, we discuss a systematic and broadly
  applicable technique for constructing left-orders on the fundamental
  groups of rational homology \3-spheres.  Specifically, for a compact
  3-manifold $M$ with torus boundary, we give several criteria which
  imply that whole intervals of Dehn fillings of $M$ have
  left-orderable fundamental groups.  Our technique uses certain
  representations from $\pi_1(M)$ into $\PSLRtilde$, which we organize
  into an infinite graph in $H^1(\partial M; \R)$ called the
  translation extension locus.  We include many plots of such loci
  which inform the proofs of our main results and suggest interesting
  avenues for future research.
\end{abstract}
\maketitle

\setcounter{tocdepth}{1}
\tableofcontents
\pagebreak 

\section{Introduction}

A group is called \emph{left-orderable} when it admits a total ordering
that is invariant under left multiplication (see
\cite{ClayRolfsen2015} for an introduction to the role of orderable groups
in topology).  We will say that a closed $3$-manifold $Y$ is
\emph{orderable} when $\pi_1(Y)$ is left-orderable.  (Technical aside:
by convention, the trivial group is not left-orderable, and so $S^3$
is not orderable.)  Our focus here is on the following question: given
a compact orientable $3$-manifold $M$ with torus boundary, which Dehn
fillings of $M$ are orderable?  We care about this question because of
its relationship with the following conjecture.

\begin{conjecture}\label{BGWconjecture}
  For an irreducible $\Q$-homology \3-sphere $Y$, the following
  are equivalent:
  \begin{enumerate}
  \item \label{item:orderable}
    $Y$ is orderable;
  \item \label{item:notL}
    $Y$ is not an $L$-space; 
  \item \label{item:taut}
    $Y$ admits a co-orientable taut foliation. 
  \end{enumerate}
\end{conjecture}
Recall from \cite{OSLensSpace2005} that an \emph{$L$-space} is a
$\Q$-homology 3-sphere with minimal Heegaard Floer homology,
specifically one where $\dim \HFhat(Y) = |H_1(Y;\Z)|$.  The
equivalence of (\ref{item:orderable}) and (\ref{item:notL}) was boldly
postulated by Boyer, Gordon, and Watson in
\cite{BoyerGordonWatson2013}, which includes a detailed discussion of
this conjecture.  The equivalence of (\ref{item:notL}) and
(\ref{item:taut}) was formulated as a question by Ozsv\'ath and Szab\'o
after they proved that (\ref{item:taut}) implies (\ref{item:notL})
\cite{OSgenusbounds2004, KazezRoberts2015, Bowden2015}.  On its face,
Conjecture \ref{BGWconjecture} is quite surprising given the disparate
nature of these three conditions, but there are actually a number of
interconnections between them summarized in Figure~\ref{fig:conj}.
Despite much initial skepticism, substantial evidence has accumulated
in favor of Conjecture~\ref{BGWconjecture}.  For example, it holds for
\emph{all} graph manifolds
\cite{HanselmanRasmussenRasmussenWatson2015, BoyerClay2014}, many
branched covers of knots in the \3-sphere (\cite{GordonLidman2014} and
references therein), as well as more than 100{,}000 small-volume
hyperbolic \3-manifolds \cite{Dunfield2015}.

\begin{figure}
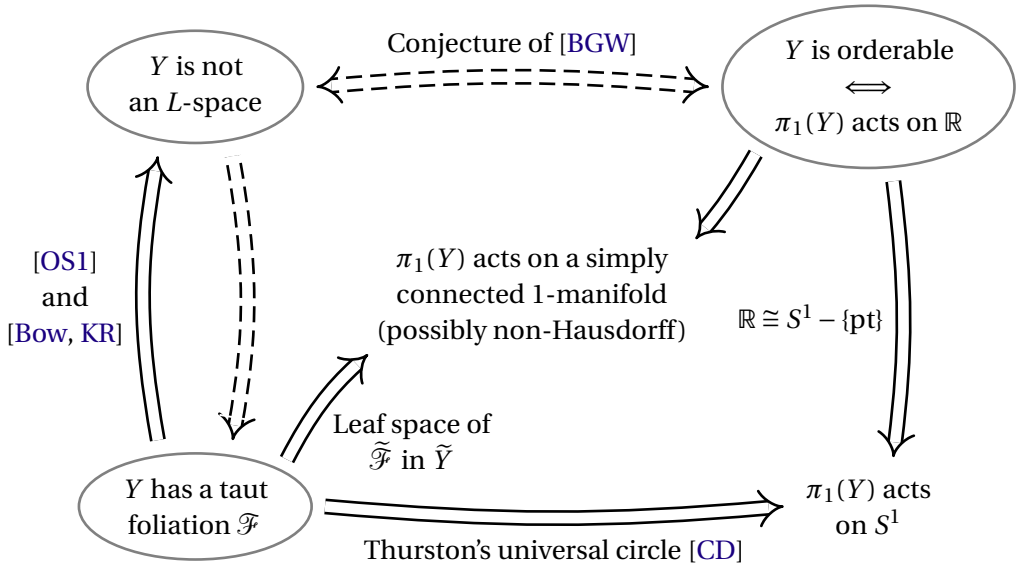

  \begin{center}
    \input plots/diagram
  \end{center}
  \caption{Some results related to Conjecture~\ref{BGWconjecture},
    which asserts the equivalence of the three circled conditions.
    Here $Y$ is an irreducible $\Q$-homology 3-sphere, all
    foliations are co-orientable, and all actions are nontrivial,
    faithful, and orientation preserving; the solid arrows are
    theorems and dotted ones conjectures.  See
    \cite{BoyerGordonWatson2013} for a complete discussion.}
  \label{fig:conj}
\end{figure}

Here, we provide further evidence for Conjecture~\ref{BGWconjecture}
by giving tools that order whole families of $\Q$-homology \3-spheres
arising by Dehn filling a fixed manifold with torus boundary.
To formulate our first main result, we
introduce a new concept: a knot exterior is called \emph{lean} when
the longitudinal Dehn filling $M(0)$ is prime and every closed
essential surface in $M(0)$ is a fiber in a fibration over $S^1$.
(See Section~\ref{sec:background} for precise definitions of standard
terminology and conventions used in this introduction.) 

\begin{restatable}{theorem}{maintheoremone}
  \label{MainTheoremOne}
  Suppose $M$ is the exterior of a knot in a $\Z$-homology
  \3-sphere.  If $M$ is lean and its Alexander polynomial
  $\Delta_M$ has a simple root on the unit circle, then there exists
  $a>0$ such that for every rational $r\in (-a, a)$ the Dehn filling
  $M(r)$ is orderable.
\end{restatable}
\noindent
In fact, with slightly more technical hypotheses, we extend this
result to $\Q$\hyp homology \3-spheres in
Theorem~\ref{MainTheoremOneOne} below.  The latter result also weakens
the requirement that $M$ is lean, replacing it by a condition
involving $\PSL{2}{\C}$-character varieties.  Combining
Theorem~\ref{MainTheoremOne} with Roberts' construction of foliations
on Dehn fillings of fibered manifolds in \cite{Roberts2001}
immediately gives:

\begin{corollary}
  Suppose $M$ is the exterior of a knot in a $\Z$-homology
  \3-sphere.  Suppose that $M$ is lean and fibers over the circle. If $\Delta_M$
  has a simple root on the unit circle, then
  Conjecture~\ref{BGWconjecture} holds for all $M(r)$ with $r \in (-a,
  a)$ for some $a > 0$.  In particular, these
  $M(r)$ are orderable and have a co-orientable taut foliation.
\end{corollary}

Our second main result is the following, and applies to branched
covers as well as Dehn fillings; see Section~\ref{sec:realtraces} for
the definition of the trace field of a hyperbolic \3-manifold.

\begin{restatable}{theorem}{maintheoremtwo}
  \label{MainTheoremTwo}
  Let $K$ be a hyperbolic knot in a $\Z$-homology 3-sphere $Y$. If the
  trace field of the knot exterior $M$ has a real embedding then:
\begin{enumerate}
  \item\label{item:branched}
    For all sufficiently large $n$, the $n$-fold cyclic 
    cover of $Y$ branched over $K$ is orderable.
  \item\label{item:fill} There is an interval $I$ of the form
    $(-\infty, a)$ or $(a, \infty)$ so that the Dehn filling $M(r)$ is
    orderable for all rational $r \in I$.
  \item \label{item:newfill} There exists $b >0$ so that for every rational
    $r \in (-b, 0) \cup (0, b)$ the Dehn filling $M(r)$ is orderable.
\end{enumerate}
\end{restatable}
\noindent
The reason the slope $0$ is excluded from the conclusion in
(\ref{item:newfill}) is that $M(0)$ might have a lens space
connect-summand and hence not be orderable. Part (\ref{item:branched})
of Theorem~\ref{MainTheoremTwo} was also proven independently by
Steven Boyer (personal communication); the lemma behind part
(\ref{item:newfill}) was pointed out to us by Ian Agol and David
Futer.

\subsection{Translation extension loci}
\FloatBarrier
We prove Theorems~\ref{MainTheoremOne} and \ref{MainTheoremTwo} by
studying representations of \3-manifold groups into the nonlinear Lie
group $\Gtil = \PSLRtilde$.  Using such representations to order
\3-manifold groups goes back at least to
\cite{EisenbudHirschNeumann1981}, and has been exploited repeatedly of
late to provide evidence for Conjecture \ref{BGWconjecture}.  Closest
to our results here, representations to $\Gtil $ were used to obtain
ordering results for Dehn surgeries on two-bridge knots in
\cite{HakamataTeragaito2014, Tran2015c}, as well as branched covers of
two-bridge knots in \cite{Hu2015, Tran2015b}.  Indeed, some of the
results on branched covers in \cite{Hu2015, Tran2015b, Gordon2016} can be viewed
as special cases of both the statement and the proof of
Theorem~\ref{MainTheoremTwo}(a).

\begin{figure}
  \begin{center}
    \hspace{3pt}%
 \pgfkeys{/matplotlibfigure, default, width=6cm, font=\scriptsize}%
 \input{plots/s841_small}

 \pgfkeys{/matplotlibfigure, default, width=6cm, font=\scriptsize}%
 \input{plots/o9_04139_small}

    \vspace{0.4cm}

 \pgfkeys{/matplotlibfigure, default, width=6cm, font=\scriptsize}%
 \input{plots/t03632_small}

 \pgfkeys{/matplotlibfigure, default, width=6cm, font=\scriptsize}%
 \input{plots/o9_30426_small}

  \end{center}
  \caption{Some translation extension loci that are discussed in
    detail in Section~\ref{sec:menagerie}.}
  \label{fig:transexs}
\end{figure}
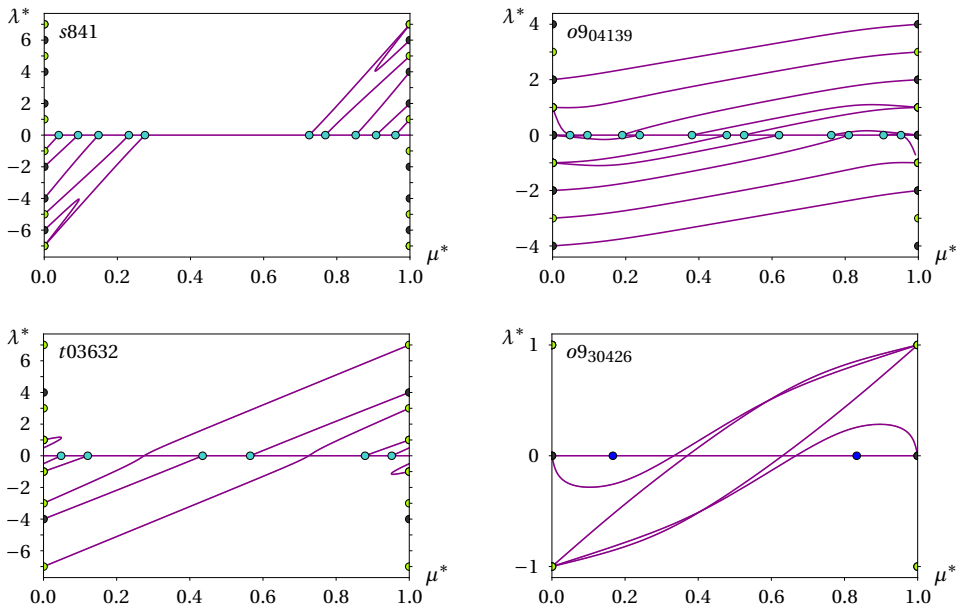

Our main contribution here is to provide a framework for
systematically studying representations to $\Gtil$ in a way tailored
for applications such as Theorems~\ref{MainTheoremOne} and
\ref{MainTheoremTwo}.  Specifically, given the exterior $M$ of a knot
in a $\Q$-homology sphere, we organize the representations of
$\pi(M) \to \Gtil $ that are elliptic on $\pi_1(\partial M)$ into a
graph in $H^1(\partial M; \R)$ called the \emph{translation extension
  locus} and denoted $\TEL{M}$.  Very roughly, the locus $\TEL{M}$ is
the image of the ``character variety of $\Gtil $ representations'' of
$\pi_1(M)$ in the corresponding object for $\pi_1(\partial M)$ under
the map induced by $\partial M \hookrightarrow M$; as such, it
parallels the $A$-polynomial story of \cite{CCGLS}.  This locus was
first studied by Khoi in his computations of Seifert volumes of
certain hyperbolic \3-manifolds \cite{Khoi2003}.  While the graph
$\TEL{M}$ is infinite, it is actually compact modulo a discrete group
of symmetries, and so it is possible to draw a picture of it: see
Figure~\ref{fig:transexs} for some examples, and
Section~\ref{sec:menagerie} for many more.

To prove Theorems~\ref{MainTheoremOne} and \ref{MainTheoremTwo}, we
give a simple criteria (Lemma~\ref{lemma:key}) which says roughly that
if the line in $H^1(\partial M; \R)$ of slope $-r$ meets $\TEL{M}$
away from the origin, then the Dehn surgery $M(r)$ is orderable.
(Lemma~\ref{lemma:branched} is our analogous result for branched
covers.)  In Section~\ref{sec:menagerie}, we use Lemma~\ref{lemma:key}
to order large intervals of Dehn fillings in some specific examples;
indeed, the conclusions of Theorems~\ref{MainTheoremOne} and
\ref{MainTheoremTwo} are much weaker than what we typically observe in
practice.

Once we establish the basic properties of these loci in
Theorem~\ref{thm:structure}, our main results are proved by using the
given hypotheses to produce at least a small arc of $\TEL{M}$ in a
certain location in $H^1(\partial M; \R)$, and then applying
Lemma~\ref{lemma:key} at many points along the arc.  For
Theorem~\ref{MainTheoremOne}, we build the arc by using
\cite{HeusenerPorti2005} to deform reducible representations
corresponding to a root of $\Delta_M$ on the unit circle to more
interesting representations in $\PSL{2}{\R}$.  In
Theorem~\ref{MainTheoremTwo}, we first use a combination of hyperbolic
geometry and algebraic geometry to produce an arc which contains (the
image of) the representation associated with the real embedding of the
trace field. The key issue of the arc's location in
$H^1(\partial M; \R)$ hinges on the result of \cite[Corollary
2.4]{Calegari2006} that for a lift of the holonomy representation of
the hyperbolic structure of $M$ to $\SL{2}{\C}$ the trace of the
longitude is $-2$ rather than $2$.

\subsection{Applicability}

While there are many cases where neither Theorem~\ref{MainTheoremOne}
or Theorem~\ref{MainTheoremTwo} applies, we next argue that some of
their hypotheses are quite generic and therefore our results should be
interpreted as providing a profusion of orderable \3-manifolds.

For example, the Alexander polynomial hypothesis of
Theorem~\ref{MainTheoremOne} holds for the vast majority of the
simpler \3-manifolds that one can tabulate: specifically, it occurs
for 81.2\% of the 1.7 million prime knots with at most 16 crossings
\cite{HosteThistlethwaiteWeeks1998} and 96.2\% of the 59{,}068
hyperbolic $\Q$-homology solid tori that can be triangulated with at
most 9 ideal tetrahedra \cite{Burton2014}.  We also looked at more
complicated knots by taking braid closures of random 10-strand braids
with 100--1{,}000 crossings (conditioned on the closure being a knot
rather than a link); of the 100{,}000 such knots we examined, some
99.87\% had Alexander polynomial with a simple root on the unit
circle.  Finally, of particular interest in light of
Conjecture~\ref{BGWconjecture} are the $L$-space knots in $S^3$, that
is, those with a non-trivial Dehn surgery producing an $L$-space.  The
Alexander polynomials of such knots have a very special form
\cite[Corollary~1.3]{OSLensSpace2005}, and it follows from
\cite{KonvalinaMatache2004} that $L$-space knots have $\Delta_M$ with
a root on the unit circle; experimentally, there is always a simple
root, but we are unable to prove this.

Turning to Theorem~\ref{MainTheoremTwo}, it is also very common for a
hyperbolic 3-manifold to have a trace field with a real embedding.
For example, Goerner~\cite{GoernerDatabase} has calculated the trace
fields of all 61{,}911 cusped manifolds that can be triangulated with
at most 9 ideal tetrahedra \cite{Burton2014}.  Among these, some
95.5\% had trace fields with a real embedding.  Indeed, about 36.3\%
of the roots of the polynomials defining these fields were real.  We
conjecture that, for any reasonable model of a random hyperbolic
\3-manifold, the trace field has a real embedding with probability
one.

In contrast, the leanness condition of Theorem~\ref{MainTheoremOne},
whose use in the proof is more technical, is hardly ubiquitous.
While it can easily be arranged by, for example, taking the exterior
of a knot in $S^2 \times S^1$ which generates
$H_1(S^2 \times S^1; \Z)$, it seems that a generic knot in $S^3$ is
not lean: work of Gabai \cite{Gabai1987} implies that a lean knot must
be fibered, and the latter condition is experimentally probability 0
in the above models of random knots (see also
\cite{DunfieldThurston2006}).  That said, the strengthened version of
Theorem~\ref{MainTheoremOne}, namely Theorem~\ref{MainTheoremOneOne},
requires the weaker hypothesis that $M$ is longitudinally rigid (see
Section~\ref{sec:alexordering}).  This condition might well be
generic---we know of only a few cases where it fails---but
unfortunately it is hard to study in bulk.

% Unsurprisingly, the 2,798 trace fields that lack a real embedding
% typically had smaller degree than the sample as a whole: mean degree
% 11.8 versus 44.1.

\subsection{Outline of the rest of the paper}

Sections~\ref{sec:background} and \ref{sec:basicfacts} review some
definitions and background results; Section~\ref{sec:background}
discusses topology, character varieties, and real algebraic geometry,
whereas Section~\ref{sec:basicfacts} is focused on the group $\Gtil$.
Section~\ref{sec:TEL} defines the translation extension locus and
states its basic properties.  Section~\ref{sec:menagerie} is the
longest and arguably the heart of the paper; it gives 12 examples of
translation extension loci and discusses their properties as they
relate to our results here.  Section~\ref{sec:structureproofs} proves
the basic structure result for these loci
(Theorem~\ref{thm:structure}), as well as Lemmas~\ref{lemma:key} and
\ref{lemma:branched}.  Sections~\ref{sec:alexordering} and
\ref{sec:realtraces} then prove Theorems~\ref{MainTheoremOne} and
\ref{MainTheoremTwo} respectively.  Section \ref{sec:realtraces}
includes Remark~\ref{rem:allL} which answers \cite[Question
6]{LidmanWatson2014} by giving an example of a hyperbolic $\Q$\hyp
homology solid torus that is not fibered and where every
non-longitudinal Dehn filling is an $L$-space.  Finally,
Section~\ref{sec:open} lists ten related open problems.

\subsection{Acknowledgements} 

We were partially supported by several US NSF grants: DMS-1105476,
DMS-1207720, DMS-1510204, and also the GEAR Network (DMS-1107452);
additionally, Dunfield was partially supported by a Simons
Fellowship. Some of our work was conducted during extended stays at
ICERM, the University of Melbourne, and the Institute for Advanced
Study.  We gratefully thank Steve Boyer for helpful conversations and
comments, as well as Ian Agol and Dave Futer for discovering
Lemma~\ref{lem:iandave} and graciously letting us include it here.
Finally, we thank the referee for providing helpful comments and
suggestions.

\section{Background}
\label{sec:background}

\subsection{Topological terminology and conventions} 

We first review some basic concepts that will be used throughout this
paper, and in the process set some standing conventions.  First, all
\3-manifolds will be assumed connected and orientable unless noted
otherwise.  A \emph{knot} $K$ in a \3-manifold $Y$ is a smoothly
embedded $S^1$ inside of $Y$.  The \emph{exterior} of $K$ is $Y$ with
an open tubular neighborhood about $K$ removed; this is a compact
\3-manifold with boundary a torus.  A $\Q$-homology \3-sphere is a
closed \3-manifold whose rational homology is the same as that of
$S^3$; a $\Z$-homology \3-sphere is defined analogously.  A
$\Q$-homology solid torus is a compact \3-manifold $M$ with boundary a
torus where $H_*(M; \Q) \cong H_*(D^2 \times S^1; \Q)$; this is
equivalent to $M$ being the exterior of a knot in some $\Q$-homology
\3-sphere.  Analogously, a $\Z$-homology solid torus is a compact
\3-manifold $M$ with boundary a torus where
$H_*(M; \Z) \cong H_*(D^2 \times S^1; \Z)$; again, this is equivalent
to $M$ being the exterior of a knot in a $\Z$-homology \3-sphere.

We call a compact orientable surface $F$ in a \3-manifold $M$
\emph{essential} when it is properly embedded, incompressible, and not
boundary parallel; here, incompressible means that
$\pi_1(F) \to \pi_1(M)$ is injective and that $F$ is not a \2-sphere
bounding a \3-ball.

\subsection{Framings and slopes}
\label{sec:framings}

For a $\Q$-homology solid torus $M$, we denote the inclusion map of
its boundary by $\inc \maps \partial M \to M$.  We define a
\emph{homologically natural framing} to be a generating set
$(\mu, \lambda)$ for $H_1(\partial M; \Z)$ where
$\inc_*(\lambda) = 0$ in the rational homology $H_1(M; \Q)$.  While
the \emph{homological longitude} $\lambda$ is defined up to sign,
there are infinitely many choices for $\mu$.

An isotopy class of unoriented essential simple closed curves in
$\partial M$ is called a \emph{slope}.  A slope can be recorded by a
primitive element in $H_1(\partial M; \Z)$ which is well-defined up to
sign.  Once we fix a framing $(\mu, \lambda)$ for $\partial M$, we
shall identify slopes with elements of $\Q \cup \{\infty\}$ via
$p/q \leftrightarrow \pm(p \mu + q \lambda)$.

\subsection{Representation and character varieties} 
\label{sec:repcharvar}

Throughout this paper, we will use $\GC$ to denote the Lie group
$\PSL{2}{\C} \cong \PGL{2}{\C}$.  We now review some basic facts about
representation and character varieties with target group $\GC$; for
details, see e.g.~\cite{HeusenerPorti2004}.  For a compact manifold
$M$, the representation space $R(M) = \Hom(\pi_1(M), \GC)$ is an
affine algebraic set in some $\C^n$.  However, we are usually only
interested in representations up to conjugacy by elements of $\GC$, so
we consider the minimal Hausdorff quotient $X(M)$ of $R(M)$ generated
by the orbits of the conjugation action of $\GC$.  Equivalently,
$X(M)$ is the invariant theory quotient $R(M) \sslash \GC$; hence
$X(M)$ is again an affine algebraic set, which is referred to as the
$\GC$-\emph{character variety} of $M$.  (These algebraic sets are not
always irreducible, but we still refer to them as ``varieties'' for
historical reasons.)  For each group element $\gamma \in \pi_1(M)$,
there is a regular function $\tr^2_\gamma \maps X(M) \to \C$ given by
$\tr^2_\gamma\left([\rho]\right) = \tr\left(\rho(\gamma)\right)^2$; we
must take the square here because the trace of a matrix in $\GC$ is
only defined up to sign.  One can always choose a finite set of
elements in $\pi_1(M)$ so that the corresponding $\tr^2$ functions
give a complete system of coordinates for the affine algebraic set
$X(M)$ \cite[Corollary 2.3]{HeusenerPorti2004}.

We will frequently regard $\GC$ as the group of orientation preserving
isometries of hyperbolic \3-space $\H^3$.  The group $\GC$ acts on
$\Pone$ by M\"obius transformations, in a way that extends the action
on $\H^3$ to the sphere at infinity
$\partial \H^3 = S^2_\infty \cong \Pone$.  A representation
$\rho \in R(M)$ is called \emph{reducible} when $\rho(\pi_1(M))$ fixes
a point in $\Pone$ under the M\"obius action of $\GC$; otherwise
$\rho$ is called \emph{irreducible}.  A character $\chi \in X(M)$ is
called \emph{reducible} if any (equivalently all) representations
$\rho$ mapping to $\chi$ are reducible, and analogously for
\emph{irreducible}.  While non-conjugate representations can have the
same character, this can only happen in the reducible case
\cite[Lemma~3.15]{HeusenerPorti2004}.

Now suppose $M$ is a compact \3-manifold with torus boundary, and let
$\inc \maps \partial M \to M$ be the inclusion map.  By restricting
representations, we get an induced regular map
$\inc^* \maps X(M) \to X(\partial M)$.  We will need the following
fact in the proof of Theorem~\ref{thm:structure}.

\begin{lemma}\label{lem:1d-image}
  The image of $\inc^* \maps X(M) \to X(\partial M)$ has complex
  dimension at most 1.
\end{lemma}
For $\SL{2}{\C}$-character varieties, rather than the $\GC$ ones we
work with here, the corresponding result is
\cite[Corollary~10.1]{CooperLong1996}.  As not every representation
$\pi_1(M) \to \GC$ lifts to $\SL{2}{\C}$ by
\cite[Theorem~1.4]{HeusenerPorti2004}, we must prove
Lemma~\ref{lem:1d-image} directly.  However, the argument is
essentially identical to the $\SL{2}{\C}$ case, and the proof may
be safely skipped at first reading.

\begin{proof}[Proof of Lemma~\ref{lem:1d-image}]

We identify $\pi_1(\partial M)$ with $\Z\oplus\Z$ via a fixed framing
$(\mu, \lambda)$.  We will view the character variety $X(\partial M)$
as the minimal Hausdorff quotient of
$R(\partial M) = \Hom(\Z\oplus\Z, \GC)$ by the conjugation action.  It
is shown in \cite[Lemma~7.4]{HeusenerPorti2005} that $R(\partial M)$
has exactly two irreducible components.  The first consists exactly of
the conjugacy class of a representation onto a Klein $4$-group whose
non-trivial elements are rotations about three mutually orthogonal
lines in $\H^3$.  The other component consists of representations that
either send both $\mu$ and $\lambda$ to non-parabolic elements with a
common axis, or to parabolic elements with a common fixed point.  By
\cite[Lemma~7.4]{HeusenerPorti2005}, this component is 4\hyp
dimensional and smooth away from the trivial representation.  We will
denote its image in $X(\partial M)$ by $D$.  The conjugacy class of a
non-parabolic representation is closed and isomorphic to the coset
space $\GC/S$, where $S$ is the stabilizer of the axis in $\GC$, and
so the conjugacy class has dimension $2$.  The conjugacy class of a
parabolic representation, on the other hand, is not closed and
contains the trivial representation in its closure.  Thus the
conjugacy class of any parabolic representation maps to the same point
in $X(\partial M)$ as the trivial representation.  These two facts
imply that the complex variety $D$ is $2$-dimensional.

Since $D$ is the only irreducible component of $X(\partial M)$ with
dimension larger than $1$, to prove the lemma it suffices to show that
if $Z$ is an irreducible component of $X(M)$ such that
$\inc^*(Z)\subset D$ then $\inc^*(Z)$ has dimension at most $1$.  For this it
is convenient to pass to a $2$-fold cover of $X(M)$.

Following \cite{Dunfield2003}, we define the \emph{augmented
  representation variety} $\augrepvar{M}$ to be the subalgebraic set
of $R(M) \times \P^1$ consisting of all pairs $(\rho, x)$ where $x$ is
a point of $\Pone$ that is fixed by the image of
$\pi_1(\partial M)$ under $\rho\circ\inc$.  On a typical irreducible
component of $R(M)$ there are generically two points fixed by the
group $\rho(\pi_1(\partial M))$, and so the projection
$(\rho, x) \mapsto \rho$ gives a regular map of degree $2$ onto an
irreducible component of $R(M)$.  There is a natural diagonal action
of $\GC$ on $\augrepvar{M}$ which acts by conjugation on the
representation $\rho$ and by the induced M\"obius transformation on
$\P^1$.  The quotient
$\augcharvar{M} = \augrepvar{M}\sslash\GC$ is the
\emph{augmented character variety} of $M$.

The augmented representation and character varieties of the boundary
torus, $\augrepvar{\partial M}$ and $\augcharvar{\partial M}$, are
defined analogously.  These augmented varieties for $\partial M$ are
in fact irreducible, since the pesky Klein 4-group representations
have no fixed points on $\Pone$ and hence are missing from
$\augrepvar{\partial M}$.  Our choice of framing $(\mu, \lambda)$
determines an identification of $\augcharvar{\partial M}$ with
$\Csquare$ as follows. If $\rho$ is given by
$$
\rho(\mu) = \pm\left(\begin{matrix}z&0\\0&z^{-1}\end{matrix}\right) \mtext{and}
\rho(\lambda) = \pm\left(\begin{matrix}w&0\\0&w^{-1}\end{matrix}\right)\,,
$$
and $\infty$ denotes the point of $\Pone$ with homogeneous coordinates
$[1:0]$, then the $\GC$-orbit of the pair $(\rho,\infty)$ is
identified with the point $(z^2, w^2)$, that is, with the pair
consisting of the holonomies of $\rho(\mu)$ and $\rho(\lambda)$.

Since the conjugacy class of any parabolic representation of
$\pi_1(\partial M)$ has the same image in $X(\partial M)$ as the
trivial representation, when $\augcharvar{\partial M}$ is identified
with $\Csquare$ in this way, any pair
$(\rho, x)\in\augrepvar{\partial M}$ where $\rho$ is parabolic will be
mapped to $(1, 1)$ by the quotient map
$\augrepvar{\partial M}\to\augcharvar{\partial M}$. Also, the deck
transformation for the branched covering
$\augcharvar{\partial M} \to D$ is given by $(z,w)\mapsto(1/z, 1/w)$.

The augmented and unaugmented character varieties fit into the
following commutative diagram
\[
\begin{tikzcd}
\augcharvar{M} \arrow{r}{\auginc^*} \arrow{d} &
   \augcharvar{\partial M} \arrow{d} \\ 
   \SLcharvar{M} \arrow{r}{\inc^*} & 
   \SLcharvar{\partial M}
\end{tikzcd}
\]
in which the vertical maps are induced by the projection
$(\rho, x) \mapsto \rho$.  Since the vertical maps are finite, it
suffices to show that for each irreducible component $\Zhat$ of
$\augcharvar{M}$, its image $\auginc^*(\Zhat)$ is at most
$1$-dimensional.

To prove this we apply the same argument used in \cite[Corollary
10.1]{CooperLong1996} in the case of $\SL{2}{\C}$.  Namely, we
consider the real $1$-form
\[\omega = \log|z|\;d\!\arg w - \log|w|\;d\!\arg z\] defined on $\Csquare$,
viewed as a real manifold. The form $\omega$ is not closed
since
\[d\omega = d\!\log|z|\wedge d\!\arg w - d\!\log|w|\wedge d\!\arg z
.\]
However, since $d\omega$ is the imaginary part of the complex $2$-form
$d\log z \wedge d\log w$, it does restrict to a closed form on any complex
curve in $\Csquare$.  Moreover, it follows from a result in Craig
Hodgson's thesis (see \cite[\S 4.4]{CCGLS}) that $\omega$ pulls back
under $\auginc^*$ to an exact $1$-form on $\Zhat$.  In fact, the
pull-back of $\omega$ is equal to ${-2\mkern1mu}d\!\vol$ where $\vol$
is the real analytic function on $\augcharvar{M}$ that assigns to
$([\rho], z)$ the volume of the representation $\rho$, as defined in
\cite[\S 2.1]{Dunfield1999}.  (In particular, there is a mysterious
cohomology class which obstructs a given complex curve in $\Csquare$
from arising as a component of the image of $\auginc^*$.)  To complete
the argument, we just observe as in $\cite{CooperLong1996}$ that since
$\omega$ is not exact on $\Csquare$, but pulls back to an exact
$1$-form on $\Zhat$, we would obtain a contradiction if
$\auginc^*(\Zhat)$ were dense in $\Csquare$.  Thus $\auginc^*(\Zhat)$
must be at most $1$-dimensional, as required.
\end{proof}

\subsection{Real points}\label{sec:realpoints}

We will need a few basic facts from real algebraic geometry; for a
general reference, see \cite{BasuPollackRoy2006}.  If $X$ is an affine
algebraic set in $\C^n$, we denote the real points $X \cap \R^n$ by
$X_\R$.  When $X$ can be cut out by polynomials with real
coefficients, we say that $X$ is \emph{defined over $\R$}; in this
case, the set $X$ is invariant under coordinate-wise complex
conjugation $\tau \maps \C^n \to \C^n$, and $X_\R$ is precisely the
set of fixed points of $\tau$.  If $X$ is a quasi-projective variety
in $\projsp^n(\C)$ that can be defined by real polynomials, then the
real points $X_\R$ are again the fixed points of the involution $\tau$
on $\projsp^n(\C)$ which acts by complex conjugation of the projective
coordinates; in any affine chart whose hyperplane at infinity is
defined by a real linear form, the points of $X_\R$ are precisely the
points of $X$ whose coordinates are real.

In real algebraic geometry, the projective space
$\projsp^n(\R)$ is isomorphic to an \emph{affine} algebraic
variety, and hence any quasi-projective variety is isomorphic to an
affine one.  When working with real algebraic varieties, it is often
natural to consider the larger collection of semialgebraic sets, that
is, those defined by polynomial inequalities, and to consider
properties such as irreducibility of a real algebraic set in that enlarged category.
The dimension of a real semialgebraic set is equal to its topological
dimension. Here, we will need only the following three results.  

\begin{proposition}[\protect{\cite[Theorem 5.43]{BasuPollackRoy2006}}]
  \label{prop:real1}
  Suppose $X$ is an affine real semialgebraic set which is closed and
  bounded.  If the dimension of $X$ is at most 1, then $X$ is
  homeomorphic to a finite graph, where graphs are allowed to have
  isolated vertices.
\end{proposition}

\begin{proposition}[\protect{\cite[Theorem 5.48]{BasuPollackRoy2006}}]
  \label{prop:real3}
  A real semialgebraic set is locally path connected.  
\end{proposition}

\begin{proposition}\label{prop:real2}
  Suppose $X$ is a complex affine algebraic curve defined over $\R$.  If
  $x_0$ is a smooth point of $X$ that lies in $X_\R$, then $x_0$ has a
  classical neighborhood in $X_\R$ which is a smooth arc.
\end{proposition}

\begin{proof}
The curve $X$ has finitely many singular points which are permuted by
$\tau$. Let $X'$ be the complementary set of smooth points.  Now
$X'$ is a smooth surface and the restriction of $\tau$ to $X'$ is an
orientation reversing involution.  Using a Riemannian metric on $X'$
which is invariant under $\tau$, it is easy to see that $X'_\R$ is a
smooth \1-manifold, proving the proposition.
 \end{proof}

\subsection{Real representations}

Throughout this paper, we will set $\G = \PSL{2}{\R}$ and
$K = \PSU_2$, where both groups are viewed as subgroups of
$\GC = \PSL{2}{\C}$.  We will also occasionally consider the subgroup
$\PGL{2}{\R}$ in $\GC$, which makes sense via the identification of
$\PSL{2}{\C}$ with $\PGL{2}{\C}$; geometrically, the subgroup
$\PGL{2}{\R}$ is the full stabilizer in $\GC$ of the copy of $\H^2$
fixed by $G$ (in particular, $\PGL{2}{\R}$ includes orientation
reversing isometries of $\H^2$).  We will view
$\RG{M} = \Hom(\pi_1(M), G)$ as a subset of $R(M)$, and we will denote
by $\XG{M}$ the image of $\RG{M}$ under the quotient map
$t: R(M) \to X(M)$.  Thus $\XG{M}\subset X_\R(M)$.  By \cite[Lemma
10.1]{HeusenerPorti2004}, in fact the set $X_\R(M)$ is the image of
$R_{\PGL{2}{\R}}(M) \cup R_K(M)$ under $t$.  Since $\XG{M}$ is the
image of a real algebraic set under a polynomial map, it is a real
semialgebraic subset of $X_\R(M)$.  Note that $\XG{M}$ is \emph{not}
the quotient of $\RG{M}$ under the action of $\G$ by conjugation, even
neglecting the issue of nonclosed orbits; rather, it is essentially
the quotient of $\RG{M}$ under conjugation by the larger group
$\PGL{2}{\R}$.  Geometrically, the point is that $\PGL{2}{\R}$, not
$\G$, is the full stabilizer in $\GC = \Isom^+(\H^3)$ of the standard
$\H^2$ in $\H^3$.  Since $G$ can be characterized as the subgroup of
$\PGL{2}{\R}$ which preserves the orientation of $\H^2$, considering
representations into $\G$ up to conjugacy in $\GC$ amounts to
forgetting the orientation on $\H^2$.  We also use $X_K(M)$ to denote
the image of $R_K(M)$ in $X_\R(M)$; in this case, the set $X_K(M)$ is
the ordinary quotient $R_K(M)/K$.  Let $S$ be the subgroup
$\PSO_2 = G \cap K \cong S^1$, which is the stabilizer of a point in
$\H^2$ under the action of $G$.  As usual, we use $X_S(M)$ to denote
the image of $R_S(M)$ in $X_\R(M)$; note here that any representation
in $R_S(M)$ factors through $H_1(M; \Z)$ since $S$ itself is abelian.
The next two lemmas will be used in the proof of
Theorem~\ref{MainTheoremTwo}.

\begin{lemma}\label{lem:real_red}
  The intersection $X_K(M) \cap \XG{M}$ is exactly $X_S(M)$.  In
  particular, if $[\rho]$ is in $X_K(M) \cap \XG{M}$ then $[\rho]$ is
  reducible over $\GC$.
\end{lemma}

\begin{proof}
Consider $[\rho] \in X_K(M) \cap \XG{M}$.  If $[\rho]$ is irreducible
over $\GC$, then any two representatives of $[\rho]$ are conjugate in
$\GC$, so we can assume that $\rho \in \RG{M}$ and that every
$\rho(\gamma)$ is elliptic or trivial as $\rho$ is also conjugate into
$R_K(M)$.  However, every subgroup of $\G$ consisting solely of
elliptic elements has a global fixed point $x_0$ in $\H^2$ (see
e.g.~\cite[Theorem~4.3.7]{Beardon1983}). The representation $\rho$
then fixes pointwise the geodesic in $\H^3$ that is perpendicular to
$\H^2$ and contains $x_0$; in particular, it is reducible over $\GC$,
contradicting our initial assumption.  So we have reduced to the case
where $[\rho]$ is reducible over $\GC$, and there we can choose the
representative $\rho$ to be diagonal.  As $[\rho]$ is in $X_K(M)$, we
have $\tr_\gamma^2(\rho)$ in $[0, 4]$ for all $\gamma \in \pi_1(M)$.
A diagonal matrix $A$ in $\GC$ with $\tr^2(A)$ in $[0, 4]$ has nonzero
entries on the unit circle, and so $\rho$ comes from a homomorphism
$\pi_1(M) \to S^1/\{\pm 1\}$.  In particular, the representation
$\rho$ is conjugate into $R_S(M)$ as desired.
\end{proof}

\begin{lemma}\label{lem:pathlift}
  The map $t \maps R(M) \to X(M)$ has the weak path lifting property,
  that is, given a path $c \maps I \to X(M)$ there is a $\ctil \maps I
  \to R(M)$ with $c = t \circ \ctil$.  The same is true for its
  restrictions $R_K(M) \to X_K(M)$, $\RG{M} \to \XG{M}$, and
  $R_{\PGL{2}{\R}}(M) \to  X_{\PGL{2}{\R}}(M)$.   Moreover, if $c(0)$
  is an irreducible character, then we can require $\ctil(0)$ to be
  any specified representation in $t^{-1}\left(c(0)\right)$.   
\end{lemma}

\begin{proof}
With regards to the main claim, for the group $K$ this is
\cite[Section II.6]{Bredon1972}, for the group $\GC$ this is
\cite[Corollary 3.3]{KraftPetrieRandall1989}, and the other two cases
follow from \cite[Lemma 2.1]{BiswasLawtonRamras2015}.  The fact that
we can specify $\ctil(0)$ when $c(0)$ is irreducible is simply because
in this case everything in $t^{-1}\left(c(0)\right)$ is actually
conjugate.
\end{proof}

The next lemma is a comforting fact, but it is not needed for any of
the main results in this paper; indeed, we use it only in
Lemma~\ref{lem:smallnoideal}, which is just a remark to justify a claim
about the examples in Section~\ref{sec:menagerie}; you should
therefore skip the proof at first reading.

\begin{lemma}\label{lem:closed}
  The subsets $X_K(M)$ and $\XG{M}$ are closed in $X(M)$ in the
  classical topology.  
\end{lemma}

\begin{proof}
For ease of notation, we set $\GG = \PGL{2}{\R}$ and
$\Gamma = \pi_1(M)$, and also surpress the manifold $M$ from our
representation and character varieties.  The subset $X_K$ is compact
since it is the image of the compact set $R_K$ under a continuous map,
and so we turn immediately to $X_G$.  Given $\chi \in \Xbar_G$, which
must be in $X_\R$, we need to show that $\chi$ is in $X_G$.  We give
separate arguments depending on whether $\chi$ is reducible over
$\GC$.

First, suppose $\chi$ is reducible over $\GC$.  Let $\rho$ be a
diagonal representation into $\GC$ with character $\chi$.  The
top-left entry of $\rho$ gives a homomorphism
$\psi \maps \Gamma \to \Cunits/\{\pm 1\}$.  Since $\chi \in \Xbar_G$,
we have that $\tr^2_\gamma(\rho) \in [0, \infty)$ for all
$\gamma \in \Gamma$.  Consequently, all $\psi(\gamma)$ are in
$S^1 \cup \Runits$.  In fact, the image $\psi(\Gamma)$ must be
contained entirely in one of $S^1$ or $\Runits$, as otherwise we can
easily find a $\psi(\gamma)$ that is not in $S^1 \cup \Runits$.  If
$\psi(\Gamma)$ is in $S^1$, then $\rho$ is conjugate into $S \leq G$,
and if instead $\psi(\Gamma)$ is in $\Runits$ then $\rho$ is already
in $R_G$.  Thus, when $\chi$ is reducible we have shown that
$\chi \in X_G$ as desired.

Suppose instead that $\chi$ is irreducible over $\GC$.  By \cite[Lemma
10.1]{HeusenerPorti2004}, we need to consider two cases, depending on
whether $\chi$ is in the image of $R_K$ or $R_\GG$.  To start, suppose
$\chi$ can be realized by a $\rho \in R_K$; in particular,
$\rho(\Gamma)$ fixes a point $x_0 \in \H^3$.  As $\rho$ is
irreducible, there must be $\gamma_1$ and $\gamma_2$ in $\Gamma$ where
the $\rho(\gamma_i)$ are elliptic elements with rotation axes $L_i$
and $L_1 \cap L_2 = \{x_0\}$.  By Proposition~\ref{prop:real3} and
Lemma~\ref{lem:pathlift}, we can approximate $\rho$ by an
irreducible $\rho'$ in $R$ whose character $\chi'$ is in $X_G$ and
where the $\rho'(\gamma_i)$ are still elliptic with axes $L_i'$ very close
to the $L_i$.  As $\rho'$ is conjugate into $R_G$, there is a totally
geodesic plane $P'$ preserved by $\rho'(\Gamma)$, and the axes $L_i'$
must meet $P'$ in right angles; in particular, the angle between $L_1'$
and $L_2'$, as measured along their perpendicular bisector (which is
contained in $P'$), is $0$.  For $\rho'$ close enough to $\rho$, this
is impossible as $L_1 \cap L_2 = \{x_0\}$.  So we cannot have $\chi$
in $X_K$.

Thus our final case is when $\chi$ is irreducible and in $X_\GG$.  As
$X_\GG$ is locally path connected by Proposition~\ref{prop:real3} and
$\chi$ is a limit point of $X_\G \subset X_\GG$, we can find a path
$\chi_t$ from $\chi_0$ in $X_G$ to $\chi$.  Applying
Lemma~\ref{lem:pathlift} to $R_\GG \to X_\GG$, we lift
$\chi_t$ to a path $\rho_t$ starting at $\rho_0 \in R_G$.  As $G$ is a
connected component of $\GG$ and each $\rho_0(\gamma) \in G$, it
follows by continuity that $\rho_1(\gamma)$ is also in $G$.  Thus
$\rho_1$ is in $R_G$ and so $\chi$ is in $X_G$ as desired, proving the
lemma.
\end{proof}

\section{Basic facts about $\PSLRtilde$}
\label{sec:basicfacts}

For the group $\G = \PSL{2}{\R}$, consider its universal covering Lie group
$\Gtil = \PSLRtilde$, which is also its universal central extension
(see \cite[\S 5]{Ghys2001} or \cite[\S 2.3.3]{Calegari2009}):
\[
0 \to \Z \to \Gtil \xrightarrow{p} \G \to 1
\]
Concretely, we realize $\Gtil$ as follows. We identify
$S^1 = \projsp^1(\R)$ with $\R/\Z$ and view the quotient map as the
universal covering map $\R \to \projsp^1(\R)$.  The projective action
of $\G$ on $\projsp^1(\R)$ is faithful, so we identify $\G$ with its
image subgroup in $\Homeo^+(S^1)$.   Every
homeomorphism of $\projsp^1(\R)$ in $G$ lifts to countably many
homeomorphisms of $\R$. We define $\Gtil$ to be the subgroup of
$\Homeo^+(\R)$ consisting of all lifts of elements of $\G$.
The kernel of $p \maps \Gtil \to \G$, which is also the center of
$\Gtil$, is the deck group of $\R \to \projsp^1(\R)$, namely the group
of integer translations. We let $s$ be the element of the center which
acts by $x \mapsto x + 1$, and write elements of the center
multiplicatively as $s^k$.  An element of $\Gtil$ is called elliptic,
parabolic, or hyperbolic when its image in $\G$ is of that type.  The
disjoint partition of $\G$ into elliptic, parabolic, hyperbolic, and
trivial elements means that $\Gtil$ is similarly partitioned into
elliptic, parabolic, hyperbolic, and central elements.

\subsection{Translation number}
\label{sec:transnum}

An important concept for us is the \emph{translation number} of an
element $\gtil \in \Gtil$, which is defined as
\begin{equation}
\label{eq:transdef}
\trans(\gtil) = \lim_{n \to \infty} \frac{\gtil^n(x) - x}{n}
\mtext{for some $x \in \R$.}
\end{equation}
This is well-defined since the value of the limit is independent
of the choice of $x$.

Here are some key properties of the translation number; see \cite[\S
5]{Ghys2001} or \cite[\S 2.3.3]{Calegari2009} for extensive background
and details.  First, the map $\trans:\Gtil\to\R$ is continuous and is
constant on conjugacy classes in $\Gtil$.  Also, it is a homogenous
quasimorphism for $\Gtil$ in the sense discussed in
Section~\ref{sec:milnor-wood} below.  Considering the center
$Z(\Gtil) = \pair{s}$ as above, we have $\trans(s^k) = k$, and
moreover $\trans(\gtil \cdot s^k) = \trans(\gtil) + k$ for any $\gtil$
in $\Gtil$.

Since they map to elements in $G$ that have a fixed point in
$\projsp^1(\R)$, all parabolic and hyperbolic elements of $\Gtil$ have
integral translation numbers.  In contrast, any real number arises
as the translation number of an elliptic element. Moreover, if $\gtil$
is an elliptic element of $\Gtil$, then $2 \pi \trans(\gtil)$ is equal,
modulo $2\pi$, to the rotation angle of $p(\gtil)$ at its unique fixed
point in $\H^2$.

\subsection{The Euler class}
\label{sec:euler}

Given a group $\Gamma$ and a representation $\rho:\Gamma\to \G$, the
Euler class $\euler{\rho}\in H^2(\Gamma; \Z)$ is a complete
obstruction to lifting $\rho$ to a representation
$\rhotil:\Gamma\to\Gtil$ such that $p\circ\rhotil = \rho$.  Here is a
review of its definition; see e.g.~\cite[\S 6.2]{Ghys2001} for
details. Choose an arbitrary section $\sigma:\Gamma\to\Gtil$, that is,
a function satisfying $p\circ\rhotil = \rho$.  Define a
function $\phi_\sigma \maps \Gamma\times\Gamma\to \Z$ by
$$
s^{\phi_\sigma(\gamma_1, \gamma_2)} =
\sigma(\gamma_1)\sigma(\gamma_2)\sigma(\gamma_1\gamma_2)^{-1}
\mtext{where $Z(\Gtil) = \pair{s}$.}
$$
Associativity of group multiplication implies that
$\phi_\sigma$ satisfies the $2$-cocycle relation
$$
\phi_\sigma(\gamma_2,\gamma_3) - \phi_\sigma(\gamma_1\gamma_2, \gamma_3)
+ \phi_\sigma(\gamma_1, \gamma_2\gamma_3) - \phi_\sigma(\gamma_1, \gamma_2) = 0 .
$$
We define $\euler{\rho}$ to be the class in $H^2(\Gamma; \Z)$
represented by $\phi_\sigma$.  To see that this is well-defined, note
that if $\sigma'$ is another section, then
$\phi_\sigma - \phi_{\sigma'}$ is the coboundary of the 1-cochain
$\tau \maps \Gamma \to \Z$ determined by
$s^{\tau(\gamma)} = \sigma(\gamma) \sigma'(\gamma)^{-1}$.  Now a
section $\sigma$ is actually a lift of the representation $\rho$ when
the \2-cocycle $\phi_\sigma$ is identically zero; if $\phi_\sigma$ is
merely a coboundary, say $\phi_\sigma = \delta(\tau)$, then the
section $\sigma'$ determined by
$\sigma'(\gamma) = \sigma(\gamma) s^{-\tau(\gamma)}$ has
$\phi_{\sigma'} = 0$ on the nose. Thus a lift of $\rho$ exists
precisely when the cohomology class $\euler{\rho}$ vanishes.

Now suppose that $\rho_t$ is a continuous path of representations
$\Gamma \to \G$.  We
may choose a continuous family $\sigma_t$ of sections, for example by
choosing generators for $\Gamma$ and defining $\sigma_t(\gamma)$ in terms
of a fixed representation of $\gamma$ as a word in the generators.
This gives a continuous $1$-parameter family of cocycles.  In the general
setting, since the coboundaries are a closed subspace of the cocycles,
this implies that the map $\rho \mapsto \euler{\rho}$ is continuous.
In our setting, this means that for any $3$-manifold $M$ the Euler
class is constant on connected components of $\RG{M}$.

\subsection{Parameterizing lifts} \label{sec:paramlifts}

When $\rho \maps \Gamma \to \G$ lifts to $\Gtil$, there are many
lifts. Specifically, when $\rho$ lifts, the set of all lifts is
a 1-dimensional affine space over $H^1(\Gamma; \Z)$.  Concretely,
given some lift $\rhotil \maps \Gamma \to \Gtil$ and a
$\phi \in H^1(\Gamma; \Z)$, then, taking
$Z(\Gtil) = \langle s\rangle$, we can construct another lift
$\phi \cdot \rhotil$ via
$\gamma\mapsto \rhotil(\gamma) s^{\phi(\gamma)}$, where we are viewing
$\phi\in H^1(\Gamma; \Z)$ as a homomorphism $\Gamma \to \Z$.
Conversely, if $\rhotil_1$ and $\rhotil_2$ are two lifts of $\rho$, then
we claim that they differ by some $\phi \in H^1(\Gamma; \Z)$.  Since
$p\circ\rhotil_1(\gamma) = p\circ\rhotil_2(\gamma)$ for all $\gamma\in\Gamma$,
we have $\rhotil_1(\gamma) = \rhotil_2(\gamma)s^{\phi(\gamma)}$ for
some well-defined function $\phi \maps \Gamma\to\Z$.  To see that $\phi$ is
a homomorphism, note that
$$
\rhotil_1(\gamma_1\gamma_2)s^{\phi(\gamma_1\gamma_2)} =
\rhotil_2(\gamma_1\gamma_2) =
\rhotil_1(\gamma_1)s^{\phi(\gamma_1)}\rhotil_1(\gamma_2)s^{\phi(\gamma_2)} =
\rhotil_1(\gamma_1)\rhotil_1(\gamma_2)s^{\phi(\gamma_1) + \phi(\gamma_2)}
$$
which implies that $\phi(\gamma_1\gamma_2) = \phi(\gamma_1) + \phi(\gamma_2)$.
\subsection{Representations of $\Z^2$}
\label{sec:repsZZ}

For $\Lambda = \Z^2$, consider the set of representations
$\RGtil{\Lambda} = \Hom(\Lambda, \Gtil)$.  A representation
$\rhotil \in \RGtil{\Lambda}$ is called elliptic, parabolic, or
hyperbolic if the image group $\rhotil(\Lambda)$ contains an element
of the corresponding type.  Since $\Lambda$ is abelian, every
non-trivial element of $\rhotil(\Lambda)$ must be of the same type.
Thus these categories are disjoint; the remaining representations
which are not in any of these categories are called central since
$\rhotil(\Lambda)$ lies there.  For a fixed
$\rhotil \in \RGtil{\Lambda}$, we get a map
$(\trans \circ \rhotil) \maps \Lambda \to \R$.  The map
$\trans \circ \rhotil$ is actually a homomorphism; this is because a
homogenous quasimorphism is actually a homomorphism on any abelian
subgroup (see \cite[Prop.~2.65]{Calegari2009} or
\cite[Theorem~6.16]{Ghys2001}).

Identifying $\Hom(\Lambda, \R)$ with
$H^1(\Lambda; \R)$, we get a map
\[
\trans \maps \RGtil{\Lambda} \to H^1(\Lambda; \R)
   \mtext{defined by $\rhotil \mapsto \trans \circ \rhotil$.}
\]
This map is far from injective: any parabolic or hyperbolic element of
$\Gtil$ has an integral translation number, and it follows easily that
the preimage of any class in $H^1(\Lambda; \Z)$ contains many
nonconjugate parabolic and hyperbolic representations. However, for
elliptic and central representations, the homomorphism
$\trans(\rhotil)$ is a complete conjugacy invariant.  In particular,
it is easy to see that:
\begin{lemma}\label{lem:trans_on_elliptic}
  Suppose $\rhotil \in \RGtil{\Lambda}$ is elliptic or
  central.  If $\trans(\rhotil)(\nu) = 0$ for some $\nu \in
  \Lambda$, then $\rhotil(\nu) = 1$.  
\end{lemma}

\section{Translation extension loci} 
\label{sec:TEL}

We will now define the translation extension locus, which is the
central object in this paper.  Let $M$ be an irreducible $\Q$-homology
solid torus, and let $\inc \maps \partial M \to M$ be the inclusion
map.  Inside $\RGtil{M} = \Hom\left(\pi_1(M), \Gtil\right)$, let
$\RGtilPE{M}$ be the subset of representations whose restriction to
$\pi_1(\partial M)$ is either elliptic, parabolic, or central in the
sense of Section~\ref{sec:repsZZ}.  Consider the composition
\[
\RGtil{M} \xrightarrow{\inc^*} \RGtil{\partial M}
\xrightarrow{\trans} H^1(\partial M; \R).
\]
The closure in $H^1(\partial M; \R)$ (with respect to the vector space
topology) of the image of $\RGtilPE{M}$ under $\trans \circ \inc^*$ is
called the \emph{translation extension locus} and denoted $\TEL{M}$.
We distinguish two special kinds of points of $\TEL{M}$.  First, those
which are not in the image of $\RGtilPE{M}$, but only its closure, are
called \emph{ideal points}. Second, those coming from elements of
$\RGtilPE{M}$ which restrict to parabolic representations in
$\RGtil{\partial M}$ are called \emph{parabolic}; such points
necessarily lie on the integer lattice $H^1(\partial M; \Z)$.  The
translation extension locus was first considered by Khoi
\cite{Khoi2003} in his work on computing Seifert volumes of hyperbolic
\3-manifolds.

Let $T = \inc^*\left(H^1(M; \Z)\right) \subset H^1(\partial M; \R)$.
Consider the group of affine isomorphisms of $H^1(\partial M; \R)$
generated by the map $x \mapsto -x$ together with all translations by
elements of $T$.  As $T$ is isomorphic to $\Z$, this affine group is
isomorphic to an infinite dihedral group whose action on
$H^1(\partial M; \R)$ preserves the line containing $T$; we will
denote this dihedral group by $\SymTEL{M}$.

\subsection{Coordinates and lines}\label{subsection: coordinates}
It will be helpful to have concrete coordinates for the translation
extension locus. To this end, fix a homologically natural framing
$(\mu, \lambda)$ for $H_1(\partial M; \Z)$ as discussed in
Section~\ref{sec:framings}.  We now identify $H^1(\partial M; \R)$
with $\R^2$ by using the basis $(\mu^*, \lambda^*)$ that is
algebraically dual to the basis $(\mu, \lambda)$ of
$H_1(\partial M; \R)$, that is, $\mu^*(\mu) = \lambda^*(\lambda) = 1$
and $\mu^*(\lambda) = \lambda^*(\mu) = 0$.  Note that while $\lambda$
is unique up to sign and $\mu$ depends on our choice of framing, it is
$\mu^*$ that is unique (up to sign) and $\lambda^*$ that depends on
the framing; geometrically, the point is that $\mu^*$ is the
Poincar\'e dual of $\pm \lambda$.  Let $k \in \N$ be the order of
$\inc_*(\lambda)$ in $H_1(M; \Z)$; by Poincar\'e duality, the number
$k$ is also the index of $\pair{\inc_*(\mu)}$ in
$\Honefree = \Honefreedef \cong \Z$.  Hence, in our coordinates, the
subgroup $T = \inc^*\left(H^1(M; \Z)\right)$ is the points $(kn, 0)$
for $n \in \Z$.  Moreover, the group $\SymTEL{M}$ consists of
horizontal translations by shifts in $k\Z$ and $\pi$\hyp rotations
about every point of the form $(kn/2, 0)$ for $n \in \Z$.

To state our tool for constructing orders, we need the following
concept.  Given a slope $r$ on $\partial M$, which we can specify by a
primitive element $\gamma \in H_1(\partial M; \Z)$, define the line
$L_r = L_\gamma$ to be the subspace of $H^1(\partial M; \R)$
consisting of linear functionals that vanish on the $1$-dimensional
subspace of $H_1(\partial M; \R)$ determined by $\gamma$.  Thus the
line $L_\infty = L_\mu$ is the span of $\lambda^*$, which is the
vertical axis in our coordinates, and the line $L_0 = L_\lambda$ is
the span of $\mu^*$, which is the horizontal axis.  In general,
$L_{r}$ is a line through the origin in $\R^2$ of slope $-r$.

\subsection{Key results}

Here is the basic structural result about $\TEL{M}$, which is
roughly that it is a family of immersed arcs invariant under
$\SymTEL{M}$, such that the quotient is a finite
graph.

\begin{theorem}\label{thm:structure}
  The extension locus $\TEL{M}$ is a locally finite union of analytic
  arcs and isolated points.  It is invariant under
  $\SymTEL{M}$ with quotient homeomorphic to a finite graph.
  The quotient contains finitely many points which are ideal or
  parabolic in the sense defined above. The locus $\TEL{M}$ contains
  the horizontal axis $L_\lambda$, which comes from representations to
  $\Gtil$ with abelian image.
\end{theorem}
Moreover, here are our key tools for constructing orders.

\begin{lemma}\label{lemma:key}
  Suppose $M$ is a compact orientable irreducible \3-manifold with
  $\partial M$ a torus, and assume the Dehn filling $M(r)$ is
  irreducible.  If $L_r$ meets $\TEL{M}$ at a nonzero
  point which is not parabolic or ideal, then $M(r)$ is orderable.
\end{lemma}

\begin{lemma}\label{lemma:branched}
  Suppose $K$ is a knot in a $\Z$-homology \3-sphere $Y$ whose
  exterior $M$ is irreducible.  Let $(\mu, \lambda)$ be a
  homologically natural framing with $M(\mu) = Y$.  Assume also that
  the $n$-fold cyclic cover $\Ytil$ of $Y$ branched over $K$ is
  irreducible.  If the vertical line $\mu^* = 1/n$ meets $\TEL{M}$ at
  a point which is not ideal, then $\Ytil$ is orderable.
\end{lemma}
The proofs of Theorem~\ref{thm:structure}, Lemma~\ref{lemma:key}, and
Lemma~\ref{lemma:branched} occupy Section~\ref{sec:structureproofs}.
Before tackling them, we show pictures of various $\TEL{M}$ to get a
feel for these objects.

\section{A menagerie of translation extension loci}
\label{sec:menagerie}

\begin{figure}
  \begin{center}
 \pgfkeys{/matplotlibfigure, default, width=11cm}%
 \input{plots/m016}

  \end{center}
  
  \vspace{-0.4cm}

  \caption{The translation extension locus for $m016$.  See
    Section~\ref{sec:menagerie} for how to read this and subsequent
    plots.}
  \label{fig:m016}
\end{figure}
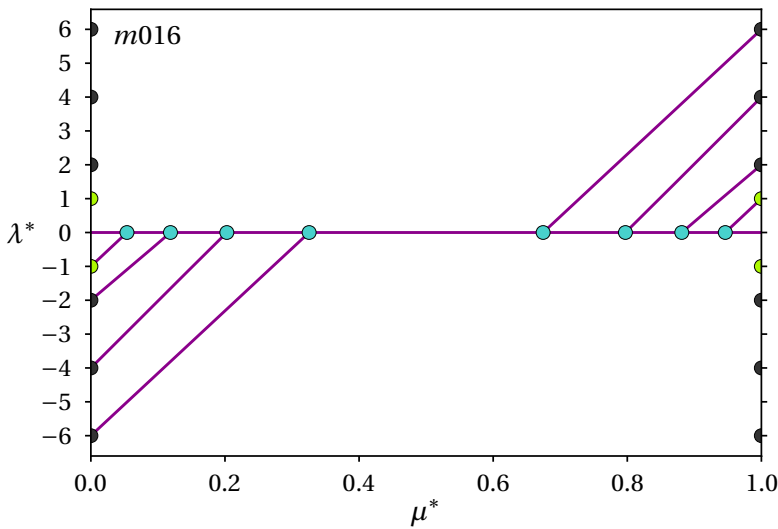

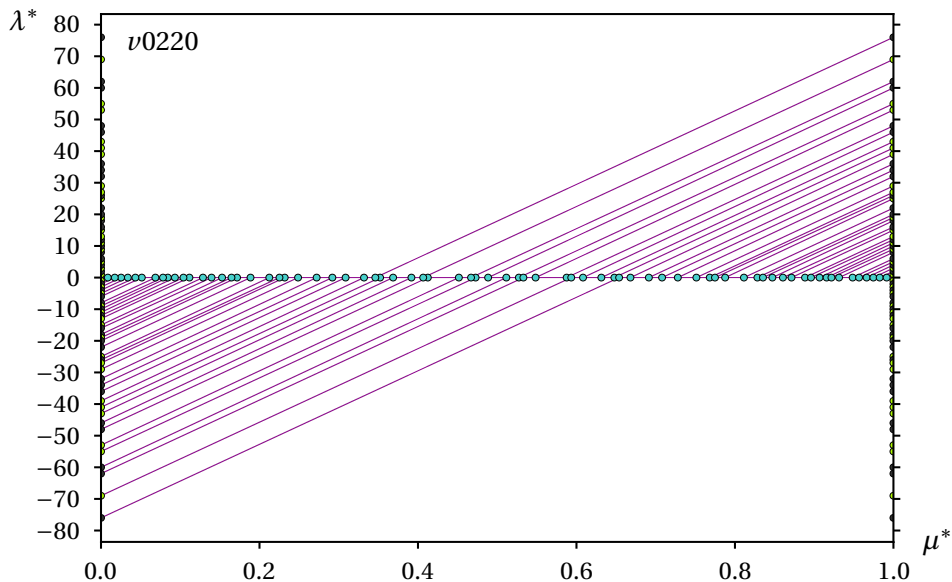
\begin{figure}
  \begin{center}
 \pgfkeys{/matplotlibfigure, default, width=13cm}%
 \input{plots/v0220}

  \end{center}

  \caption{ This extension locus follows the same basic pattern as
    Figure~\ref{fig:m016}, but there are some 72 diagonal arcs each
    joining a parabolic point to an Alexander point.  The manifold
    here is $M = v0220$, which is the exterior of the knot
    $k7_6 = T(7,-17,2,1)$ in $S^3$ found in
    \cite{ChampanerkarKofmanPatterson2004}.  The framing is such that
    $M(\mu) = S^3$ and $M(-117)$ is the lens space $L(117, 43)$; in
    SnapPy's default framing $\mu = (1,0)$ and $\lambda = (-116, 1)$.
    The manifold $M$ fibers over the circle with fiber of genus 47.
    Here, we can use Lemma~\ref{lemma:key} to order $M(r)$ for all
    $r \in (-75, \infty)$; in contrast, the interval of non-$L$-space
    slopes is $(-93, \infty)$.  It is remarkable how complicated
    $\TEL{M}$ is given that $M$ has an ideal triangulation with only
    seven tetrahedra!}
  \label{fig:v0220}
\end{figure}

\begin{figure}
  \begin{center}
 \pgfkeys{/matplotlibfigure, default, width=11cm}%
 \input{plots/o9_34801}

  \end{center}

  \caption{Like those in Figures~\ref{fig:m016} and \ref{fig:v0220},
    this locus consists of arcs that run between parabolic and
    Alexander points, but a key difference is that the parabolic
    points lie on the horizontal axis.  The manifold $M = o9_{34801}$
    here is the exterior of a genus 2 fibered knot in $S^3$, and as
    usual $M(\mu) = S^3$.  (In SnapPy's default framing $\mu = (1,0)$
    and $\lambda = (-1, 1)$.)  Using Lemma~\ref{lemma:key}, we can
    order $M(r)$ for $r \in [-0.36, 3.6)$, where the endpoints of the
    interval are approximate.  In contrast, the interval of
    non-$L$-space slopes is $(-\infty, \infty)$ since the Alexander
    polynomial $t^4 - 2 t^3 + t^2 - 2 t + 1$ does not satisfy the
    condition of \cite[Corollary~1.3]{OSLensSpace2005}. This example
    illustrates the difficulty of strengthening the proof of
    Theorem~\ref{MainTheoremOne} to give a lower bound on the size of
    the interval $(-a, a)$ in the conclusion.}
  \label{fig:o9_34801}
\end{figure}
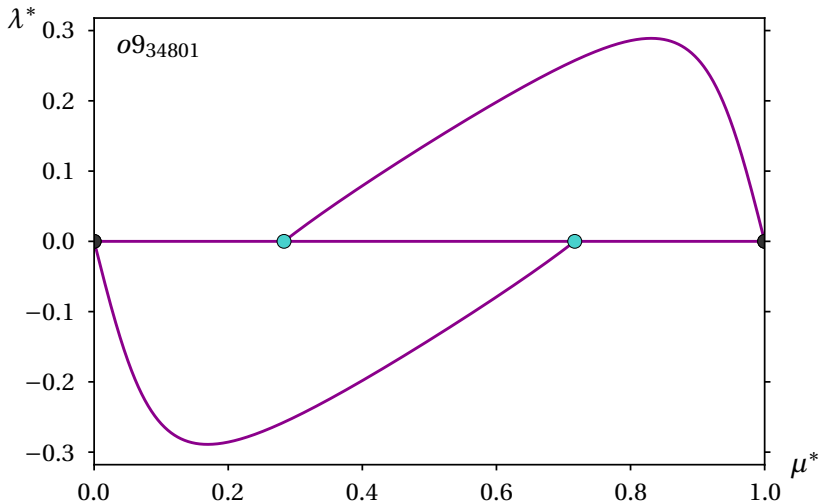

\begin{figure}
  \begin{center}
 \pgfkeys{/matplotlibfigure, default, width=11cm}%
 \input{plots/t11462}

  \end{center}

  \caption{This locus has arcs that run between two parabolic points,
    rather than from parabolic to Alexander points.  The manifold
    $M = t11462$ is the exterior of a genus 3 fibered knot in $S^3$,
    namely $k8_{249}$ from \cite{ChampanerkarKofmanMullen2014}.  (In
    SnapPy's default framing $\mu = (1,0)$ and $\lambda = (3, 1)$, and
    as usual $M(\mu) = S^3$.)  Using Lemma~\ref{lemma:key}, we claim
    that we can order $M(r)$ for $r$ in
    $(-2, -1) \cup (-1, 2) \cup [a, \infty)$ where $a \approx 4.84$.
    For example, the arc labeled $A$ in the figure gives orderings for
    $r\in(-2, -1)$, and the arc labeled $B$ shown gives orderings for
    $r\in[a, \infty)$.  The translates of $A$ by positive shifts
    contribute the intervals $(-2/k, \, -1/k)$ for $k \geq 1$, as do
    all the translates of $B$ by negative shifts; the union of these
    inverals is $(-2, -1) \cup (-1, 0)$.  The other translates of $A$
    and $B$ contribute half-open intervals that contain, but are
    slightly larger than, $[1/k, \, 2/k)$ for $k \geq 1$; the union of
    these is $(0, 2)$. The interval of non-$L$-space slopes is
    $(-\infty, \infty)$ since the Alexander polynomial
    $t^6 - 2t^5 + 3t^4 - 5t^3 + 3t^2 - 2t + 1$ does not satisfy the
    condition of \cite[Corollary~1.3]{OSLensSpace2005}. This example
    illustrates the difficulty of strengthening the proof of
    Theorem~\ref{MainTheoremTwo}(\ref{item:fill}) to give an interval
    $(a, \infty)$ where $a$ is bounded above. }
  \label{fig:t11462}
\end{figure}
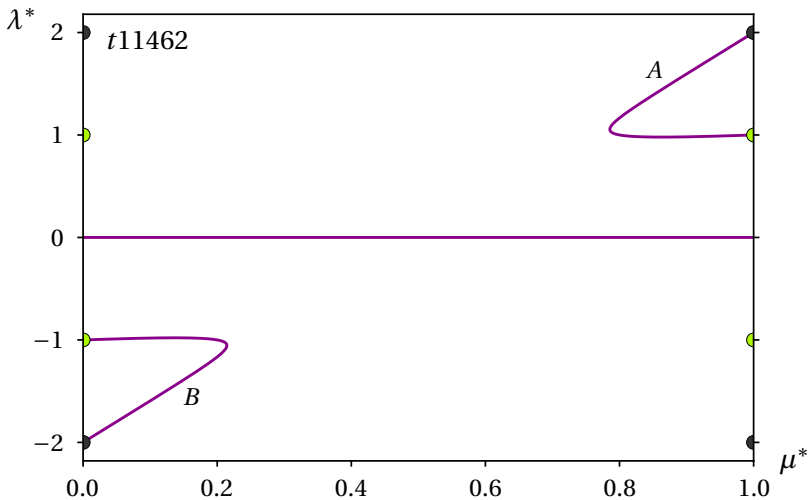

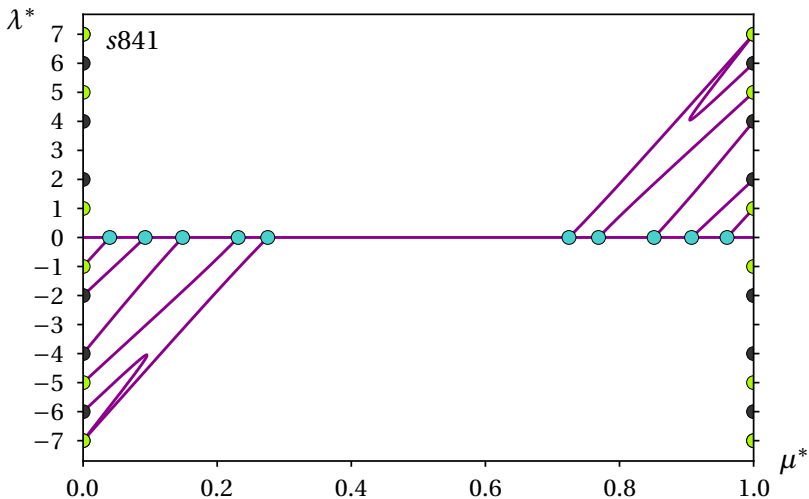
\begin{figure}
  \begin{center}
 \pgfkeys{/matplotlibfigure, default, width=11cm}%
 \input{plots/s841}

  \end{center}

  \caption{This locus has a mix of the behaviors shown in the previous
    figures.  The manifold $M = s841$ is the exterior of a genus 7
    fibered knot in $S^3$, namely $k6_{38}$ from
    \cite{ChampanerkarKofmanMullen2014}. (In SnapPy's default framing
    $\mu = (1,0)$ and $\lambda = (22, 1)$, and as usual
    $M(\mu) = S^3$.)  Using Lemma~\ref{lemma:key}, we can order $M(r)$
    for $r$ in $(-7, \infty)$; the interval of non-$L$-space slopes is
    $(-\infty, \infty)$ since the Alexander polynomial does not
    satisfy the condition of
    \cite[Corollary~1.3]{OSLensSpace2005}. There are actually two
    distinct Galois conjugates of the holonomy representation that
    give rise to each of the points $(0, \pm 7)$ and $(1, \pm 7)$. 
    This is why there are two separate arcs of $\TEL{M}$ emerging from
    these parabolic points instead of the one you might expect from
    the proof of Theorem~\ref{MainTheoremTwo}.  }
  \label{fig:s841}
\end{figure}

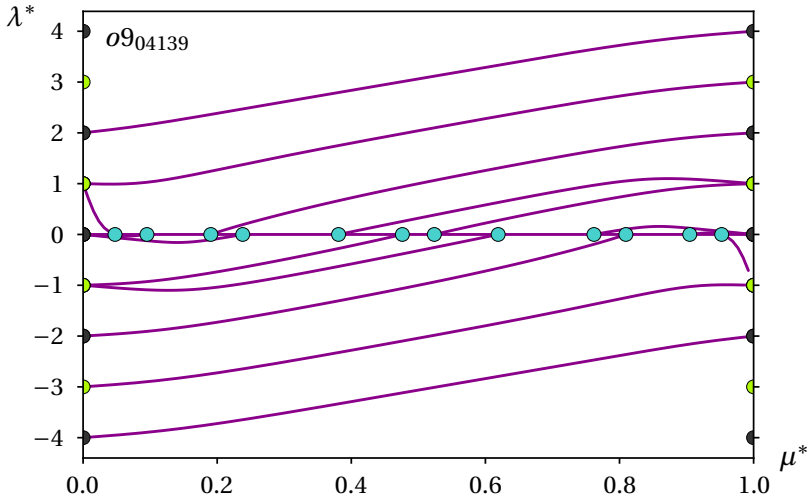
\begin{figure}
  \begin{center}
 \pgfkeys{/matplotlibfigure, default, width=11cm}%
 \input{plots/o9_04139}

  \end{center}

  \caption{This locus has arcs between parabolics that are on opposite
    sides of the strip.  It is one of the few instances we found where
    Lemma~\ref{lemma:key} allows us to order $M(r)$ for all $r$ in
    $(-\infty, \infty)$. (Another such example is $v1971$ in
    Figure~\ref{fig:mult}.)  The manifold $M = o9_{04139}$ is the
    exterior of a genus 6 fibered knot in $S^3$.  As usual
    $M(\mu) = S^3$; in SnapPy's default framing $\mu = (1,0)$ and
    $\lambda = (-1, 1)$.  While the Alexander polynomial does statisfy
    \cite[Corollary~1.3]{OSLensSpace2005}, it turns out that
    the set of non-$L$-space slopes is $(-\infty, \infty)$; using the
    criterion of \cite{Roberts2001} and the program flipper
    \cite{flipper}, Mark Bell and the second author were able to show
    that every nontrivial Dehn filling on $M$ has a co-orientable taut
    foliation.  There are actually four distinct Galois conjugates of
    the holonomy representation that give rise to each of the points
    $(0, \pm 1)$ and $(1, \pm 1)$, explaining the arcs that emerge
    from them. }
  \label{fig:o9_04139}
\end{figure}

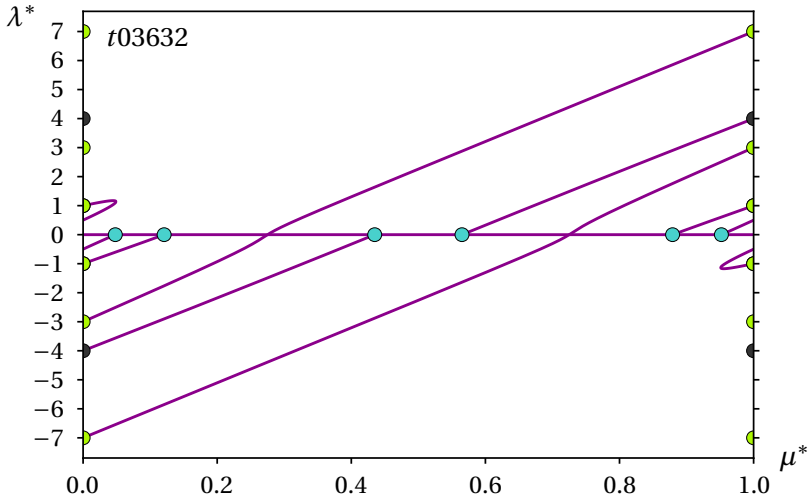
\begin{figure}
  \begin{center}
 \pgfkeys{/matplotlibfigure, default, width=11cm}%
 \input{plots/t03632}

  \end{center}

  \caption{The manifold $M = t03632$ is our first example that is not
    the exterior of a knot in $S^3$, with $M(\mu)$ being the small
    Seifert fibered space $S^2\left((2,1), (3, 1), (7, -6)\right)$
    which is a $\Z$-homology \3-sphere.  (In SnapPy's default framing
    $\mu = (1,0)$ and $\lambda = (-8, 1)$.)  One new phenomenon is that
    $\TEL{M}$ meets the sides of the strip at a point which is not an
    integer lattice point, namely the intersections at approximately
    $(0, \pm 1/2)$ and $(1, \pm 1/2)$.  Such nonintegral points of
    $\TEL{M}$ come from representations to $\Gtil$ which factor
    through $M(\mu)$, which is why they could not appear in the
    earlier examples where $M(\mu) = S^3$.  Another new phenomenon is
    that some arcs of $\TEL{M}$ cross the horizontal axis away from the
    Alexander points; the crossing points correspond to representations to $\Gtil$
    which factor through $M(\lambda)$ and have nonabelian image.  Such
    crossings also happen for certain exteriors of knots in $S^3$, for
    example with $o9_{21236}$, though not in any of the examples we
    show here.  }
  \label{fig:t03632}
\end{figure}

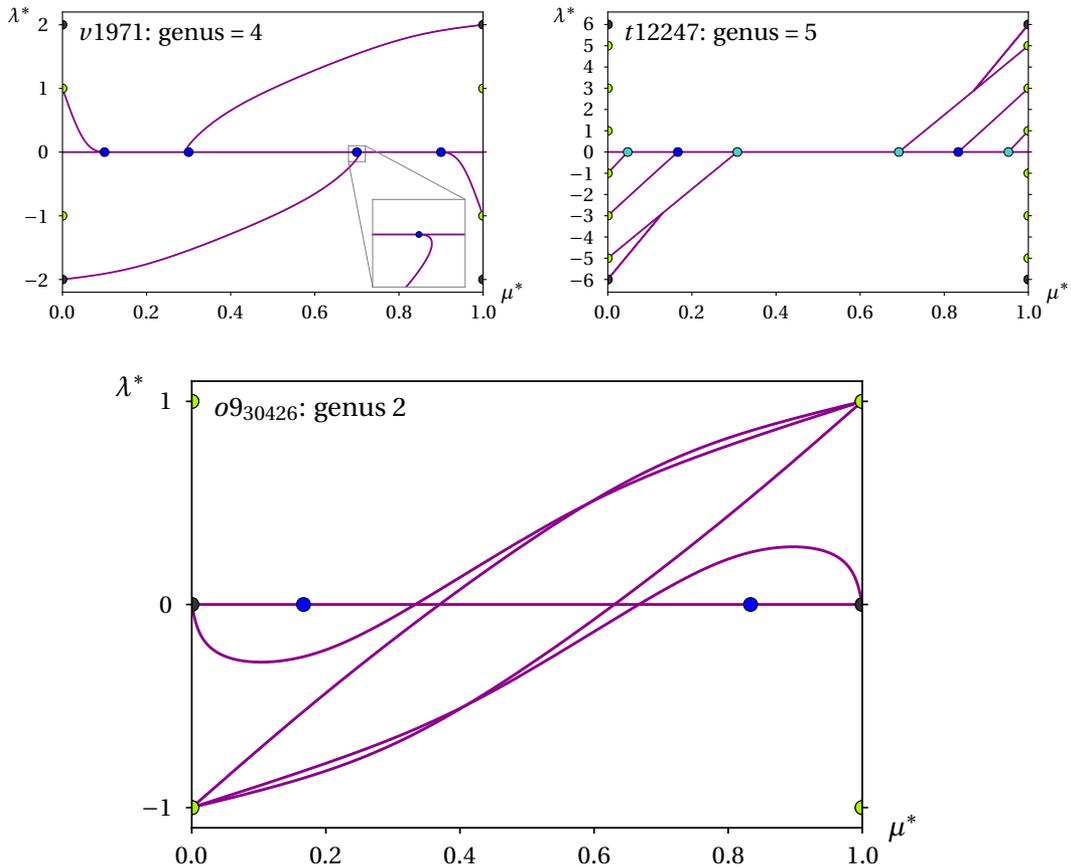
\begin{figure}
  \begin{center}

 \pgfkeys{/matplotlibfigure, default, width=6.9cm}%
 \input{plots/v1971}
 \hspace{-0.5cm}
 \pgfkeys{/matplotlibfigure, default, width=6.9cm}%
 \input{plots/t12247}

    %
 \pgfkeys{/matplotlibfigure, default, width=11cm}%
 \input{plots/o9_30426}

  \end{center}
  \caption{These three loci show some possible behaviors when the
    Alexander polynomial has a multiple root (for such a root, the
    corresponding Alexander point is dark blue rather than light
    blue).  The top left example is the knot exterior
    $v1971 = k7_{74}$ from \cite{ChampanerkarKofmanPatterson2004}.
    There, the arcs leaving the Alexander points are tangent to the
    horizontal axis, which is a common pattern for multiple roots.
    However, such tangencies are not required as the top right example
    of the knot exterior $t12247 = k8_{279} = 12n574$ from
    \cite{ChampanerkarKofmanMullen2014} shows.  The last example of
    $M = o9_{30426}$ is perhaps the most interesting: there are no
    nonhorizontal arcs of $\TEL{M}$ leaving the two Alexander points
    at all!  In fact, the corresponding reducible representations to
    $\GC$ are deformable to irreducible representations, but only into
    $\PSU_2$, not $\G$.  Here, $M(\mu)$ is the Seifert fibered space
    $S^2((2,1), (3,1), (11,-9))$, which is a $\Z$-homology \3-sphere, and
    there are three separate Galois conjugates of the holonomy
    representation at the points $(0, \pm 1)$ and $(1, \pm 1)$ in
    $\TEL{M}$.  The bottom example shows why we need the hypothesis
    that $\Delta_M$ has a simple root in the proof of
    Theorem~\ref{MainTheoremOne}, since the picture near the Alexander
    points does not match Figure~\ref{fig:arcstocone}.
    }
  \label{fig:mult}
\end{figure}

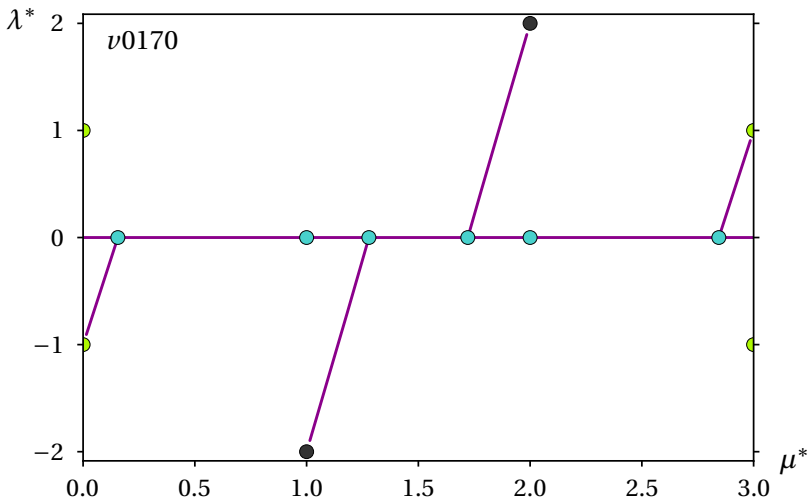
\begin{figure}
  \begin{center}
 \pgfkeys{/matplotlibfigure, default, width=11cm}%
 \input{plots/v0170}

  \end{center}

  \caption{The manifold $M = v0170$ is our first example of something
    that is not a $\Z$-homology solid torus.  In particular, the
    homological longitude $\lambda$ has order $k=3$ in $H_1(M; \Z)$, which
    is why the shown fundamental domain for the action of
    $T \leq \SymTEL{M}$ has width 3.  The filling $M(\mu)$ is the lens
    space $L(9, 2)$ with the core of the added solid torus
    representing three times a generator of
    $H_1(L(9, 2); \Z) \cong \Z/9\Z$.  (In SnapPy's default framing
    $\mu = (1,0)$ and $\lambda = (-5, 1)$.)  The manifold $M$ fibers
    over the circle with fiber a genus 4 surface with 3 boundary
    components.  For a root $\xi$ of $\Delta_M$, the corresponding
    Alexander point is plotted as $3 \arg(\xi)/ 2\pi$ to account for
    the fact that $\mu$ maps to three times a generator in
    $\Honefree$.  The two Alexander points at $(1,0)$ and $(2,0)$
    demonstrate the necessity of the hypothesis that $\xi^k \neq 1$
    for the proof of Theorem~\ref{MainTheoremOne}, since the local
    picture there does not match Figure~\ref{fig:arcstocone}.  The
    trace field of $M$ has 6 real embeddings, but above there is only
    one parabolic point modulo $\SymTEL{M}$; this is because most of
    the Galois conjugates into $\G$ do not lift to $\Gtil$.
    }
  \label{fig:v0170}
\end{figure}

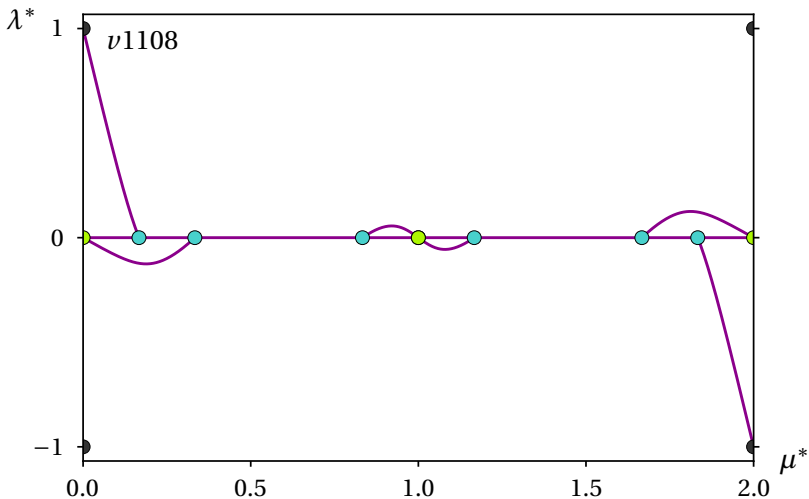
\begin{figure}
  \begin{center}
 \pgfkeys{/matplotlibfigure, default, width=11cm}%
 \input{plots/v1108}

  \end{center}

  \caption{The manifold $M = v1108$ is another example that is not a
    $\Z$-homology solid torus.  In particular, the homological
    longitude $\lambda$ has order $k=2$ in $H_1(M; \Z)$ and the
    filling $M(\mu)$ is the lens space $L(4, 1)$. (For once, the
    $(\mu, \lambda)$ framing is the same as SnapPy's default.)  The
    manifold $M$ fibers over the circle with fiber a genus 3 surface
    with 2 boundary components.  The parabolic points $(0, 0)$,
    $(1,0)$, and $(2, 0)$ are all double, that is, come from two
    distinct Galois conjugates of the holonomy representation.  In
    addition to being a parabolic point, the point (1,0) is also a
    simple Alexander point.  However, this Alexander point doesn't
    contribute an arc to $\TEL{M}$ because it corresponds to the root
    $\xi = -1$ and $\xi^k = 1$.  }
  \label{fig:v1108}
\end{figure}

% \begin{figure}
%   \begin{center}
%     \matplotlibfigure[width=11cm]{m389}
%   \end{center}

%   \caption{\todo{Not sure we really need this one, though it does give
%     an example of bending. Also, X-representations and degree 2 maps,
%     but those are not currently mentioned elsewhere in the paper.}
%   }
%   \label{fig:m389}
% \end{figure}

We now give 12 examples of translation extension loci which will motivate
the various results in this paper.  Indeed, for us these examples form
the intellectual core of this paper, directly inspiring all of the
theorems here.  The reader should peruse these examples carefully
before continuing, as they illustrate both the ideas and the potential
pitfalls in the proofs of the main theorems.  The first pictures of
this type appeared in Figure 8 of \cite{Khoi2003}.

The 12 examples come from hyperbolic \3-manifolds that have ideal
triangulations with at most 9 tetrahedra, and the nomenclature follows
\cite{CallahanHildebrandWeeks1999, Burton2014, SnapPy}.  We selected
them from a sample of about 600 translation extension loci of
such manifolds to illustrate a range of behaviors.

We start with $M = m016$, which is homeomorphic to the exterior of the
$(-2, 3, 7)$--pretzel knot in $S^3$.  Its translation extension locus
is shown in Figure~\ref{fig:m016}, and we discuss it in detail to explain
how to read the plots here.  We use a homological framing
$(\mu, \lambda)$ where $M(\mu) = S^3$ and $M(-18)$ and $M(-19)$
are lens spaces.  (In SnapPy's default framing, $\mu = (1, 0)$ and
$\lambda = (18, 1)$.)  The figure shows the intersection of $\TEL{M}$
with the strip $0 \leq x \leq 1$ in our usual
$(\mu^*, \lambda^*)$-coordinates on $H^1(\partial M; \R)$.  This strip
is a fundamental domain for the action of $T \leq \SymTEL{M}$ which is
generated by translation by $\mu^*$.  The symmetry of $\TEL{M}$ under
the element of $\SymTEL{M}$ which is $\pi$-rotation about $(1/2, 0)$
is visually clear.

There are 16 parabolic points of $\TEL{M}$ in this picture, which are
marked by the dark and light half disks on the vertical sides of the
strip.  (As mentioned, parabolic points are necessarily integer
lattice points.)  When the sides of the strip are glued by $T$, these
16 half disks are paired up to form 8 full discs; down in the full quotient
$\TELquo{M} = \TEL{M}/\SymTEL{M}$, there are only 4 parabolic points.

The color of the half disks indicates when the corresponding
representation to $G$ is Galois conjugate to the holonomy
representation of the complete hyperbolic structure on $M$ (see
Section~\ref{sec:galois} for the definition), with the light green
being ``geometric'' in this limited sense and black indicating other
``random'' parabolic $G$-representations.

There are no ideal points in this $\TEL{M}$ or in any of our example
translation loci; all of the manifolds involved are small, and
Lemma~\ref{lem:smallnoideal} below rules out any ideal points in this
situation. (The smallness of these manifolds was checked using
Regina~\cite{Regina}.)

The disks on the $\mu^*$-axis $L_\lambda$ correspond to the roots of
the Alexander polynomial that lie on the unit circle.  Specifically,
for each such root $\xi$, we plot $\left(\frac{\arg(\xi)}{2\pi}, 0\right)$
and call this an \emph{Alexander point}.  Simple roots, such as all
the ones for this manifold, are shown as light turquoise disks; in
later examples, multiple roots will be shown in dark blue.  Notice
that there is a nonhorizonal arc of $\TEL{M}$ leaving each Alexander
point. Such arcs are used to prove Theorem~\ref{MainTheoremOne}
and come from deforming an abelian representation to irreducible
representations, which is only possible at Alexander points (see
Section~\ref{sec:alexordering} for a complete discussion).

Since the line $L_r$ has slope $-r$ in our picture, and $M$ has no
reducible Dehn fillings, we see that Lemma~\ref{lemma:key} applies to
show $M(r)$ is orderable for all $r \in (-6, \infty)$.  To compare
with Conjecture~\ref{BGWconjecture}, the interval of non-$L$-space
fillings for $M$ is precisely $(-9, \infty)$ for the following
reason. As $M$ has two lens space fillings, it is \emph{Floer simple}
in the sense of \cite{RasmussenRasmussen2015}, and hence the interval
of $L$-space fillings is $[-\infty, -(2g -1)]$ where $g$ is the
Seifert genus; the latter is $5$ as that is the genus of the fiber in
the fibration of $M$ over the circle.  In fact, both
Theorems~\ref{MainTheoremOne} and \ref{MainTheoremTwo} apply to $M$,
though we got much better results by applying Lemma~\ref{lemma:key}
directly.

A summary of the overall structure of this $\TEL{M}$ is that,
besides the horizontal line of abelian representations, it consists of
diagonal arcs with a parabolic point at one end and an Alexander point
at the other. Moreover, none of the arcs overlap.  This pattern was
quite common in our sample, and a much more complicated instance is
shown in Figure~\ref{fig:v0220}.  Overall, there are many different
behaviors that are relevant to us here; please see Figures
\ref{fig:o9_34801}--\ref{fig:v1108} and their captions for details. 

\subsection{Numerical methods and caveats}

To compute points in $X_\R(M)$ corresponding to representations which
send $\mu$ to an elliptic isometry, we worked with the \emph{gluing
  variety} $\cG(\cT)$, where $\cT$ is an ideal triangulation of
$M$. Each $\cG(\cT)$ is an affine algebraic set described in
coordinates which are the shape parameters for the tetrahedra in
$\cT$. There is one equation for each edge, specifying that the
tetrahedra match around that edge, and the variety determined by these
has dimension $1$ in our examples.  The holonomy $H_\mu$ is the
square of an eigenvalue of the image of $\mu$, and can be expressed in
these coordinates to give a polynomial map $H_\mu:\cG(\cT)\to\C$.
We randomly chose a complex number $z_0$ near the unit circle and used
homotopy continuation with a start system given by the mixed volume
method to find the $0$-dimensional algebraic set $H_\mu^{-1}(z_0)$.
This computation was done with PHCpack \cite{Verschelde1999,
  PHCpack}. Once the fiber over $z_0$ had been computed, we used the
Newton-Raphson method to do path-lifting to our branched cover of $\C$
by $\cG(\cT)$.  With some care to avoid singularities, this allowed us
to compute the fiber over all $N^{\mathit{th}}$ roots of unity, where
$N$ was typically $128$ to start with, but sometimes needed to be
increased.  Each point of one of these fibers determined a character
in $X(M)$ corresponding to a representation sending $\mu$ to an
elliptic with rotation angle $2k\pi/N$ for some $k$.  These
representations were computed to standard floating point accuracy (53
bits) and it was numerically decided which of them gave points of
$\XG{M}$.
 
Once we had constructed a representation $\rho \maps \pi_1(M) \to G$,
we used the Newton\hyp Raphson method to polish it to very high
precision (typically 1{,}000 bits).  The Euler cocycle of
Section~\ref{sec:euler} was then computed and used to lift $\rho$ to
$\rhotil \maps \pi_1(M) \to \Gtil$. The peripheral translations of
$\rhotil$ were computed and then normalized under the action of
$\SymTEL{M}$ to be plotted in the figure.  (For the examples in
Figures \ref{fig:v0170} and \ref{fig:v1108}, frequently there was no
lift $\rhotil$ as $\euler{\rho}$ was nonzero in $H^2(M; \Z)$.)  For
each figure, we sampled as many as $2{,}000$ different holonomy values
for $\mu$ in order to get the smooth curves you see.

While we believe our plots of these loci are accurate, they were not
rigorously computed.  Moreover, there are reasons beyond numerical
accuracy that sometimes cause computations using gluing varieties to
produce incomplete results, with some arcs missing from the diagram.
(On the other hand, using gluing varieties rather than character
varieties hugely simplifies the computation, making it feasible to
handle larger examples.)  The key issue is that the natural map
$\cG(\cT) \to X(M)$ is not always onto; while each irreducible
component of $\cG(\cT)$ corresponds to some irreducible component of
$X(M)$, there can be components of $X(M)$ that are not seen in
$\cG(\cT)$ \cite[\S 10.3]{Dunfield2003}.  As $\cG(\cT)$ depends
fundamentally on the triangulation $\cT$, such ``missing components''
can sometimes be dealt with by changing the triangulation.  In other
cases, especially when there are components corresponding to
representations that factor through a proper quotient of $\pi_1(M)$,
changing the triangulation did not help.  (In \cite{Segerman2012}
Segerman constructs an ``extended'' version of $\cG(\cT)$ and shows
that there always exists a triangulation such that all components of
the character variety can be parametrized in terms of the associated
extended gluing variety. However, his technique has not been
implemented in software.)

In some cases we were able to detect missing components from
inconsistencies in our picture of $\TEL{M}$.  In the case of
$M = m389$, we obtained a plot of $\TEL{M}$ with a simple Alexander
point from which no arcs emerged, violating the proof of
Theorem~\ref{MainTheoremOne}.  It turns out that for the Dehn filling
$Y = m389(\mu + \lambda)$ there is a surjection from $\pi_1(Y)$ onto
$\PSL{2}{\Z} \cong C_2 * C_3$, giving a component of $X(M)$ that could
not be seen by our $\cG(\cT)$.  Another fairly common situation that
leads to missing components is when a Dehn filling $Y$ contains an
essential torus: if $X(Y)$ is nonempty, then $\dim_\C X(Y) \geq 1$
because it is possible to ``bend'' representations using the structure
of $\pi_1(Y)$ as a free product with amalgamation along the
$\Z^2$-subgroup corresponding to the essential torus.  It seems to be
common that components of $X(M)$ obtained by bending do not appear in
the image of $\cG(\cT)$.

Another issue with gluing varieties is that points at infinity of
$\cG(\cT)$ can correspond to non-ideal points of the character variety
$X(M)$.  Geometrically, this means that the shapes of some tetrahedra
degenerate even though the associated characters converge.  We
call these \emph{Tillmann points} after \cite{Tillmann2012}.  These
points cause numerical difficulties and complicate determining which
points of $\TEL{M}$ are ideal. Such Tillmann points occur reasonably
frequently in our examples.  Specifically, we used Goerner's
database \cite{GoernerDatabase} of boundary parabolic
representations to $\GC$ to identify which of the parabolic points
correspond to Galois conjugates of the holonomy representation of the
hyperbolic structure, and as a check to our own computations.  While
Goerner used Ptolemy equations rather than gluing equations, his
method still depends on a choice of triangulation, and parabolic
representations can go missing for the same reason.  In our examples,
there were five cases where our plot of $\TEL{M}$ indicated a
parabolic or ideal point on the vertical sides of the diagram that
were not present in \cite{GoernerDatabase}. For example, this occurred
with the point $(0, -2)$ in Figure~\ref{fig:t11462}.  Using
Lemma~\ref{lem:smallnoideal}, we were able to conclude that these are
all Tillmann points missed by our preferred triangulation, rather than
ideal points.

\section{Proof of the structure theorem}
\label{sec:structureproofs}

This section is devoted to the proofs of Theorem~\ref{thm:structure},
Lemma~\ref{lemma:key}, and Lemma~\ref{lemma:branched}.  An impatient
and trusting reader can skip ahead as the rest of the paper only
relies on the statements of these three results.  We begin attacking
Theorem~\ref{thm:structure} by proving the following two lemmas.

\begin{lemma}\label{lem:telsym}
  The extension locus $\TEL{M}$ is invariant under $\SymTEL{M}$.
\end{lemma}

\begin{lemma}\label{lem:finitecomp}
  The quotient space $\TELquo{M} = \TEL{M}/\SymTEL{M}$ has finitely
  many connected components.
\end{lemma}

\begin{proof}[Proof of Lemma~\ref{lem:telsym}]
Since invariance is preserved under taking closures, it suffices to
show that the image $I$ of $\RGtilPE{M}$ under $\trans \circ \inc^*$
is invariant under $\SymTEL{M}$.  Consider any
$\rhotil \in \RGtilPE{M}$ and let $t = \trans(\rhotil \circ \inc)$ be
the corresponding point in $I$.  If $\phi \in H^1(M; \Z)$, then, as
described in Section~\ref{sec:paramlifts}, one has
$\phi \cdot \rhotil$ in $\RGtilPE{M}$ which is also a lift of
$p \circ \rhotil$.  The image of $\phi \cdot \rhotil$ in $I$ differs
from $t$ via translation by $\inc^*(\phi)$ in $H^1(\partial M; \Z)$
since
\[
\trans\left(\phi \cdot \rhotil(\gamma)\right) = 
\trans\left(\rhotil(\gamma) s^{\phi(\gamma)} \right) =
\trans\left(\rhotil(\gamma)\right) + \phi(\gamma) 
\mtext{for all $\gamma \in \pi_1(M)$.}
\]
  In particular, this shows
that $I$ is invariant under translation by elements of
$T = \inc^*\left(H^1(M; \Z)\right) \subset H^1(\partial M; \R)$.

To complete the proof, it remains to show $I$ is invariant under
$x \mapsto -x$.  To this end, we will exhibit an automorphism
$\nu \maps \Gtil \to \Gtil$ where
$\trans\left(\nu(\gtil)\right) = -\trans(\gtil)$ for all
$\gtil \in \Gtil$.  Given such a $\nu$, the image of
$\nu \circ \rhotil$ in $I$ will be $-t$, proving invariance.  To
start, consider the element $r \in \Homeo(\R)$ which sends
$y \mapsto -y$.  Conjugation by $r$ preserves the subgroup $\Gtil$
because $r$ descends to the map of $\projsp^1(\R)$ induced by
$C = \mysmallmatrix{1}{0}{0}{-1} \in \PGL{2}{\R}$, and conjugation by
$C$ normalizes $\G \leq \PGL{2}{\R}$.  Let $\nu$ be conjugation of
$\Gtil$ by $r$.  Taking $x = 0$ in the definition (\ref{eq:transdef})
of translation number we get
\[
\trans\left(\nu(\gtil)\right) = \lim_{n \to \infty} \frac{
  \left(r \circ \gtil \circ r\right)^n(0)}{n} = \lim_{n \to \infty}
\frac{-\gtil^n(-0)}{n} = -\trans(\gtil)
\]
as required. 
\end{proof}

\begin{proof}[Proof of Lemma~\ref{lem:finitecomp}]
Consider the map $P \maps \RGtil{M} \to \RG{M}$ induced by
$p \maps \Gtil \to \G$.  Let $\RGPE{M}$ be the subset of $\RG{M}$
consisting of representations whose restrictions to
$\pi_1(\partial M)$ consist only of elliptic, parabolic, and trivial
elements.  Note that $\RGPE{M}$ is a real semialgebraic set.  Let
$\RGlift{M} \subset \RGPE{M}$ be the image of $\RGtilPE{M}$ under $P$.
By continuity of the Euler class (see Section~\ref{sec:euler}), the
subset $\RGlift{M}$ is a union of connected components of $\RGPE{M}$,
and hence also a real semialgebraic set. As described in
Section~\ref{sec:paramlifts}, the cohomology $H^1(M; \Z)$ acts freely
on $\RGtilPE{M}$ with quotient $\RGlift{M}$; consequently,
$P \maps \RGtilPE{M} \to \RGlift{M}$ is a (regular) covering map.
Because the action of $H^1(M; \Z)$ on $\RGtilPE{M}$ induces the action
of $T \leq \SymTEL{M}$ on $\TEL{M}$, the map $\trans \circ \inc^*$
below factors through $\psi$ as shown:
\[
\begin{tikzcd}
 \RGtilPE{M} \arrow{r}{\trans \circ \inc^*} \arrow{d}[left]{P} & \TELquo{M} \\ 
 \RGlift{M} \arrow[dashrightarrow]{ru}[below]{\psi} & 
\end{tikzcd}
\]
The map $\psi$ must be continuous as the vertical arrow $P$ is a
covering map.  As the set $\RGlift{M}$ has finitely many connected
components, it follows that
\[
\overline{\psi\left(\RGlift{M}\right)} = \TELquo{M}
\]
has finitely many components, proving the lemma.
\end{proof}

\subsection{Milnor-Wood bounds}  
\label{sec:milnor-wood} 

The remaining tool we need to prove Theorem~\ref{thm:structure} is:
\begin{lemma}\label{lem:compact}
  The space $\TELquo{M}$ is compact.
\end{lemma}
The proof of Lemma~\ref{lem:compact} hinges on knowing that $\TEL{M}$
is contained in a horizontal strip of bounded height; to show this, we
use the following result, which is closely related to the Milnor-Wood
inequality.

\begin{proposition}\label{prop:milnor-wood}
  Suppose $S$ is a compact orientable surface with one boundary
  component.  For all $\rhotil \maps \pi_1(S) \to \Gtil$ one has 
  \[
  \abstrans{\rhotil(\delta)}  \leq \max\left(-\chi(S), 0\right) \mtext{where
    $\delta$ is a generator of $\pi_1(\partial S)$.}
  \] 
\end{proposition}
Before discussing Proposition~\ref{prop:milnor-wood}, let us derive
Lemma~\ref{lem:compact} from it.  

\begin{proof}[Proof of Lemma~\ref{lem:compact}]
Recall that $M$ is a $\Q$-homology solid torus and let $k$ be the
order of the homological longitude $\lambda \in \pi_1(\partial M)$ in
$H_1(M;\Z)$.  There is a proper map of an oriented surface
$f \maps S \to M$ where $S$ has one boundary component and where
$f_*(\delta) = \lambda^k$ in $\pi_1(M)$ for $\delta$ a generator of
$\pi_1(\partial S)$.  Because $\trans$ is a homomorphism on cyclic
subgroups of $\Gtil$, we have
\[
k \cdot  \abstrans{\rhotil(\lambda)}  = \abstrans{\rhotil(\lambda^k)} 
   = \abstrans{(\rhotil \circ f_*)(\delta)}
\] 
Applying Proposition~\ref{prop:milnor-wood} to $\rhotil \circ f_*$
bounds the rightmost term in the previous equation, giving
\[
\abstrans{\rhotil(\lambda)} \leq \frac{\max\left(-\chi(S), 0\right)}{k}
\]
In particular, in our usual $(\mu^*, \lambda^*)$-coordinates on
$H^1(\partial M; \R)$, the locus $\TEL{M}$ lies in a horizontal strip
whose height is bounded by something that only depends on topological
information about $M$.  Thus, since $\SymTEL{M}$ contains horizontal
translations of $\R^2$ by shifts in $k \Z$, the quotient $\TELquo{M}$
is compact.
\end{proof}

We now discuss Proposition~\ref{prop:milnor-wood} in detail.  Recall
that a real-valued function $\phi$ on a group $\Gamma$ is called a
\emph{quasimorphism} if there exists a number $D$ such that
\[
\abs{\phi(xy) - \phi(x) - \phi(y)} \leq D  
  \mtext{for all $x,y\in\Gamma$,}
\]
and that the infimum of all such $D$ is
called the \emph{defect} of $\phi$.  The standard references \cite[\S
5]{Ghys2001} and \cite[\S 2.3.3]{Calegari2009} contain proofs that
for any representation $\rhotil:\Gamma\to\Gtil$, the function given by
$\phi = \trans\circ\rho$ is a quasimorphism.  It is also well-known
that this quasimorphism has defect at most $1$, although it is harder
to extract this fact from the literature.  It is stated in
\cite[Proposition~3.7]{Thurston1997}, with a sketch of a proof that
uses a construction of a connection on a circle bundle over a surface
in terms of a harmonic measure on a foliation transverse to the
fibers.  It is also a consequence of the ``ab Theorem'' of
\cite[Theorem~3.9]{CalegariWalker2011}, which was conjectured and
almost proved by Jankins and Neumann \cite{JankinsNeumann1985}, the
proof having been completed by Naimi \cite{Naimi1994}.  The proof of
Calegari and Walker is simpler and effective (see also \cite{Mann}).
With these facts in hand, we turn to the proof of the proposition.

\begin{proof}[Proof of Proposition~\ref{prop:milnor-wood}]
Let $g$ be the genus of $S$. The case of $g = 0$ is immediate as then
$\rhotil$ must be trivial and so $\trans(\rhotil(\delta)) = 0$; thus
we will assume $g > 0$.  Choose standard generators 
$\alpha_1, \beta_1, \ldots, \alpha_g, \beta_g$ for $\pi_1(S)$
where 
\[
\delta = [\alpha_1, \beta_1]\cdots[\alpha_g, \beta_g].
\]
Because $\trans:\Gtil\to\R$ is a quasimorphism of defect at most
$1$, we have 
\[
|\trans(xy)| \le |\trans(x)| + |\trans(y)| + 1 \mtext{for all $x, y\in\Gtil$.}
\]
It follows by induction that
\[
|\trans(x_1\cdots x_n)| \le |\trans(x_1)| + \cdots +|\trans(x_n)| +
(n-1) \mtext{for all $x_1,\ldots, x_n\in\Gtil$.} 
\]
As $\trans$ is constant on
conjugacy classes and satisfies $\trans(x^{-1}) = -\trans(x)$, we 
have
\[
\abstrans{[x,y]} = \left|\trans\left([x,y]\right) 
    - \trans\left(xyx^{-1}\right) - \trans\left(y^{-1}\right)\right|
   \le 1 \mtext{for all $x, y \in \Gtil$.}
\]
Combining these properties, we have
\[
\abstrans{\delta} = \abstrans{[\rhotil(\alpha_1), \, 
\rhotil(\beta_1)]\cdots[\rhotil(\alpha_g), \, \rhotil(\beta_g)]} \le g + (g - 1) = -\chi(S)
\]
as required.
\end{proof}

\begin{proof}[Proof of Theorem~\ref{thm:structure}] 

Define $c \maps H^1(\partial M; \R) \to \PSLRcharvar{\partial M}$ by
sending $\phi \maps \pi_1(\partial M) \to \R$ to the character of
the elliptic representation $\rho$ given by 
\[
\rho(\mu) = \pm\twobytwomatrix{e^{2\pi i \phi(\mu)}}{0}{0}{e^{-2\pi i
    \phi(\mu)}}
\mtext{and}
\rho(\lambda) = \pm\twobytwomatrix{e^{2 \pi
    i\phi(\lambda)}}{0}{0}{e^{- 2\pi i\phi(\lambda)}}
\]
We may use the dual basis to $(\mu, \lambda)$ and the trace-squared
coordinates on $\PSLRcharvar{\partial M}$ to express the map $c$ in
coordinates as:
\[
c(x, y) = 4\left(\cos^2( 2 \pi x), \ \cos^2( 2 \pi y), \ \cos^2\left(2\pi(x + y)\right)\right)
\]
For integers $m$ and $n$ we have $c(x+m, y+n) = c(x, y)$, and also
$c(\pm x, \pm y) = c(x, y)$.  Thus the map $c$ is topologically an
orbifold covering map from $\R^2$ onto a pillowcase, i.e.~a Euclidean
orbifold with underlying manifold $S^2$ and four cone points of angle
$\pi$.  Moreover, the following commutes:
\[
\begin{tikzcd}
  \RGtilPE{M} \arrow{r}{\trans\circ\inc^*} \arrow{d}
    & H^1(\partial M; \R) \arrow{d}{c} \\
  \PSLRcharvar{M} \arrow{r}{\inc^*} & \PSLRcharvar{\partial M}
\end{tikzcd}
\]
Note that $c$ maps $\TELquo{M}$ into
$\overline{\inc^*\left(\XG{M}\right)}$.  Now by
Lemma~\ref{lem:1d-image}, the complex algebraic set
$\inc^*\left(X(M)\right) \subset X(\partial M)$ has complex dimension
at most 1; hence the real semialgebraic set
$\overline{\inc^*\left(\XG{M}\right)}$ has real dimension at most 1.
Moreover, the set $\overline{\inc^*\left(\XG{M}\right)}$ is compact
since the subset of $X(\partial M)$ corresponding to representations
that are parabolic, elliptic, or trivial is compact.  Hence by
Proposition~\ref{prop:real1}, the set
$\overline{\inc^*\left(\XG{M}\right)}$ is a finite graph.  Thus, its
preimage under $c$ is a locally finite graph with analytic edges that
is invariant under $\SymTEL{M}$ by Lemma~\ref{lem:telsym}.  As
$\TELquo{M}$ is compact by Lemma~\ref{lem:compact}, we can conclude
that it lives in some finite graph in $H^1(\partial M; \R)/\SymTEL{M}$
with analytic edges.  Now, since $\TELquo{M}$ has finitely many
connected components by Lemma~\ref{lem:finitecomp}, it follows that it
too must be a finite graph in $H^1(\partial M; \R)/\SymTEL{M}$ with
analytic edges.  This proves the hardest part of the theorem.

To see that there are only finitely many parabolic points, note that
these only occur at images of lattice points in $H^1(\partial M; \Z)$,
and there can only be finitely many such points in the compact set
$\TELquo{M}$.  Also, the space $\TELquo{M}$ is the closure in a finite
graph of a set with finitely many components, and thus there are only
finitely many ideal points.  Finally, consider the copy of $\R$ in
$\Gtil$ sitting above $\PSO_2 \leq G$.  As $\Honefree = \Honefreedef \cong \Z$, we
get a \1-parameter family of abelian representations
$\pi_1(M) \to \Gtil$ by sending the generator of $\Honefree$
to any chosen element of $\R$.  Since $\lambda$ is zero in
$\Honefree$ whereas $\mu$ is nonzero, we see that these abelian
representations give rise to the line $L_\lambda$ inside of $\TEL{M}$,
finishing the proof of the structure theorem. 
\end{proof}

\subsection{Constructing orderings}  We now turn to the proofs of
the lemmas that we use to construct orderings of \3-manifold groups.

\begin{proof}[Proof of Lemma~\ref{lemma:key}]
Let $\phi$ be such a point in $L_r \cap \TEL{M}$.  As it is neither
parabolic nor ideal, there is a $\rhotil \in \RGtil{M}$ which maps to
$\phi$ where the restriction of $\rhotil$ to $\pi_1(\partial M)$ is
either elliptic or central. Let $\gamma$ be an element of
$\pi_1(\partial M)$ realizing the slope $r$.  By the definition of
$L_r$, we have $\phi(\gamma) = (\trans \circ \rhotil)(\gamma) = 0$.
It follows from Lemma~\ref{lem:trans_on_elliptic} that $\gamma$ is in
the kernel of $\rhotil$, and hence we get an induced representation
$\rhobar \maps \pi_1\left( M(r) \right) \to \Gtil$.  As $\phi$ is not
the origin in $H^1(M;\R)$, the new representation $\rhobar$ is
nontrivial since some element of $\pi_1(\partial M)$ is mapped to an
element of $\Gtil$ with nonzero translation number. Thus we have found
a nontrivial homomorphism $\pi_1\left(M(r)\right) \to \Gtil$.
Regarding $\Gtil$ as subgroup of $\Homeo^+(\R)$ and using that $M(r)$
is irreducible, Theorem~1.1 of \cite{BoyerRolfsenWiest2005} applies to
promote this nontrivial homomorphism
$\pi_1\left( M(r)\right) \to \Homeo^+(\R)$ to a faithful one;
equivalently, the group $\pi_1\left( M(r) \right)$ is left-orderable
as claimed.
\end{proof}

\begin{figure}
  \begin{center}
    \input plots/branched
  \end{center}
  \caption{This picture illustrates the proof of
    Lemma~\ref{lemma:branched} in a case where the covering map
    $\pi \maps \Mtil \to M$ has degree 3.  At left is $\TEL{M}$, where
    its intersections with the vertical axis are parabolic points and
    there are no ideal points.  Note that $\TEL{M}$ meets the
    vertical line $\mu^* = 1/3$ in two points. Its image under
    $\pi^* \maps H^1(\partial M; \R) \to H^1(\partial \Mtil; \R)$ is
    shown at right as the darker curves; the image is just a copy of
    $\TEL{M}$ stretched horizontally by a factor of 3.  In addition,
    $\TEL{\Mtil}$ contains the lighter curves shown, which are other
    translates of $\pi^*\left(\TEL{M}\right)$ under $\SymTEL{\Mtil}$.
    It is the lighter curves that contribute non-parabolic
    intersections of $\TEL{\Mtil}$ with the vertical axis $L_\mutil$,
    corresponding to the original intersections of $\TEL{M}$ with
    $\mu^* = 1/3$, and so allow us to order $\Ytil = \Mtil(\mutil)$
    via Lemma~\ref{lemma:key}.  }
  \label{fig:branched}
\end{figure}

\begin{proof}[Proof of Lemma~\ref{lemma:branched}]
Let $\pi \maps \Mtil \to M$ be the covering map corresponding to
$\Ytil \to Y$.  Restricting representations from $\pi_1(M)$ to
$\pi_1(\Mtil)$, we get a natural subset of $\TEL{\Mtil}$ from
$\TEL{M}$.  Specifically, the locus $\TEL{\Mtil}$ contains the image
of $\TEL{M}$ under
$\pi^* \maps H^1(\partial M; \R) \to H^1(\partial \Mtil; \R)$.  We use
$(\mutil, \lambdatil)$ as a basis for $H_1(\partial \Mtil; \Z)$, where
$\mutil$ maps to $n \mu$ in $H_1(\partial M;\Z)$ and $\lambdatil$ maps
to $\lambda$.  In the dual bases, we thus have that
$\pi^* \maps H^1(\partial M; \R) \to H^1(\partial \Mtil; \R)$ is given
by $\mu \mapsto n \mutil$ and $\lambda \mapsto \lambdatil$.  Hence
$\pi^*(\TEL{M})$ is basically $\TEL{M}$ stretched horizontally by
a factor of $n$.  If we act on $\pi^*(\TEL{M})$ by $\SymTEL{\Mtil}$,
we get additional copies of $\pi^*(\TEL{M})$ as shown in
Figure~\ref{fig:branched}.  (These additional translates still come
from representations $\pi_1(M) \to G$, but correspond to lifts
$\pi_1(\Mtil) \to \Gtil$ that do not extend to all of $\pi_1(M)$; the
point is that we can adjust a lift by any element in $H^1(\Mtil; \Z)$
and the image of $H^1(M; \Z)$ has index $n$.)  The key observation is
that as $\TEL{M}$ meets the line $\mu^* = 1/n$, the locus
$\TEL{\Mtil}$ meets the line $\mutil^* = 1$, and hence by the action of
$\SymTEL{\Mtil}$ the locus $\TEL{\Mtil}$ meets $L_\mu$ at a point
$t = (0, y)$.  The desired conclusion now follows from
Lemma~\ref{lemma:key} provided we can show that $t$ is neither ideal
nor parabolic.  The former is ruled out by the hypothesis that the
initial intersection of $\TEL{M}$ with $\mu^* = 1/n$ was not an ideal
point.  The latter is impossible since, when restricting a
representation $\Z^2 \to \Gtil$ to a finite index subgroup, the only
possible change of type (as defined in Section~\ref{sec:repsZZ}) is
from elliptic to trivial, and the initial intersection of $\TEL{M}$
with $\mu^* = 1/n$ is not parabolic as it is not in the lattice
$H^1(M; \Z)$.  Thus we can apply Lemma~\ref{lemma:key} to order
$\Ytil$ as required.
\end{proof}

\subsection{Ideal points}  

The following result was used in Section~\ref{sec:menagerie}, but is not
central to this paper and the proof can be safely skipped. 

\begin{lemma}
  \label{lem:smallnoideal}
  Suppose $M$ is a $\Q$-homology solid torus which is small, that is,
  contains no closed essential surfaces.  Then $\TEL{M}$ has no ideal
  points. 
\end{lemma}

\begin{proof}
Suppose $t_0$ is an ideal point of $\TEL{M}$.  Pick a sequence
$\rhotil_i \in \RGtilPE{M}$ whose images in $\TEL{M}$ converge to $t$.
Consider the representations $\rho_i = p \circ \rhotil_i$ in $\RG{M}$
and the corresponding characters $[\rho_i]$ in $\XG{M}$. Passing to a
subsequence, we arrange that the $[\rho_i]$ lie in a single
irreducible component $X'$ of $X(M)$.  As $M$ is small, the variety
$X'$ must be a complex affine curve by \cite[\S 2.4]{CCGLS}.  As
$\XG{M}$ is closed in $X(M)$ by Lemma~\ref{lem:closed}, we have that
$X'_G = X' \cap \XG{M}$ is closed in $X'$.  Passing to a subsequence,
either the $[\rho_i]$ limit to a character in $\XG{M}$ or the
$[\rho_i]$ march off to infinity in the noncompact curve $X'$.  In the
latter case, since we have
$\big\{ \tr^2_\gamma \rho_i \big\} \in [0, 4]$ for all
$\gamma \in \pi_1(\partial M)$, the argument of \cite[\S 2.4]{CCGLS}
produces a closed essential surface associated to a certain ideal
point of $X'$, contradicting our hypothesis that $M$ is small.

Now consider the case when the $[\rho_i]$ limit to $\chi$ in $\XG{M}$.
By Proposition~\ref{prop:real1}, we pass to a subsequence where there
is an arc $\cbar$ in $\XG{M}$ starting at $[\rho_0]$, ending at $\chi$
and containing all the $[\rho_i]$.  Using Lemma~\ref{lem:pathlift},
lift $\cbar$ to a path $c$ in $\RG{M}$ starting at $\rho_0$ and ending
at some $\rho$ whose character is $\chi$.  In the notation of the
proof of Lemma~\ref{lem:finitecomp}, we have that the $\rho_i$ are in
$\RGPE{M}$.  Note that $\rho$ is also in $\RGPE{M}$ as it is in
$\RG{M}$ and $\tr^2_\gamma \rho$ must be in $[0, 4]$ by continuity for
all $\gamma \in \pi_1(\partial M)$.  As in the proof of
Lemma~\ref{lem:finitecomp}, we have that $c$ is in $\RGlift{M}$ and so
we can lift $c$ to a path $\ctil$ in $\RGtilPE{M}$ starting at
$\rhotil_0$.  After possibly changing $\ctil$ by a deck transformation
of $\RGtilPE{M} \to \RGlift{M}$, we can assume that the image of
$\ctil(1)$ in $\TEL{M}$ is exactly $t_0$.  Thus $t_0$ is not actually
an ideal point, proving the lemma.
\end{proof}

\section{Alexander polynomials and orderability}
\label{sec:alexordering}

In this section we prove our first main result,
Theorem~\ref{MainTheoremOneOne}, which implies
Theorem~\ref{MainTheoremOne} from the introduction.  To state the more
general result, we need a pair of definitions.  First, we say a
compact \3-manifold $Y$ has \emph{few characters} if each positive
dimensional component of $X(Y)$ consists entirely of characters of
reducible representations.  An irreducible $\Q$-homology solid torus
$M$ is called \emph{longitudinally rigid} when its Dehn filling along
the homological longitude $M(0)$ has few characters.  Here is the
statement of Theorem~\ref{MainTheoremOneOne}, where the manifold $M$
has a fixed homologically natural framing $(\mu, \lambda)$.

\begin{theorem}
  \label{MainTheoremOneOne}
  Suppose that $M$ is a longitudinally rigid irreducible $\Q$-homology
  solid torus and that the Alexander polynomial of $M$ has a simple root
  $\xi$ on the unit circle.  When $M$ is not a $\Z$-homology solid torus,
  further suppose that $\xi^k \neq 1$ where $k > 0$ is the order of
  the homological longitude $\lambda$ in $H_1(M; \Z)$.  Then there
  exists $a>0$ such that for every rational $r\in (-a,0) \cup (0, a)$
  the Dehn filling $M(r)$ is orderable.
\end{theorem}
Steven Boyer told us in a private communication that there is an
analog of Theorem~\ref{MainTheoremOne} when the simple root $\xi$ is
on the positive real axis.  Here is the argument that this
implies Theorem~\ref{MainTheoremOne}.

\begin{proof}[Proof of Theorem~\ref{MainTheoremOne}]
Comparing the statements, there are two things to do: show that $M$
being lean implies that $M$ is longitudinally rigid, and establish
that $M(0)$ is orderable.  The latter is immediate from Theorem 1.1 of
\cite{BoyerRolfsenWiest2005} since $H^1(M(0); \Z) \cong \Z$ and $M(0)$
is either irreducible or $S^2 \times S^1$.  The former is an immediate
consequence of
\begin{claim} \label{claim:fibertofew}
  Suppose $Y$ is an irreducible closed \3-manifold.  If
  the only essential surfaces in $Y$ are fibers in fibrations over the
  circle, then $Y$ has few characters.
\end{claim}
Here is the proof of the claim. Suppose instead that $X(Y)$ has a
positive dimensional component $Z$ containing an irreducible
character $\chi_0$.  Recall from Section~\ref{sec:repcharvar}
that the functions $\tr_\alpha^2$ for a finite set of
$\alpha\in\pi_1(Y)$ give coordinates on the complex affine algebraic
set $X(Y)$.  Pick an irreducible curve $X_0 \subset Z$ that contains
$\chi_0$, which we can do by e.g.~Corollary 1.9 of
\cite{CharlesPoonen2016}. As affine algebraic curves over $\C$ are
noncompact, there is at least one ideal point of $X_0$ in the sense of
\cite[\S 4]{BoyerZhang1998}.  This gives an action of $\pi_1(Y)$ on a
simplicial tree, which in turn has an essential dual surface.  Let $F$
be a connected component of this dual surface.  By hypothesis, the
surface $F$ must be a fiber in a fibration of $Y$ over the
circle. In fact, since \emph{every} essential surface in $Y$ is a
fiber, it follows from \cite[Pages~113--115]{Thurston1986} that
$b_1(Y) = 1$ and that $F$ is the unique connected essential surface
in $Y$ up to isotopy.  Therefore, the surface associated to any other
ideal point of $X_0$ must also consist of parallel copies of $F$.
Hence for every $\gamma \in \pi_1(F)$, the function $\tr^2_\gamma$
takes a finite value at every ideal point of the curve $X_0$, which
forces the function $\tr^2_\gamma$ to actually be constant on $X_0$.
Thus every character in $X_0$ has the same restriction to
$\pi_1(F)$, which we denote by $\eta \in X(F)$.  There are two cases
depending on whether or not $\eta$ is reducible. 

Suppose $\eta$ is irreducible.  As per Section~\ref{sec:repcharvar},
all representations $\pi_1(F) \to \GC$ with character $\eta$ are
irreducible and conjugate, and let us fix one such representation
$\rho$.  If $f \maps F \to F$ is the monodromy of the fibration, we
have the usual presentation
\[
  \pi_1(M) = \spandef{\tau, \pi_1(F)}{
    \mbox{$ \tau \gamma \tau^{-1} = f_*(\gamma)$ for all
      $\gamma \in \pi_1(F)$}}
\]
Thus, a representation $\rhohat:\pi_1(M)\to\GC$ that restricts to $\rho$ on
$\pi_1(F)$ is determined by the element $T = \rhohat(\tau) \in \GC$; moreover,
$T$ must conjugate $\rho$ to $\rho \circ f_*$.  As $\rho$
is irreducible, its stabilizer under conjugation is finite
\cite[Proposition~3.16(i)]{HeusenerPorti2004}, and hence there are
only finitely many possibilities for $T$.  But then $X_0$ is finite, a
contradiction.

Suppose instead that $\eta$ is reducible. Let
$\psi \maps \pi_1(M) \to \GC$ be an irreducible representation with
character in $X_0$.  Note that $\psiF$ is nontrivial as otherwise
$\psi$ factors through $\pi_1(M)/\pi_1(F) \cong \Z$ making $\psi$
itself reducible.  As $\psiF$ has character $\eta$, it is reducible
and has either exactly one or exactly two fixed points on $\Pone$.  If
$\psiF$ had a unique fixed point $p_0 \in \Pone$, then, since
$\pi_1(F)$ is normal, it follows that $\psi$ itself fixes $p_0$,
making $\psi$ reducible.  So $\psiF$ has exactly two fixed points on
$\Pone$, and we conjugate $\psi$ so that these are $0 = [0:1]$ and
$\infty = [1:0]$.  After this conjugation, the image of $\psiF$
consists of diagonal matrices and its non-trivial elements are
hyperbolic or elliptic with axis the geodesic $L$ in $\H^3$ that joins
$0$ to $\infty$.  Now consider how $\psi(\tau)$ acts on the points $0$
and $\infty$ . It must not fix them individually, as then $\psi$ would
be reducible. Hence $\psi(\tau)$ is an elliptic element of order two
whose axis is orthogonal to $L$.  We can conjugate $\psi$ by a
diagonal matrix, which does not change $\psiF$, so that
$\psi(\tau) = \pm \mysmallmatrix{0}{1}{-1}{0}$.  In particular, up to
conjugacy, $\psi$ is completely determined by $\psiF$.  As a diagonal
representation such as $\psiF$ is determined up to conjugacy by its
character, we have shown that $X_0$ contains a unique irreducible
character.  But this contradicts the fact that the irreducible
characters in $X_0$ are Zariski open
\cite[Corollary~3.6]{HeusenerPorti2004}.  This completes the proof of
Claim~\ref{claim:fibertofew} and shows that
Theorem~\ref{MainTheoremOne} follows from
Theorem~\ref{MainTheoremOneOne}.
\end{proof}

We now sketch the proof of Theorem~\ref{MainTheoremOneOne}, which we
also illustrate in Figure~\ref{fig:arcstocone}.  Recall that to order
the Dehn filling $M(r)$ by applying Lemma~\ref{lemma:key}, we need an
intersection of the translation locus $\TEL{M}$ with the line $L_r$,
which is the line through the origin of slope $-r$.  So to prove the
theorem, we construct a cone $\cC$ of lines through the origin that
contains the horizontal axis $L_0$ and where every line in $\cC$ meets
$\TEL{M}$.  To this end, we use a result of Heusener and Porti
\cite{HeusenerPorti2005} to build an arc $A$ in $\TEL{M}$ which
starts at a point in $L_0$ but is otherwise disjoint from it.  The
symmetries of $\TEL{M}$ guarantee that if we have such an arc on one
side $L_0$ then we will have one on the other side as well, giving us
a big enough chunk of $\TEL{M}$ to have the desired cone $\cC$; see
Figure~\ref{fig:arcstocone} for more.  A key technical point is that
we must take care to ensure that the arc $A$ is not completely
contained in $L_0$, and this is where the hypothesis of longitudinally
rigid comes in.

A key component of the proof is the following result derived from
\cite{HeusenerPorti2005}.  

\begin{lemma}%[\cite{HeusenerPorti2005}]
  \label{lem:abeliandeforms}
  Suppose $M$ is an irreducible $\Q$-homology solid torus. If $\xi$ is a
  simple root of the Alexander polynomial that lies on the unit
  circle, then there exists an analytic path
  $\rho_t \maps [0, 1] \to \RG{M}$ where:
  \begin{enumerate}
  \item \label{item:character}
    The representation $\rho_0$ acts by rotations about a unique
    fixed point in $\H^2$, and factors through
    $H = \Honefree = \Honefreedef \cong \Z$.  A
    generator of $H$ acts via rotation by angle $\arg(\xi)$.

  \item \label{item:irreducible}
    The representations $\rho_t$ are irreducible over $\GC$ for
    $t > 0$.

  \item \label{item:charvaries}
    The corresponding path $[\rho_t]$ of characters in $\XG{M}$ is also
    a nonconstant analytic path.  
    
  \item \label{item:tracevaries} There exists
    $\gamma \in \pi_1(\partial M)$ where
    $\tr^2_\gamma\left(\rho_t\right)$ is nonconstant in $t$.
  \end{enumerate}
\end{lemma}

\begin{figure}
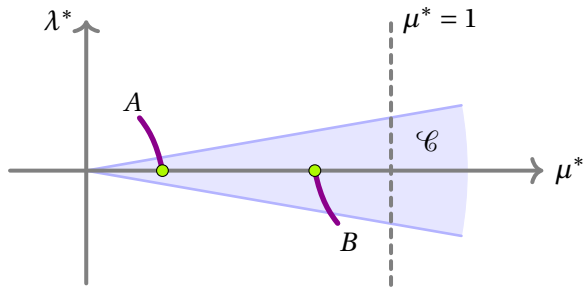

  \begin{center}
    \input plots/arcs_to_cone
  \end{center}

  \vspace{-0.2cm}

  \caption{Here is an outline of the proof of
    Theorem~\ref{MainTheoremOneOne}.  From the simple root $\xi$ of
    $\Delta_M$, we use \cite{HeusenerPorti2005} to produce an arc
    $A$ in $\TEL{M}$ leaving the horizontal axis at a
    corresponding Alexander point. Using the action of $\SymTEL{M}$,
    we can assume the arc $A$ lies in the strip
    $0 \leq x \leq 1$ as shown.  The element of $\SymTEL{M}$ which
    is $\pi$-rotation about $(1/2, 0)$ means there will be a second
    arc $B$ in this strip on the opposite side of the
    horizontal axis from $A$.  This allows us to find a cone $\cC$
    whose lines through the origin meet $\TEL{M}$ in a
    point which is neither parabolic or ideal.  The theorem will then
    follow from Lemma~\ref{lemma:key}.} \label{fig:arcstocone}
\end{figure}

\begin{proof}
Except for part (\ref{item:tracevaries}), the lemma follows
straightforwardly from the statement of Proposition 10.3 of
\cite{HeusenerPorti2005} and Lemma~\ref{lem:pathlift} of our paper.
However, it is even easier to derive claims
(\ref{item:character}--\ref{item:charvaries}) directly from the
discussion in Section 10 of \cite{HeusenerPorti2005} and we take that
approach.  Throughout, we will follow the notation of
\cite{HeusenerPorti2005} closely.  Fix a generator $h$ of $H$, and let
$\alpha \maps \pi_1(M) \to \C^\times$ be the homomorphism 
which factors through the homomorphism $H \to \C^\times$ that sends $h$ to $\xi$.
Consider the associated diagonal
representation $\rho_\alpha \maps \pi_1(M) \to \GC$ given by
\[
\rho_\alpha(\gamma) = 
    \pm
    \twobytwomatrix{\alpha^{1/2}(\gamma)}{0}{0}{\left(\alpha^{1/2}(\gamma)\right)^{-1}}
    \mtext{where $\alpha^{1/2}(\gamma)$ is either square root of $\alpha(\gamma)$.}
\]
Now the image of $\rho_\alpha$ is contained in the following subgroup
of $\GC$
\[
\PSU(1,1) = \setdef{\twobytwomatrix{a}{b}{\bbar}{\abar}}{\mbox{$a, b \in \C$
  with $\abs{a}^2 - \abs{b}^2 = 1$}}
\]
which is a conjugate of $\G$ in $\GC$ as it corresponds to the
M\"obius transformations that stabilize the unit disc $D$ in
$\C \subset \projsp^1(\C)$.  

The proof of Proposition~10.3 in \cite{HeusenerPorti2005} shows that
the cocycle defined there
\[
d_+ + d_- \in H^1\left(\pi_1(M); \  \su(1,1)_{\rho_\alpha}\right)
\]
can be integrated to an analytic path
$\rho_t \maps [0, 1] \to R_{\PSU(1,1)}(M)$ with $\rho_0 = \rho_\alpha$
and $\rho_t$ irreducible over $\GC$ for $t > 0$, which gives
(\ref{item:irreducible}).  Note that $\rho_\alpha$ stabilizes the
center of $D$ and acts on the tangent space there via $\alpha$, which
gives (\ref{item:character}).  Next, claim (\ref{item:charvaries})
that $[\rho_t]$ is nonconstant follows from (\ref{item:irreducible}),
since, over $\GC$, a reducible representation cannot have the same
character as an irreducible representation.

Finally, we tackle claim (\ref{item:tracevaries}), whose proof is more
involved; please note that claim (\ref{item:tracevaries}) is not
actually used in this paper and so you can safely skip it.  By
Theorem~1.3 of \cite{HeusenerPorti2005}, the character
$\chi_\alpha = [\rho_\alpha]$ is contained in precisely two
irreducible components of $X(M)$, both of which are (complex) curves:
one consisting solely of characters of abelian representations and the
other, which we will call $X$, whose characters generically come from
representations that are irreducible over $\GC$.  Of course, our path
$[\rho_t]$ lies in $X$.  To study $X$ near $\chi_\alpha$, we move away
from $\rho_\alpha$ to the representation $\rho^+ \in R(M)$ constructed
in \cite[\S 5]{HeusenerPorti2005}.  The representation $\rho^+$ is
also reducible with character $\chi_\alpha$ but has nonabelian image.
Proposition 7.6 of \cite{HeusenerPorti2005} gives that $\rho^+$ is a
smooth point of $R(M)$ of local dimension 4.  Let $\slrhoplus$ denote
the Lie algebra of $\GC$ as a $\pi_1(M)$-module via the action
$\Ad \circ \rho^+$.  The proof of Proposition 7.6 of
\cite{HeusenerPorti2005} shows that the Zariski tangent space of
$R(M)$ at $\rho^+$ can be identified with the space of cocycles
$\twisted{Z}{1}{M}$. (Unlike \cite{HeusenerPorti2005}, we are assuming
that $M$ is irreducible and consequently aspherical, and so do not
distinguish between cohomology of $M$ and of $\pi_1(M)$.)  As the
tangent space to the orbit of $\rho^+$ is the space of coboundaries
$\twisted{B}{1}{M}$, we can identify the Zariski tangent space of $X$
at $\chi_\alpha$ with $\twisted{H}{1}{M}$, which is $\C$ by Corollary
5.4 of \cite{HeusenerPorti2005}.  As the restriction
$\rho^+ \circ \inc$ in $R(\partial M)$ is nontrivial, the proof of
Lemma~7.4 of \cite{HeusenerPorti2005} gives that $\rho^+ \circ \inc$
is a smooth point of $R(\partial M)$, and so again we can identify
the Zariski tangent space of $X(\partial M)$ at $[\rho^+ \circ \inc]$
with $\twisted{H}{1}{\partial M} \cong \C^2$.  The claim
(\ref{item:tracevaries}) boils down to showing that
\begin{equation}\label{eq:cohomap}
\inc^* \maps  \twisted{H}{1}{M} \to \twisted{H}{1}{\partial M} 
\end{equation}
is injective, since coordinates on $X(\partial M)$ are precisely the
functions $\tr^2_\gamma$ for $\gamma \in \pi_1(M)$.

To understand the map in (\ref{eq:cohomap}), start by calculating that
the $0$\hyp cohomologies, or equivalently the $\pi_1(M)$-invariant
subspaces of $\slrhoplus$, are
\[
\twisted{H}{0}{\partial M} \cong \C \mtext{and} \twisted{H}{0}{M}
\cong 0
\]  
As $M$ has Euler characteristic $0$, this forces $\twisted{H}{2}{M}
\cong \C$, and so by duality we have $\twisted{H}{1}{M, \partial M}
\cong \C$ as well.  Suppressing the coefficients, the long exact
sequence of the pair includes
\[
\begin{tikzcd}
  H^1(\partial M) & \arrow{l}[above]{\inc^*} H^1(M)
  & \arrow{l} H^1(M, \partial M) & 
  \arrow{l}[above]{\delta}  H^0(\partial M) & \arrow{l} H^0(M) \\[-4.5ex]
  \C^2 & \C & \C & \C & 0
\end{tikzcd}
\]
which forces $\inc^*$ at left to be injective, as claimed.  This
establishes (\ref{item:tracevaries}) and hence the lemma. 
\end{proof}

\begin{proof}[Proof of Theorem~\ref{MainTheoremOneOne}]
We will use the coordinate system described in Section
\ref{subsection: coordinates} to identify $H^1(\partial M; \R)$ with
$\R^2$.  We will show that there exists a cone $\cC$ in $\R^2$
containing the positive part of the horizontal axis in its interior
such that every line contained in $\cC$ meets the subset $\TEL{M}$ in a
point which is neither ideal nor parabolic.  The theorem then follows
directly from Lemma \ref{lemma:key} once we invoke \cite[Theorem
1.2]{GordonLuecke1996} to know that all but at most three Dehn
fillings on $M$ are irreducible.

We claim it suffices to produce a path $A$ in $\TEL{M}$ which begins
at a point on the horizontal axis, and not at the origin, such that
the image of $A$ is not completely contained in the horizontal
axis. If the image of $A$ contains points of either the upper or lower
open half-plane, then the symmetries imply that there also exists a
path whose image contains points of the other half-plane; compare
Figure~\ref{fig:arcstocone}.  Thus the images of the two paths will
meet every line in some cone $\cC$.  By Theorem~\ref{thm:structure},
after shrinking these paths if necessary, we may assume that they
contain no ideal points.  Since parabolic points occur only at
integer lattice points, we may also assume that these paths contain no
parabolic points in their interior.

Let $\xi$ be a simple root of $\Delta_M$ that lies on the unit
circle.  Note that $\xi$ is different from $1$ since, as $M$ is a
$\Q$-homology solid torus, the value $|\Delta_M(1)|$ is the order of the
torsion subgroup in $H_1(M; \Z)$ and hence positive. Let $\rho_t$ be
the associated path in $\RG{M}$ given by
Lemma~\ref{lem:abeliandeforms}.  Now $\rho_0$ factors through the
free abelianization $H$ of $\pi_1(M)$, which is just $\Z$, and so it
is trivial to lift $\rho_0$ to $\rhotil_0 \maps \pi_1(M) \to \Gtil$
that still factors through $H$.  As $\lambda$ is $0$ in $H$, we have
$\trans\left(\rhotil_0\left(\lambda\right)\right) = 0$.  As $\xi$ is
not $1$, we have $\trans(\rhotil_0(\mu)) \neq 0$.  As noted in
Section~\ref{subsection: coordinates}, the index of $\pair{\inc_*(\mu)}$
in $H$ is the order $k$ of $\inc_*(\lambda)$ in $H_1(M; \Z)$.  Thus,
using Section~\ref{sec:paramlifts}, we adjust $\rhotil_0$ so that
$\trans(\rhotil_0(\mu))$ is in $(0, k]$.  In particular, $\rhotil_0$ gives
a point $(x, 0) \in \TEL{M}$ with $x > 0$ in our coordinates on
$H^1(\partial M; \R)$.

As discussed in Section~\ref{sec:euler}, the Euler class is the
complete obstruction to lifting a representation to $\Gtil$ and is
constant on connected components of $\RG{M}$.  Hence, as $\rho_0$
lifts to $\rhotil_0$, we can extend this to a continuous path
$\rhotil_t \maps [0, 1] \to \RGtil{M}$ lifting the original $\rho_t$.
Because $\xi^k \neq 1$, we have
$\tr^2_\mu(\rhotil_0) = \xi^k + 2 + \xi^{-k} < 4$, so there exists
$\epsilon > 0$ such that $\tr^2_\mu(\rho_t) < 4$ for
$t \in [0,\epsilon]$.  This means that the representation $\rho_t$
sends $\mu$ to an elliptic element and, since $\lambda$ commutes with
$\mu$, it must also send $\lambda$ to an elliptic or trivial element.
By replacing $\rho_t$ by its restriction to a subinterval of positive
length, we have that $\rho_t$ is a path in $\RGPE{M}$ and that
$\rhotil_t$ is a path in $\RGtilPE{M}$.

We now build our path $A$ by composing $\rhotil_t$ with
$\trans \circ \inc^* \maps \RGtilPE{M} \to \TEL{M}$.  By
Lemma~\ref{lem:abeliandeforms}(\ref{item:charvaries}), we know
$[\rho_t]$ is a nonconstant path in $X(M)$ and hence $\rhotil_t$ is a
nonconstant path in $\RGtilPE{M}$. However, we must still prove that
$A$ is not contained in the horizontal axis, i.e.~that
$\trans(\rhotil_t(\lambda))$ is not the zero function in $t$.  If it
were, then since $\rho_t(\lambda)$ is always elliptic or trivial, we
would have that $\rho_t(\lambda) = 1$ for all $t$; in particular, all
the $\rho_t$ factor through $\pi_1\left(M(0)\right)$ and so the path
$[\rho_t]$ lies in $X\!\left(M(0)\right) \subset X(M)$.  Thus the
$[\rho_t]$ are in an irreducible component $Z$ of
$X\!\left(M(0)\right)$ of complex dimension at least 1.  By
Lemma~\ref{lem:abeliandeforms}(\ref{item:irreducible}), the $\rho_t$
are irreducible for $t > 0$, and thus $Z$ is a component of
$X\!\left(M(0)\right)$ of positive dimension which contains an
irreducible character.  This contradicts our hypothesis that $M$ is
longitudinally rigid, and completes the proof of the theorem.
\end{proof}

\begin{remark}
  For general $\Q$-homology solid tori, there can be reducible
  representations that deform to irreducible representations but
  \emph{do not} come from roots of the Alexander polynomial; rather,
  they correspond to roots of certain \emph{twisted} Alexander
  invariants as described in \cite{HeusenerPorti2005}.  However, it
  would not help to consider such representations in the context of
  Theorem~\ref{MainTheoremOneOne}: as we now explain, the additional
  representations never lift to $\Gtil$ and hence are of no interest
  to us here.  Specifically, consider a representation
  $\rho \maps \pi_1(M) \to S \leq G$ where $S = \PSO_2 \cong S^1$; in
  the proof of Theorem~\ref{MainTheoremOneOne}, we considered such
  $\rho$ that factor through $\Honefree$ and deform to irreducible
  representations in $\RG{M}$.  More generally, we could consider any
  deformable $\rho \maps \pi_1(M) \to S \leq G$.  However, the
  preimage of $S$ in $\Gtil$ is $\R$, which is abelian and
  torsion-free; thus if $\rho$ lifts to
  $\rhotil \maps \pi_1(M) \to \Gtil$, the lift $\rhotil$ must factor
  through $\Honefree$, and so we are back in the case considered in
  Theorem~\ref{MainTheoremOneOne}.
\end{remark}

\section{Real embeddings of trace fields and orderability}
\label{sec:realtraces}

This section gives the proof of Theorem~\ref{MainTheoremTwo}, whose
statement we recall below after giving some needed background.

\subsection{Trace fields and Galois conjugate representations}  
\label{sec:galois}

Let $M$ be a compact orientable \3-manifold whose boundary is a torus.
The trace field of a representation $\rho \maps \pi_1(M) \to \GC$ is
the subfield of $\C$ generated over $\Q$ by the traces of all
$\rho(\gamma)$ for $\gamma \in \pi_1(M)$; this is well-defined even
though the trace of each $\rho(\gamma)$ only makes sense up to sign.
Of course, the trace field depends only on the conjugacy class of
$\rho$.  If $\rhohyp$ is a holonomy representation of a finite-volume
hyperbolic structure on the interior of $M$, by local rigidity its
trace field $F$ is a number field, that is, a finite extension of $\Q$
\cite[Theorem 3.1.2]{MaclachlanReid2003}.  In particular, $F$ is
contained in the subfield $\QQbar \subset \C$ of all algebraic
numbers.

As the hyperbolic structure has a cusp, we can conjugate $\rhohyp$ so
that its image lies in $\PSL{2}{F}$
\cite[Theorem~3.3.8]{MaclachlanReid2003}.  Given an embedding
$\sigma \maps F \to \C$, which must have image contained in $\QQbar$,
we get a \emph{Galois conjugate representation}
$\rho \maps \pi_1(M) \to \GC$ by composing $\rhohyp$ with the induced
map $\PSL{2}{F} \to \PSL{2}{\left(\sigma(F)\right)}$.  As irreducible
representations into $\GC$ are determined by their characters, up to
conjugacy in $\GC$ this $\rho$ depends only on $\sigma$ and not on how
we conjugated $\rhohyp$ to lie in $\PSL{2}{F}$.

Here is the statement that this section is devoted to proving:
\maintheoremtwo*
\noindent
The proof relies on the following three lemmas, the third of which was
suggested to us by Ian Agol and David Futer. 

\begin{lemma}\label{lem:transodd}
  Suppose $M$ is a hyperbolic $\Z$-homology solid torus, with
  homological longitude $\lambda \in \pi_1(\partial M)$.  Suppose
  the trace field $F$ of $M$ has a real embedding
  $\sigma \maps F \to \R$, and let $\rho \maps \pi_1(M) \to \G$ be the
  corresponding Galois conjugate of the holonomy representation.  If
  $\rhotil \maps \pi_1(M) \to \Gtil$ is any lift of $\rho$, then
  $\trans\big(\rhotil(\lambda)\big)$ is an odd integer.
\end{lemma}

\begin{lemma}\label{lem:galois}
  Suppose $M$ is an orientable 1-cusped hyperbolic \3-manifold whose
  trace field has a real embedding.  Then there exists an arc $c$ in
  $\RG{M}$ and a representation $\rho$ in its interior such
  that
  \begin{enumerate}
    \item The representation $\rho$ is a Galois conjugate of
      a holonomy representation of the hyperbolic structure on $M$.  
    \item For any slope $\gamma \in \pi_1(\partial M)$, the arc $c$ is
      parameterized near $\rho$ by $\tr^2_\gamma$.
  \end{enumerate}
\end{lemma}

\begin{lemma}[(Agol and Futer)]\label{lem:iandave}
  Suppose $\TEL{M}$ contains an arc $A$ that is not horizontal,
  i.e.~that has points with different vertical coordinates.  Then
  there exists an $a > 0$ so that the line $L_r$ meets $\TEL{M}$ for
  all $r$ in $(-a, a)$.
\end{lemma}

\begin{proof}[Proof of Lemma~\ref{lem:transodd}]
Let $\rhohyp \maps \pi_1(M) \to \PSL{2}{F}$ be a holonomy
representation for the hyperbolic structure on $M$. Let
$\rhohyp' \maps \pi_1(M) \to \SL{2}{F}$ be any lifted representation,
which exists by \cite[Proposition 3.1.1]{CullerShalen1983}.
By Corollary~2.4 of \cite{Calegari2006}, we know
$\tr\big(\rhohyp'(\lambda)\big) = -2$ since $\lambda$ is the
boundary of an orientable spanning surface.  The Galois conjugate
$\rho' = \sigma \circ \rhohyp'$ also has
$\tr\left(\rho'(\lambda)\right)= -2$, and note that $\rho'$ is a lift
of $\rho$ to $\SL{2}{\R}$.  Consider the successive quotients
\[
\begin{tikzcd}
\Gtil \arrow{r}[below]{q}
\arrow[bend left=15, start anchor=30, end anchor=150]{rr}[near end]{p}
 & \SL{2}{\R} \arrow{r} & G
\end{tikzcd}
\]
The lemma will follow immediately from the fact that
$\tr\left(\rho'(\lambda)\right) = -2$ once we show:

\begin{claim}
  Suppose $\gtil$ is a parabolic or central element of $\Gtil$ and
  $\gbar$ is its image in $\SL{2}{\R}$.  Then the parity of
  $\trans(\gtil)$ is odd precisely when $\tr(\gbar) = -2$ rather than
  $+2$.
\end{claim}
To see this, consider the subset $P$ of all parabolic or central
elements of $\Gtil$.  (Figure 1 of \cite{Khoi2003} has a detailed
picture of $P$ as well as the subsets of elliptic and hyperbolic
elements; this picture informs our approach here but is not directly 
used.)  Note that every path component of $P$ contains a central
element; this is because downstairs in $G$ any parabolic element can
be connected to the trivial element by a path all of whose interior
points are parabolic, and paths lift to covering spaces.  The
functions $\trans$ and $\tr \circ q$ are both continuous on $\Gtil$
and are integer valued on $P$.  Hence they are constant on each path
component of $P$, and it suffices to prove the claim for central
elements.  There, note that the center $\big\{s^k\big\}$ of $\Gtil$ maps to
the center $\{\pm I\}$ of $\SL{2}{\R}$ via the unique epimorphism
$\Z \to \Z/2$; hence the $s^k$ which map to $-I$ are exactly those
with $k$ odd.  This proves the claim and hence the lemma.
\end{proof}

\begin{proof}[Proof of Lemma~\ref{lem:galois}]
Let $F$ be the trace field of $M$ and
$\rhohyp \maps \pi_1(M) \to \PSL{2}{F}$ be a holonomy representation.
As $F$ has a real embedding, choose $\sigma \in \Gal(\QQbar/\Q)$ such
that $\sigma(F) \subset \R$, and define $\rho \in \RG{M}$ as
$\sigma \circ \rhohyp$.

Now both $R(M)$ and $X(M)$ are defined over $\Q$, that is, they can be
cut out by polynomials with rational coefficients.  Hence
$\Gal(\QQbar/\Q)$ acts coordinate-wise on their $\QQbar$-points.
Since $[\rhohyp]$ comes from the complete hyperbolic structure on $M$,
it is a smooth point of $\SLcharvar{M}$ where the local dimension is
1, see \cite[Corollaire~3.28]{Porti1997}; in particular, the Zariski
tangent space to $X(M)$ at $[\rhohyp]$ is 1-dimensional.  Let $X$ be
the unique $\Q$-irreducible component $X$ of $\SLcharvar{M}$ that
contains $[\rhohyp]$.  (You can construct $X$ by taking the
$\C$-irreducible component $X_0$ of $\SLcharvar{M}$ containing
$[\rhohyp]$, which must be defined over some number field, and then
taking the union of the $\Gal(\QQbar/\Q)$-orbit of $X_0$.)  Since $X$
is invariant under the $\Gal(\QQbar/\Q)$-action, it contains $[\rho]$
as well as $[\rhohyp]$. Finally, the dimension of $X$ (thought of as
an algebraic set over either $\Q$ or $\C$) is 1.

Again by \cite[Corollaire~3.28]{Porti1997}, for any slope
$\gamma \in \pi_1(\partial M)$, the trace function $\tr^2_\gamma$ is a
local parameter for $X$ on a small classical neighborhood of
$[\rhohyp]$ (the reference \cite{Porti1997} works with $\SL{2}{\C}$
rather than $\GC$ character varieties, but this makes no difference
since near both $2$ and $-2$ in $\C$ the map $z \mapsto z^2$ is
injective).  Since $\sigma$ acts on the $\QQbar$-points of $X$ taking
$[\rhohyp]$ to $[\rho]$, it follow that $[\rho]$ is also a smooth
point of $X$ where again any $\tr^2_\gamma$ is a local parameter for
the nearby $\C$ points; this is because whether a regular function is
a local parameter at a smooth point on the curve $X$ can be expressed
purely algebraically and hence is $\Gal(\QQbar/\Q)$-invariant.

Let $\tau$ denote the action of complex conjugation on $X(M)$ as in
Section~\ref{sec:realpoints}.  As $[\rho]$ is a smooth point of a
$1$-dimensional irreducible component of $X(M)$, by
Proposition~\ref{prop:real2} there is a smooth arc $\cbar$ of real
points in $X_\R(M)$ containing $[\rho]$ in its interior.  Since
$\tr^2_\gamma$ gives a local parameter for $X$ near $\rho$, the arc
$\cbar$ must be locally defined simply by the requirement that
$\tr^2_\gamma$ is real. Thus $\cbar$ is parameterized near $[\rho]$ by
the value of $\tr^2_\gamma$ in the interval
$[4 - \epsilon, 4 + \epsilon]$.  Moreover, by restricting $\epsilon$
we can assume every character in $\cbar$ comes from a
$\GC$-irreducible representation; by Lemma~\ref{lem:real_red}, this
means $C \subset \XG{M}$ since $[\rho] \in \XG{M}$.  By
Lemma~\ref{lem:pathlift}, we can lift $\cbar$ to an arc in $c$ in
$\RG{M}$.  As the function $\tr_\gamma^2$ must also be a local
parameter for $c$, we have proved the lemma.
\end{proof}

\begin{figure}
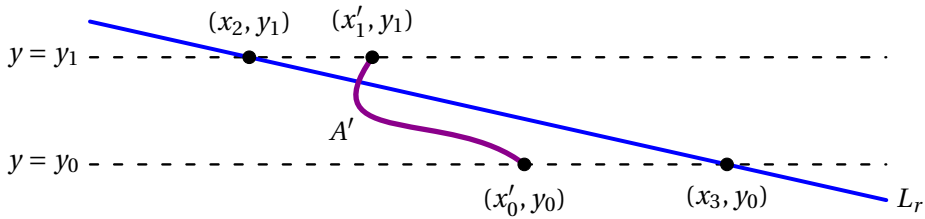

  \begin{center}
    \input plots/Aprime_arc
  \end{center}
  \caption{When the slope of $L_r$ is small enough, it must meet a
    translate $A'$ of $A$.}
  \label{fig:Aprime}
\end{figure}

\begin{proof}[Proof of Lemma~\ref{lem:iandave}]
By shortening the arc if necessary, we first arrange that $A$ lies to
one side of the horizontal axis.  As $\SymTEL{M}$ preserves $\TEL{M}$
and contains $\pi$-rotation about the origin, we may assume that $A$
lies below this axis.  We will show that there exists an $a_1 > 0$ so
that $L_r$ meets $\TEL{M}$ for all $r$ in $(0, a_1)$.  Applying the
symmetric argument to the $\pi$-rotation of $A$ about the origin will
give an $a_2 > 0$ so that $L_r$ meets $\TEL{M}$ for all $r$ in
$(-a_2, 0)$; taking $a = \min(a_1, a_2)$ will then give the promised
interval, since the horizontal axis itself is always part of $\TEL{M}$.

As usual, let $k$ be the order of $\iota_*(\lambda)$ in $H_1(M; \Z)$,
so that $\SymTEL{M}$ contains the subgroup of horizontal
translations by multiples of $k$. By shortening $A$ if necessary, we
can label its endpoints as $(x_0, y_0)$ and $(x_1, y_1)$ where
$y_0 < y_1 < 0$ and $\abs{x_1 - x_0} < k$.

We claim that $L_r$ meets $\TEL{M}$ for all $r$ where
\begin{equation}\label{eq:slopebound}
0 < r < \frac{y_1 - y_0}{2k}
\end{equation}
To see this, let $(x_2, y_1)$ be the point where $L_r$ meets the horizontal line
$y = y_1$, and let $(x_3, y_0)$ be the point where $L_r$ meets $y =
y_0$.  Consider the largest integer $n$ so that $x_0 + n k \leq x_3$, and 
let $A' \subset \TEL{M}$ be $A$ translated to the right by $n k$, so
the endpoints of $A'$ are $(x_0 + nk, y_0)$ and $(x_1 + nk, y_1)$.
Set $x'_0 = x_0 + nk$ and $x'_1 = x_1 + nk$.

We now argue that $L_r$ meets $A'$, using Figure~\ref{fig:Aprime} as a
guide.  Since the slope of $L_r$ is $-r$, and since $(y_1 - y_0)/r > 2 k$ by
(\ref{eq:slopebound}), we have $$x_3 - x_2 = (y_1 - y_0)/r > 2 k.$$
Our choice of $n$ guarantees that $x'_0 < x_3$ and $\abs{x_3 - x'_0}  < k$.
We also have $\abs{x_1' - x_0'} = \abs{x_1 - x_0} < k$.  Thus
\[
\abs{x_3 - x_1'} \leq \abs{x_3 - x_0'} + \abs{x_0' - x_1'} < 2 k.
\]
Combining, we conclude that $x_2 < x_1'$.  We also have $x'_0 < x_3$,
so we have shown that the endpoints of $A'$ lie on opposite sides of
$L_r$, as in Figure~\ref{fig:Aprime}.  This implies that $L_r$ must
meet $A'$, completing the proof of the lemma.
\end{proof}

\begin{proof}[Proof of Theorem~\ref{MainTheoremTwo}]
Let $c$ be the arc in $\RG{M}$ given by Lemma~\ref{lem:galois}, and
$\rho$ the Galois conjugate of the holonomy representation which is in
$c$.  As $H^2(M; \Z) = 0$, the Euler class of any representation in
$c$ vanishes, and hence we can lift $c$ to an arc $\ctil$ in
$\RGtil{M}$.  We fix a particular lift by requiring that $\rho$ lifts
to $\rhotil$ with $\trans(\rhotil(\mu)) = 0$.  By
Lemma~\ref{lem:transodd}, we have that $\trans(\rhotil(\lambda)) = k$
is an odd integer, and so $\rhotil$ gives rise to the point $(0, k)$
in $\TEL{M}$.

Since this is true downstairs for $c$, the function $\tr^2_\mu$ is a
local parameter for $\ctil$ where the parameter takes values in
$[4 - \epsilon, 4 + \epsilon]$.  For the subinterval
$[4 - \epsilon, 4]$, the representations on $\ctil$ must lie in
$\RGtilPE{M}$ since they each send $\mu$ to a parabolic or elliptic
element of $\Gtil$. In particular, the translation number of $\mu$ is
a local parameter for this portion of $\ctil$.

Thus we get an arc $A$ in $\TEL{M}$ which starts from $(0, k)$, where
$k$ is the aforementioned odd integer, and is locally parameterized by
the $\mu^*$-coordinate on some small interval $[0, \delta]$.
Moreover, by construction no point on $A$ is an ideal point, and the
only parabolic point on $A$ is $(0, k)$ itself.  Depending on the
sign of $k$, we get one of the two pictures in Figure~\ref{fig:thm2}.

\begin{figure}
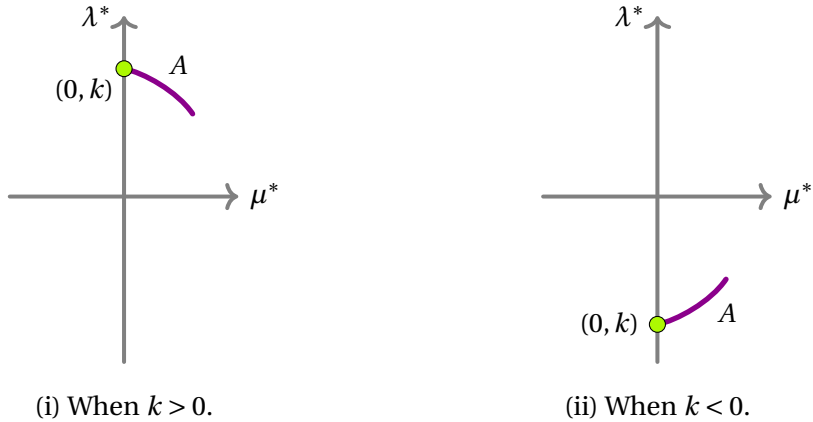

  \begin{center}
    \input plots/arc_from_parabolic
  \end{center}
  \caption{The two possibilities for the arc $A$ in the locus $\TEL{M}$
  originating from the image $(0, k)$ of a Galois conjugate of the
  holonomy representation.}
  \label{fig:thm2}
\end{figure}
 
To prove conclusion (\ref{item:branched}), consider an $n$-fold
cyclic cover $\Ytil$ of $Y$ branched over $K$.  First, by
the Hyperbolic Dehn Surgery Theorem, the manifold $\Ytil$ is hyperbolic
and hence irreducible for all large $n$.  Moreover, from
Figure~\ref{fig:thm2} it is clear that for all large $n$ the arc $A$
meets the line $\mu^* = 1/n$, so we now get $(a)$ directly from
Lemma~\ref{lemma:branched}. 

For (\ref{item:fill}), for concreteness let us focus on possibility
(ii) in Figure~\ref{fig:thm2}.  Since there are at most three Dehn
fillings on $M$ that are reducible \cite[Theorem
1.2]{GordonLuecke1996}, we can construct an interval $I = (a, \infty)$
where $M(r)$ is irreducible and $L_r$ meets $A$ for all $r \in I$.
The claim now follows immediately from Lemma~\ref{lemma:key}.

Finally, for part (\ref{item:newfill}), by Lemma~\ref{lem:galois} the
arc $c$ in $\RG{M}$ is parameterized near $\rho$ by $\tr_\lambda^2$;
thus the corresponding arc $A$ in $\TEL{M}$ is not horizontal.  Hence
by Lemma~\ref{lem:iandave}, the line $L_r$ meets $\TEL{M}$ for all $r$
in some open interval $(-b, b)$.  Shrinking $b$, we can ensure that
$M(r)$ is irreducible for all $r$ in $(-b, 0) \cup (0, b)$.  Again, 
claim (\ref{item:newfill}) now follows immediately from
Lemma~\ref{lemma:key}, completing the proof of the theorem. 
\end{proof}

\begin{remark}
  \label{rem:allL}
  The hypothesis that $Y$ is a $\Z$-homology \3-sphere is certainly
  necessary for the proof of Theorem~\ref{MainTheoremTwo} to work, and
  it is likely that the conclusion of Theorem~\ref{MainTheoremTwo}
  does not hold in general if one drops this hypothesis.
  Specifically, consider the 1-cusped hyperbolic \3-manifold
  $M = v2503$ which has $H_1(M; \Z) = \Z + \Z/10$ and
  $H^2(M; \Z) = \Z/10$. The trace field here is $\Q$ adjoin a root of
  \[
  x^{10} - 4 x^{8} + 9 x^{6} - 15 x^{4} + 12 x^{2} - 2
  \]
  which has six real embeddings.  However, none of the resulting
  representations $\pi_1(M) \to G$ lift to $\Gtil$, completely
  stymying our technique for constructing orders.

  This $M$ is interesting from the point of view of Floer theory;
  specifically, Lidman and Watson recently gave infinitely many
  $\Q$-homology solid tori which were not fibered and where every
  non-longitudinal Dehn filling is an $L$-space
  \cite{LidmanWatson2014}.  As their examples all have essential
  annuli, they asked \cite[Question 6]{LidmanWatson2014} whether there
  are hyperbolic examples with these same properties; the manifold
  $v2503$ answers that question affirmatively, as we now explain.  We
  will use the homological framing $(\mu, \lambda)$ which corresponds
  to $(0, 1)$ and $(-1, 0)$ in SnapPy's default conventions.  Then
  $M(\mu)$ is the lens space $L(50, 19)$ and $M(\lambda)$ is
  $S^2 \times S^1 \#\ \RP^{3}$.  Using \cite{RasmussenRasmussen2015},
  it is possible to show that \emph{every} non-longitudinal Dehn
  filling on $M$ is an $L$-space, even though it is not a fibered
  \3-manifold as $\Delta_M = 2(t^{4} + t^{3} + t^{2} + t + 1)$.

  Of course, if Conjecture~\ref{BGWconjecture} is true, then every
  Dehn filling on $M$ is \emph{not} orderable (the filling
  $M(\lambda)$ is not orderable as its fundamental group has torsion).
  We checked the 16 examples where the Dehn filling coefficients are at
  most 3, and in each case we were able to show that the corresponding
  Dehn filling was not orderable.  It would be interesting to show
  that this is the case for all Dehn fillings.  
\end{remark}

\section{Open questions}
\label{sec:open}

Our results in this paper and especially the examples in
Section~\ref{sec:menagerie} suggest many interesting questions and
possible avenues for future research; here are some of them:

\begin{enumerate}[label={(\arabic*)}]
\item Find topological hypotheses on a $\Z$-homology solid torus
  which imply that all Dehn surgeries in $(-1, 1)$ are orderable.

\item Find topological hypotheses which give rise to the behavior
  shown in Figure~\ref{fig:o9_04139} where one can use $\TEL{M}$ to
  order all but one Dehn filling on $M$.

\item Do all Berge knots have $\TEL{M}$ of the simple form shown in
  Figure~\ref{fig:m016} and Figure~\ref{fig:v0220}?  What about
  twisted torus knots?  In the latter case, perhaps one can view
  $\TEL{M}$ as some kind of ``perturbation'' of the very simple
  $\TEL{M}$ of the underlying torus knot.

\item In Lemma~\ref{lem:compact} we show that $\TEL{M}$ lives in a
  horizontal strip whose size is bounded.  When $M$ is a $\Z$-homology
  solid torus, our proof shows that the maximum $y$ coordinate of a point
  in $\TEL{M}$ is $2 g - 1$, where $g$ is the Seifert genus of $M$.
  In our examples, this bound is never sharp.  Is this always the
  case, and regardless, is there some way to understand this gap?  

\item Does every polynomial satisfying the conclusion of 
  \cite[Corollary 1.3]{OSLensSpace2005} have a \emph{simple} root on
  the unit circle?  Note that by \cite{KonvalinaMatache2004} such a
  polynomial always has a root on the unit circle.  Experimental
  evidence says yes.

\item Can the longitudinally rigid hypothesis in
  Theorem~\ref{MainTheoremOneOne} be eliminated by placing additional
  conditions on $\Delta_M$?  In the known examples where longitudinal
  rigidity comes into play, the ``bad'' roots of $\Delta_M$ are
  all roots of unity.  

\item Also motived by Theorem~\ref{MainTheoremOneOne}, are there
  closed atoroidal 3-manifolds with $\dim H_1(M; \Q) \leq 1$ which do
  not have few characters?  What if one restricts to $0$-surgery on a
  knot in $S^3$?  

\item There is a Chern-Simons invariant/Seifert volume/Godbillon-Vey
  invariant associated to each representation in $\RG{M}$, see
  \cite{Khoi2003}.  In our usual coordinates on $\TEL{M}$, the
  derivative is really simple, basically
  $x \mathit{dy} - y \mathit{dx}$.  Can this invariant be used to prove
  something interesting about $\TEL{M}$?

\item How can one explore the space of actions of $\pi_1(M)$ on $\R$
  so as to include some which do not arise from $\Gtil$
  representations?  It is natural to try to use some analog of the
  character variety to do this.  What is the appropriate setting for
  this?  Is it possible to draw pictures like those in
  Section~\ref{sec:menagerie} that are built from some larger class of
  maps to $\widetilde{\Homeo^+(\R)}$?

\item Motivated by Remark~\ref{rem:allL}, prove that every Dehn
  filling on $v2503$ is not orderable.

\end{enumerate}

{\RaggedRight 
  \small
  \bibliographystyle{nmd/math}
  \bibliography{\jobname} }
\end{document}

%% file: plots/diagram.tex
\begin{tikzpicture}
      [scale=0.93,
       every node/.style={align=center},
       font=\small,
       main/.style={ellipse, draw=black!50, line width=1pt, outer sep=5pt},
       other/.style={inner sep=8pt}
      ]
      \node[main] (taut) at (0,0) {$Y$ has a taut \\ foliation $\cF$};
      \node[main] (orderable) at (9.5, 6)
          {$Y$ is orderable \\ $\Longleftrightarrow$ 
           \\ $\pi_1(Y)$ acts on $\R$}; 
      \node[main] (nonL) at (0, 6) {$Y$ is not \\ an $L$-space}; 

      \node[other] (circle) at (9.5, 0) {$\pi_1(Y)$ acts \\ on $S^1$};

      \node[other, inner sep=3pt] (leaf) at (4.75, 3) 
         {$\pi_1(Y)$ acts on a simply \\ connected 
        1-manifold \\ (possibly non-Hausdorff)};
    
      \begin{scope}[
        line width=1pt,
        double distance=4pt,
      ]
      % From taut
      \draw[-implies, double, bend right=5]    
           (taut.east) to node[below=4pt] 
           {Thurston's universal circle \cite{CalegariDunfield2003}}
           (circle.west);
      \draw[-implies, double] ([xshift=-15pt]taut.north) 
         to [bend left=10] 
         node[left=3pt] {
           \cite{OSgenusbounds2004} \\ 
             and\\ 
           \cite{Bowden2015, KazezRoberts2015}}
         ([xshift=-15pt]nonL.south);
      \draw[-implies, double, bend left=10] 
         (taut.north east) to 
         node[below right] {Leaf space of \\  $\cFtil$ in $\Ytil$}
         (leaf.south west); 

      % From nonL
      \draw[-implies, double] ([xshift=15pt]nonL.south)
         to [bend left=10] ([xshift=15pt]taut.north);

      \draw[line width=10pt, loosely dashed, color=white, bend left=10]
         ([xshift=15pt]nonL.south) to ([xshift=15pt, yshift=9pt]taut.north);
      
      \draw[implies-implies, double, bend left=5] 
          (nonL.east) to node[above=3pt] 
          {Conjecture of \cite{BoyerGordonWatson2013}}
          (orderable.west);
     % hack to make it dashed
     \draw[line width=8pt, loosely dashed, color=white, bend left=5] 
         ([xshift=9pt]nonL.east) to ([xshift=-9pt]orderable.west);

      % From orderable
      \draw[-implies, double, bend left=5] (orderable.south west) to
         (leaf.north east);

      \draw[-implies, double, bend left=7] 
         ([xshift=10pt]orderable.south) to 
         node[left=3pt] {$\R \cong S^1 - \{\mbox{pt}\}$} 
         ([xshift=10pt]circle.north);
      \end{scope}
\end{tikzpicture}

%% file: plots/s841_small.tex
\begin{tikzoverlayabs}[width=\matplotlibfigurewidth]{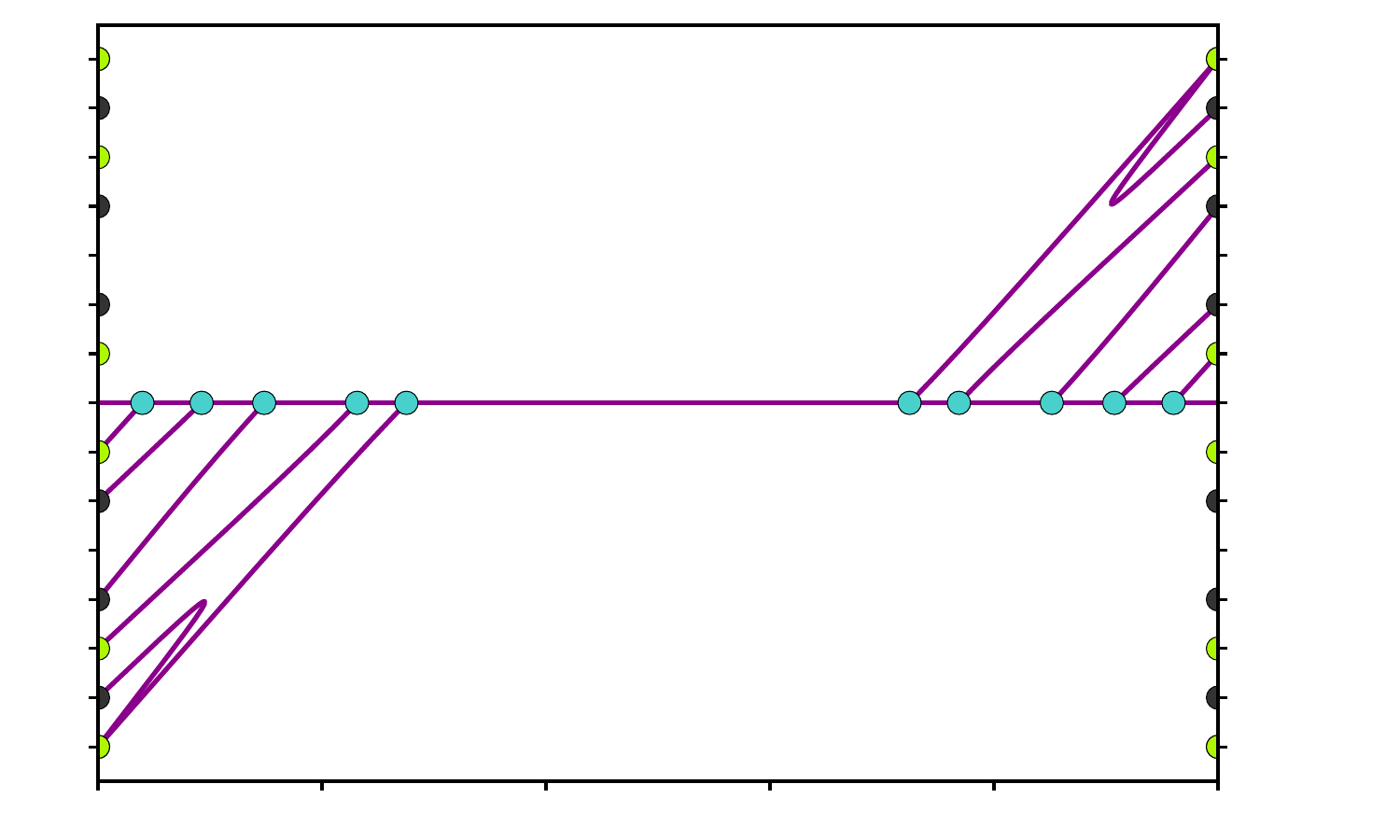}[\matplotlibfigurefont]
  \draw (0.086000, 0.952000) node[below right] {$s841$};
  \draw (0.070000, 0.049167) node[below] {$0.0$};
  \draw (0.230000, 0.049167) node[below] {$0.2$};
  \draw (0.390000, 0.049167) node[below] {$0.4$};
  \draw (0.550000, 0.049167) node[below] {$0.6$};
  \draw (0.710000, 0.049167) node[below] {$0.8$};
  \draw (0.870000, 0.049167) node[below] {$1.0$};
  \draw (0.057500, 0.169441) node[left] {$-6$};
  \draw (0.057500, 0.286429) node[left] {$-4$};
  \draw (0.057500, 0.403418) node[left] {$-2$};
  \draw (0.057500, 0.520407) node[left] {$0$};
  \draw (0.057500, 0.637396) node[left] {$2$};
  \draw (0.057500, 0.754385) node[left] {$4$};
  \draw (0.057500, 0.871374) node[left] {$6$};
  \node[right] at (0.89, 0.08) {$\mu^*$};
  \node[left] at (0.06000, 0.96)  {$\lambda^*$};
\end{tikzoverlayabs}

%% file: plots/o9_04139_small.tex
\begin{tikzoverlayabs}[width=\matplotlibfigurewidth]{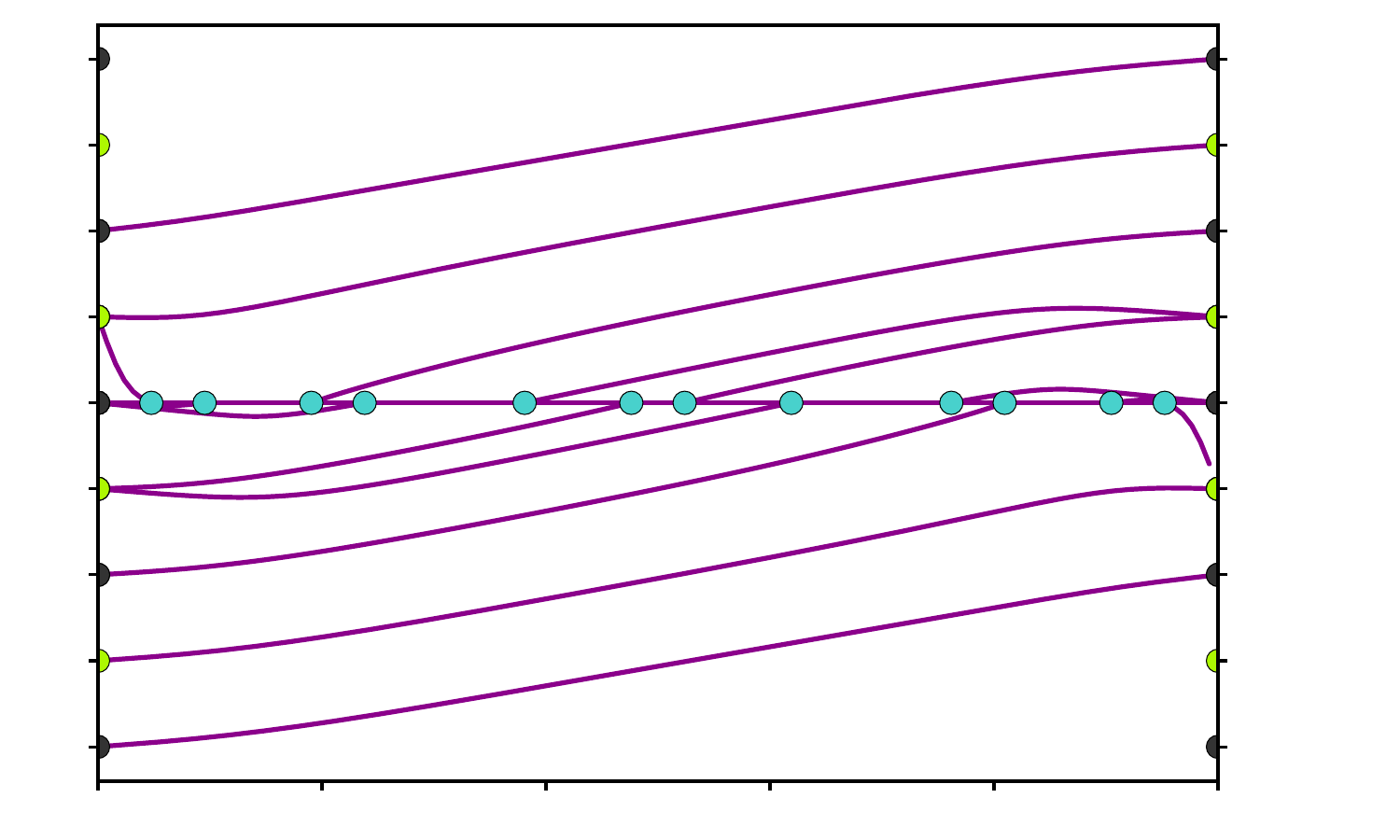}[\matplotlibfigurefont]
  \draw (0.086000, 0.952000) node[below right] {$o9_{04139}$};
  \draw (0.070000, 0.049167) node[below] {$0.0$};
  \draw (0.230000, 0.049167) node[below] {$0.2$};
  \draw (0.390000, 0.049167) node[below] {$0.4$};
  \draw (0.550000, 0.049167) node[below] {$0.6$};
  \draw (0.710000, 0.049167) node[below] {$0.8$};
  \draw (0.870000, 0.049167) node[below] {$1.0$};
  \draw (0.057500, 0.110948) node[left] {$-4$};
  \draw (0.057500, 0.315690) node[left] {$-2$};
  \draw (0.057500, 0.520432) node[left] {$0$};
  \draw (0.057500, 0.725174) node[left] {$2$};
  \draw (0.057500, 0.929916) node[left] {$4$};
  \node[right] at (0.89, 0.08) {$\mu^*$};
  \node[left] at (0.0300, 0.96)  {$\lambda^*$};
\end{tikzoverlayabs}

%% file: plots/t03632_small.tex
\begin{tikzoverlayabs}[width=\matplotlibfigurewidth]{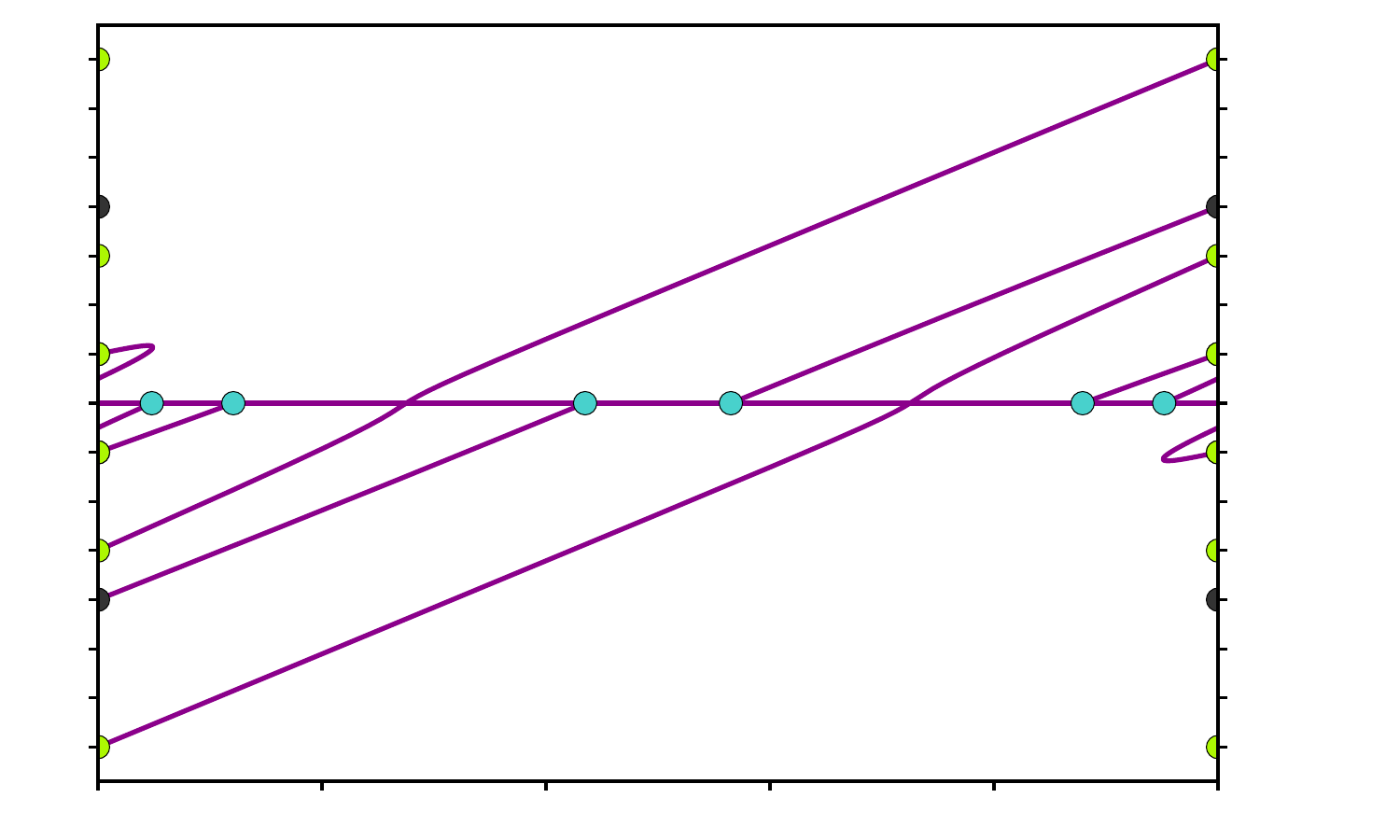}[\matplotlibfigurefont]
  \draw (0.086000, 0.952000) node[below right] {$t03632$};
  \draw (0.070000, 0.049167) node[below] {$0.0$};
  \draw (0.230000, 0.049167) node[below] {$0.2$};
  \draw (0.390000, 0.049167) node[below] {$0.4$};
  \draw (0.550000, 0.049167) node[below] {$0.6$};
  \draw (0.710000, 0.049167) node[below] {$0.8$};
  \draw (0.870000, 0.049167) node[below] {$1.0$};
  \draw (0.057500, 0.169119) node[left] {$-6$};
  \draw (0.057500, 0.286079) node[left] {$-4$};
  \draw (0.057500, 0.403040) node[left] {$-2$};
  \draw (0.057500, 0.520000) node[left] {$0$};
  \draw (0.057500, 0.636960) node[left] {$2$};
  \draw (0.057500, 0.753921) node[left] {$4$};
  \draw (0.057500, 0.870881) node[left] {$6$};
  \node[right] at (0.89, 0.08) {$\mu^*$};
  \node[left] at (0.06000, 0.96)  {$\lambda^*$};
\end{tikzoverlayabs}

%% file: plots/o9_30426_small.tex
\begin{tikzoverlayabs}[width=\matplotlibfigurewidth]{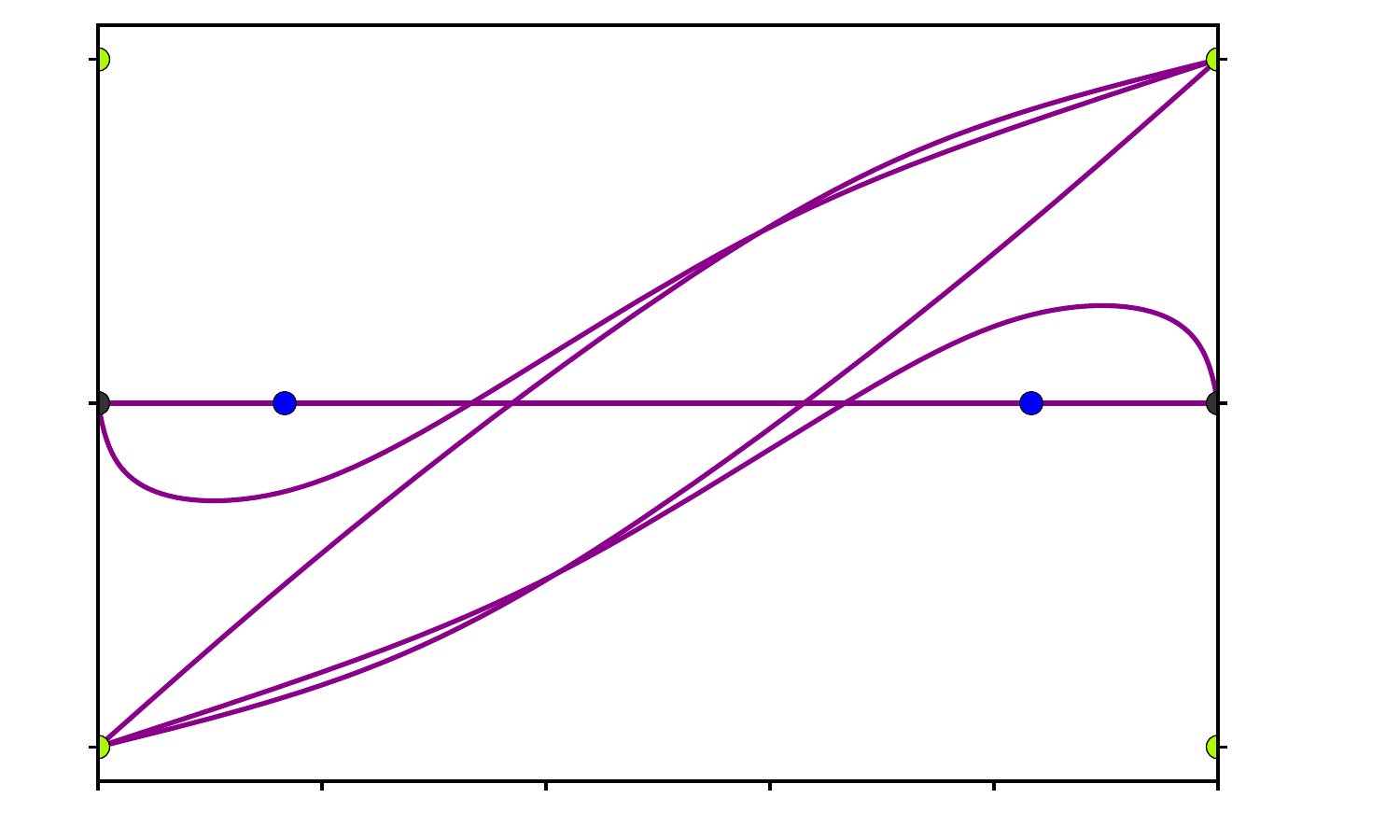}[\matplotlibfigurefont]
  \draw (0.086000, 0.952000) node[below right] {$o9_{30426}$};
  \draw (0.070000, 0.049167) node[below] {$0.0$};
  \draw (0.230000, 0.049167) node[below] {$0.2$};
  \draw (0.390000, 0.049167) node[below] {$0.4$};
  \draw (0.550000, 0.049167) node[below] {$0.6$};
  \draw (0.710000, 0.049167) node[below] {$0.8$};
  \draw (0.870000, 0.049167) node[below] {$1.0$};
  \draw (0.057500, 0.110748) node[left] {$-1$};
  \draw (0.057500, 0.520000) node[left] {$0$};
  \draw (0.057500, 0.929252) node[left] {$1$};
  \node[right] at (0.89, 0.08) {$\mu^*$};
  \node[left] at (0.03000, 0.96)  {$\lambda^*$};
\end{tikzoverlayabs}

%% file: plots/m016.tex
\begin{tikzoverlayabs}[width=\matplotlibfigurewidth]{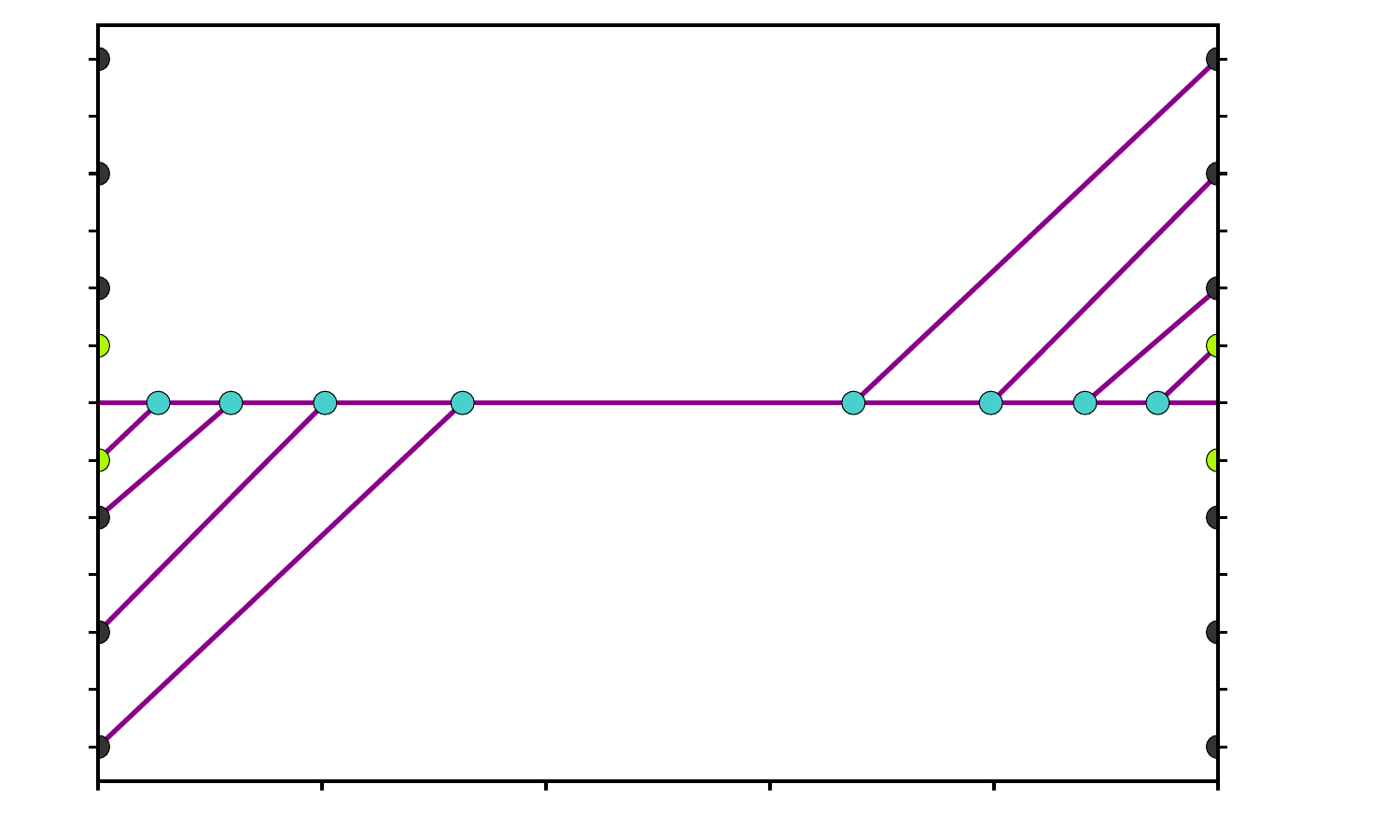}[\matplotlibfigurefont]
  \draw (0.086000, 0.952000) node[below right] {\small $m016$};
  \draw (0.070000, 0.049167) node[below] {$0.0$};
  \draw (0.230000, 0.049167) node[below] {$0.2$};
  \draw (0.390000, 0.049167) node[below] {$0.4$};
  \draw (0.550000, 0.049167) node[below] {$0.6$};
  \draw (0.710000, 0.049167) node[below] {$0.8$};
  \draw (0.870000, 0.049167) node[below] {$1.0$};
  \draw (0.057500, 0.110940) node[left] {$-6$};
  \draw (0.057500, 0.179173) node[left] {$-5$};
  \draw (0.057500, 0.247406) node[left] {$-4$};
  \draw (0.057500, 0.315639) node[left] {$-3$};
  \draw (0.057500, 0.383872) node[left] {$-2$};
  \draw (0.057500, 0.452105) node[left] {$-1$};
  \draw (0.057500, 0.520338) node[left] {$0$};
  \draw (0.057500, 0.588571) node[left] {$1$};
  \draw (0.057500, 0.656804) node[left] {$2$};
  \draw (0.057500, 0.725037) node[left] {$3$};
  \draw (0.057500, 0.793270) node[left] {$4$};
  \draw (0.057500, 0.861503) node[left] {$5$};
  \draw (0.057500, 0.929736) node[left] {$6$};
  \node[below] at (0.47, 0.0) {\small $\mu^*$};
  \node[left] at (0.017500, 0.520338)  {\small $\lambda^*$};
\end{tikzoverlayabs}

%% file: plots/v0220.tex
\begin{tikzoverlayabs}[width=\matplotlibfigurewidth]{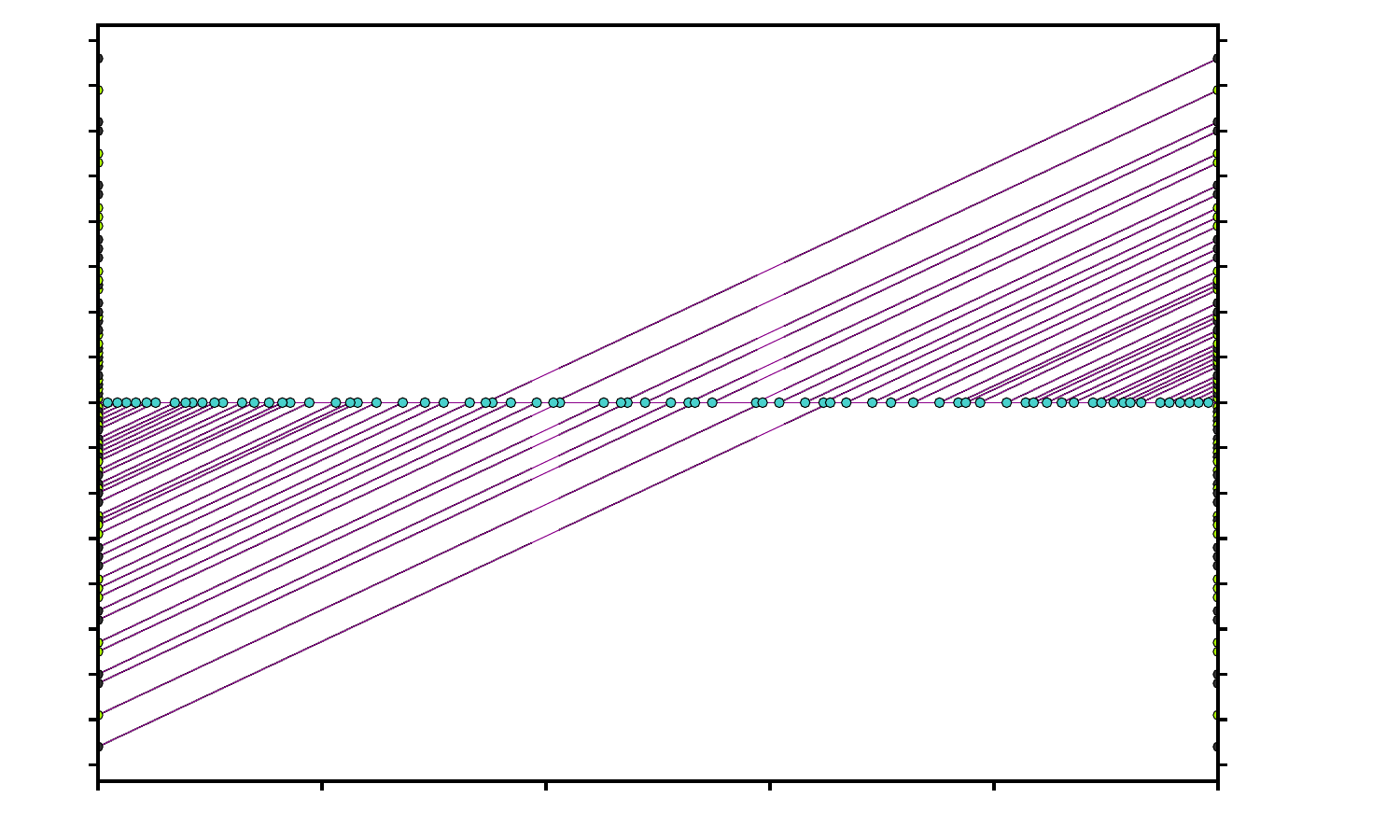}[\matplotlibfigurefont]
  \draw (0.086000, 0.952000) node[below right] {\small $v0220$};
  \draw (0.070000, 0.049167) node[below] {$0.0$};
  \draw (0.230000, 0.049167) node[below] {$0.2$};
  \draw (0.390000, 0.049167) node[below] {$0.4$};
  \draw (0.550000, 0.049167) node[below] {$0.6$};
  \draw (0.710000, 0.049167) node[below] {$0.8$};
  \draw (0.870000, 0.049167) node[below] {$1.0$};
  \draw (0.057500, 0.466766) node[left] {$-10$};
  \draw (0.057500, 0.412857) node[left] {$-20$};
  \draw (0.057500, 0.358949) node[left] {$-30$};
  \draw (0.057500, 0.305040) node[left] {$-40$};
  \draw (0.057500, 0.251132) node[left] {$-50$};
  \draw (0.057500, 0.197224) node[left] {$-60$};
  \draw (0.057500, 0.143315) node[left] {$-70$};
  \draw (0.057500, 0.089407) node[left] {$-80$};
  \draw (0.057500, 0.520674) node[left] {$0$};
  \draw (0.057500, 0.574582) node[left] {$10$};
  \draw (0.057500, 0.628491) node[left] {$20$};
  \draw (0.057500, 0.682399) node[left] {$30$};
  \draw (0.057500, 0.736307) node[left] {$40$};
  \draw (0.057500, 0.790216) node[left] {$50$};
  \draw (0.057500, 0.844124) node[left] {$60$};
  \draw (0.057500, 0.898033) node[left] {$70$};
  \draw (0.057500, 0.951941) node[left] {$80$};
  \node[right] at (0.89, 0.07) {\small $\mu^*$};
  \node[left] at (0.017500, 0.96)  {\small $\lambda^*$};
\end{tikzoverlayabs}

%% file: plots/o9_34801.tex
\begin{tikzoverlayabs}[width=\matplotlibfigurewidth]{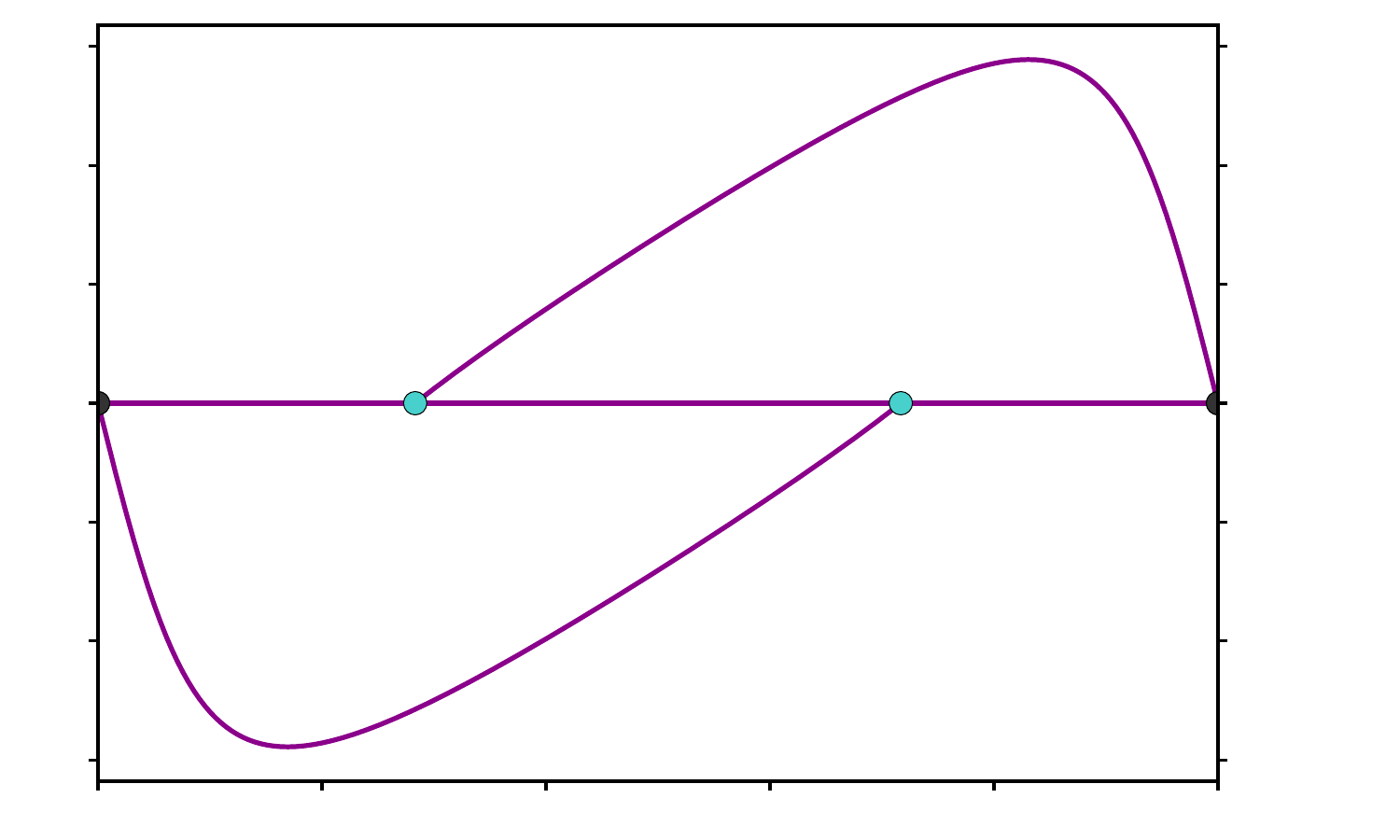}[\matplotlibfigurefont]
  \draw (0.086000, 0.952000) node[below right] {\small $o9_{34801}$};
  \draw (0.070000, 0.049167) node[below] {$0.0$};
  \draw (0.230000, 0.049167) node[below] {$0.2$};
  \draw (0.390000, 0.049167) node[below] {$0.4$};
  \draw (0.550000, 0.049167) node[below] {$0.6$};
  \draw (0.710000, 0.049167) node[below] {$0.8$};
  \draw (0.870000, 0.049167) node[below] {$1.0$};
  \draw (0.057500, 0.095273) node[left] {$-0.3$};
  \draw (0.057500, 0.236848) node[left] {$-0.2$};
  \draw (0.057500, 0.378424) node[left] {$-0.1$};
  \draw (0.057500, 0.520000) node[left] {$0.0$};
  \draw (0.057500, 0.661576) node[left] {$0.1$};
  \draw (0.057500, 0.803152) node[left] {$0.2$};
  \draw (0.057500, 0.944727) node[left] {$0.3$};
  \node[right] at (0.89, 0.07) {\small $\mu^*$};
  \node[left] at (0.014500, 0.97)  {\small $\lambda^*$};
\end{tikzoverlayabs}

%% file: plots/t11462.tex
\begin{tikzoverlayabs}[width=\matplotlibfigurewidth]{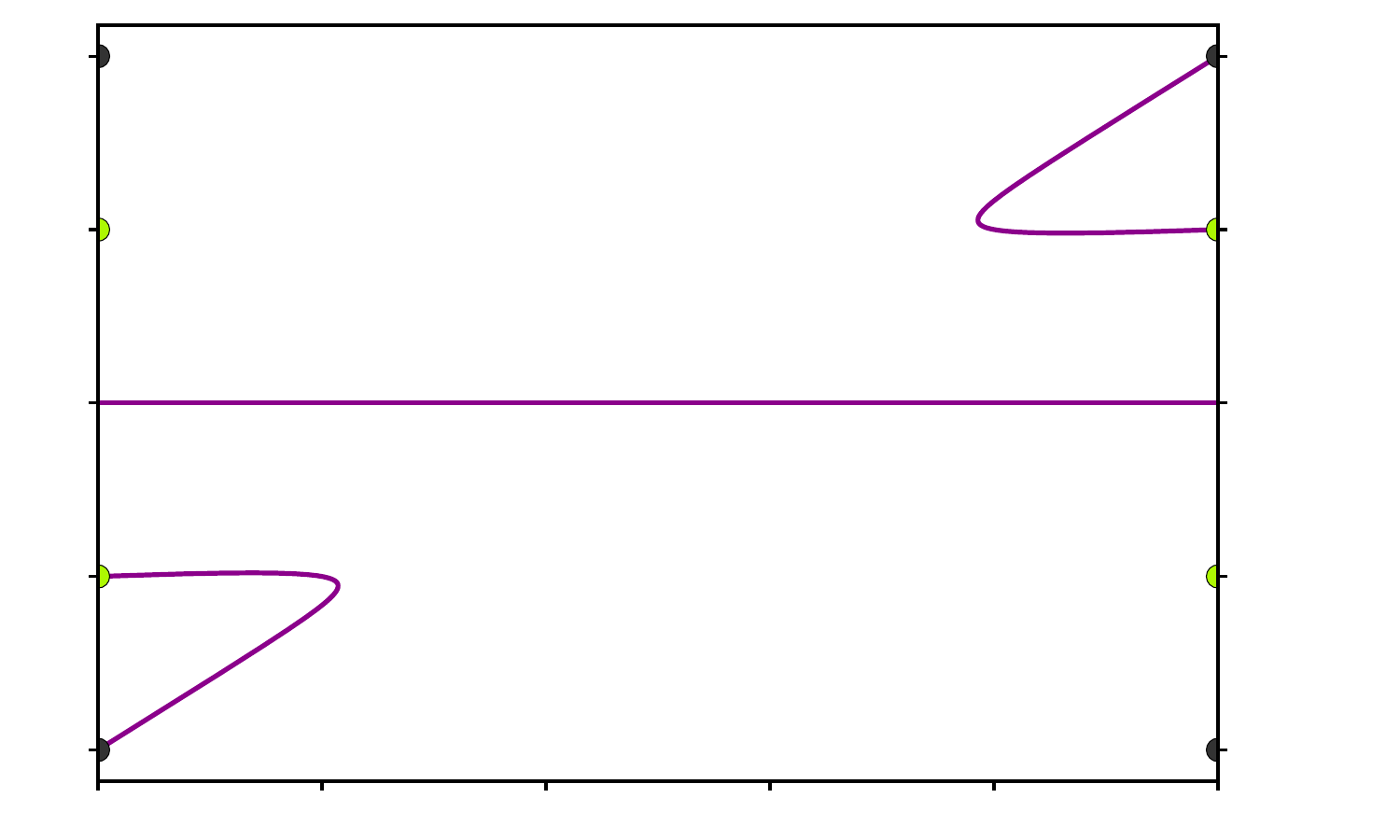}[\matplotlibfigurefont]
  \draw (0.086000, 0.952000) node[below right] {\small $t11462$};
  \draw (0.070000, 0.049167) node[below] {$0.0$};
  \draw (0.230000, 0.049167) node[below] {$0.2$};
  \draw (0.390000, 0.049167) node[below] {$0.4$};
  \draw (0.550000, 0.049167) node[below] {$0.6$};
  \draw (0.710000, 0.049167) node[below] {$0.8$};
  \draw (0.870000, 0.049167) node[below] {$1.0$};
  \draw (0.057500, 0.107269) node[left] {$-2$};
  \draw (0.057500, 0.313746) node[left] {$-1$};
  \draw (0.057500, 0.520224) node[left] {$0$};
  \draw (0.057500, 0.726701) node[left] {$1$};
  \draw (0.057500, 0.933179) node[left] {$2$};
  \node[right] at (0.89, 0.07) {\small $\mu^*$};
  \node[left] at (0.026500, 0.96)  {\small $\lambda^*$};
  \node at (0.2,0.2) {$B$};
  \node at (0.75,0.857) {$A$};
\end{tikzoverlayabs}

%% file: plots/s841.tex
\begin{tikzoverlayabs}[width=\matplotlibfigurewidth]{plots/images/s841.pdf}[\matplotlibfigurefont]
  \draw (0.086000, 0.952000) node[below right] {\small $s841$};
  \draw (0.070000, 0.049167) node[below] {$0.0$};
  \draw (0.230000, 0.049167) node[below] {$0.2$};
  \draw (0.390000, 0.049167) node[below] {$0.4$};
  \draw (0.550000, 0.049167) node[below] {$0.6$};
  \draw (0.710000, 0.049167) node[below] {$0.8$};
  \draw (0.870000, 0.049167) node[below] {$1.0$};
  \draw (0.057500, 0.110946) node[left] {$-7$};
  \draw (0.057500, 0.169441) node[left] {$-6$};
  \draw (0.057500, 0.227935) node[left] {$-5$};
  \draw (0.057500, 0.286429) node[left] {$-4$};
  \draw (0.057500, 0.344924) node[left] {$-3$};
  \draw (0.057500, 0.403418) node[left] {$-2$};
  \draw (0.057500, 0.461913) node[left] {$-1$};
  \draw (0.057500, 0.520407) node[left] {$0$};
  \draw (0.057500, 0.578901) node[left] {$1$};
  \draw (0.057500, 0.637396) node[left] {$2$};
  \draw (0.057500, 0.695890) node[left] {$3$};
  \draw (0.057500, 0.754385) node[left] {$4$};
  \draw (0.057500, 0.812879) node[left] {$5$};
  \draw (0.057500, 0.871374) node[left] {$6$};
  \draw (0.057500, 0.929868) node[left] {$7$};
  \node[right] at (0.89, 0.07) {\small $\mu^*$};
  \node[left] at (0.026500, 0.96)  {\small $\lambda^*$};
\end{tikzoverlayabs}

%% file: plots/o9_04139.tex
\begin{tikzoverlayabs}[width=\matplotlibfigurewidth]{plots/images/o9_04139.pdf}[\matplotlibfigurefont]
  \draw (0.086000, 0.952000) node[below right] {\small $o9_{04139}$};
  \draw (0.070000, 0.049167) node[below] {$0.0$};
  \draw (0.230000, 0.049167) node[below] {$0.2$};
  \draw (0.390000, 0.049167) node[below] {$0.4$};
  \draw (0.550000, 0.049167) node[below] {$0.6$};
  \draw (0.710000, 0.049167) node[below] {$0.8$};
  \draw (0.870000, 0.049167) node[below] {$1.0$};
  \draw (0.057500, 0.110948) node[left] {$-4$};
  \draw (0.057500, 0.213319) node[left] {$-3$};
  \draw (0.057500, 0.315690) node[left] {$-2$};
  \draw (0.057500, 0.418061) node[left] {$-1$};
  \draw (0.057500, 0.520432) node[left] {$0$};
  \draw (0.057500, 0.622803) node[left] {$1$};
  \draw (0.057500, 0.725174) node[left] {$2$};
  \draw (0.057500, 0.827545) node[left] {$3$};
  \draw (0.057500, 0.929916) node[left] {$4$};
  \node[right] at (0.89, 0.07) {\small $\mu^*$};
  \node[left] at (0.026500, 0.96)  {\small $\lambda^*$};
\end{tikzoverlayabs}

%% file: plots/t03632.tex
\begin{tikzoverlayabs}[width=\matplotlibfigurewidth]{plots/images/t03632.pdf}[\matplotlibfigurefont]
  \draw (0.086000, 0.952000) node[below right] {\small $t03632$};
  \draw (0.070000, 0.049167) node[below] {$0.0$};
  \draw (0.230000, 0.049167) node[below] {$0.2$};
  \draw (0.390000, 0.049167) node[below] {$0.4$};
  \draw (0.550000, 0.049167) node[below] {$0.6$};
  \draw (0.710000, 0.049167) node[below] {$0.8$};
  \draw (0.870000, 0.049167) node[below] {$1.0$};
  \draw (0.057500, 0.110639) node[left] {$-7$};
  \draw (0.057500, 0.169119) node[left] {$-6$};
  \draw (0.057500, 0.227599) node[left] {$-5$};
  \draw (0.057500, 0.286079) node[left] {$-4$};
  \draw (0.057500, 0.344559) node[left] {$-3$};
  \draw (0.057500, 0.403040) node[left] {$-2$};
  \draw (0.057500, 0.461520) node[left] {$-1$};
  \draw (0.057500, 0.520000) node[left] {$0$};
  \draw (0.057500, 0.578480) node[left] {$1$};
  \draw (0.057500, 0.636960) node[left] {$2$};
  \draw (0.057500, 0.695441) node[left] {$3$};
  \draw (0.057500, 0.753921) node[left] {$4$};
  \draw (0.057500, 0.812401) node[left] {$5$};
  \draw (0.057500, 0.870881) node[left] {$6$};
  \draw (0.057500, 0.929361) node[left] {$7$};
  \node[right] at (0.89, 0.07) {\small $\mu^*$};
  \node[left] at (0.026500, 0.96)  {\small $\lambda^*$};
\end{tikzoverlayabs}

%% file: plots/v1971.tex
\begin{tikzoverlayabs}[width=\matplotlibfigurewidth]{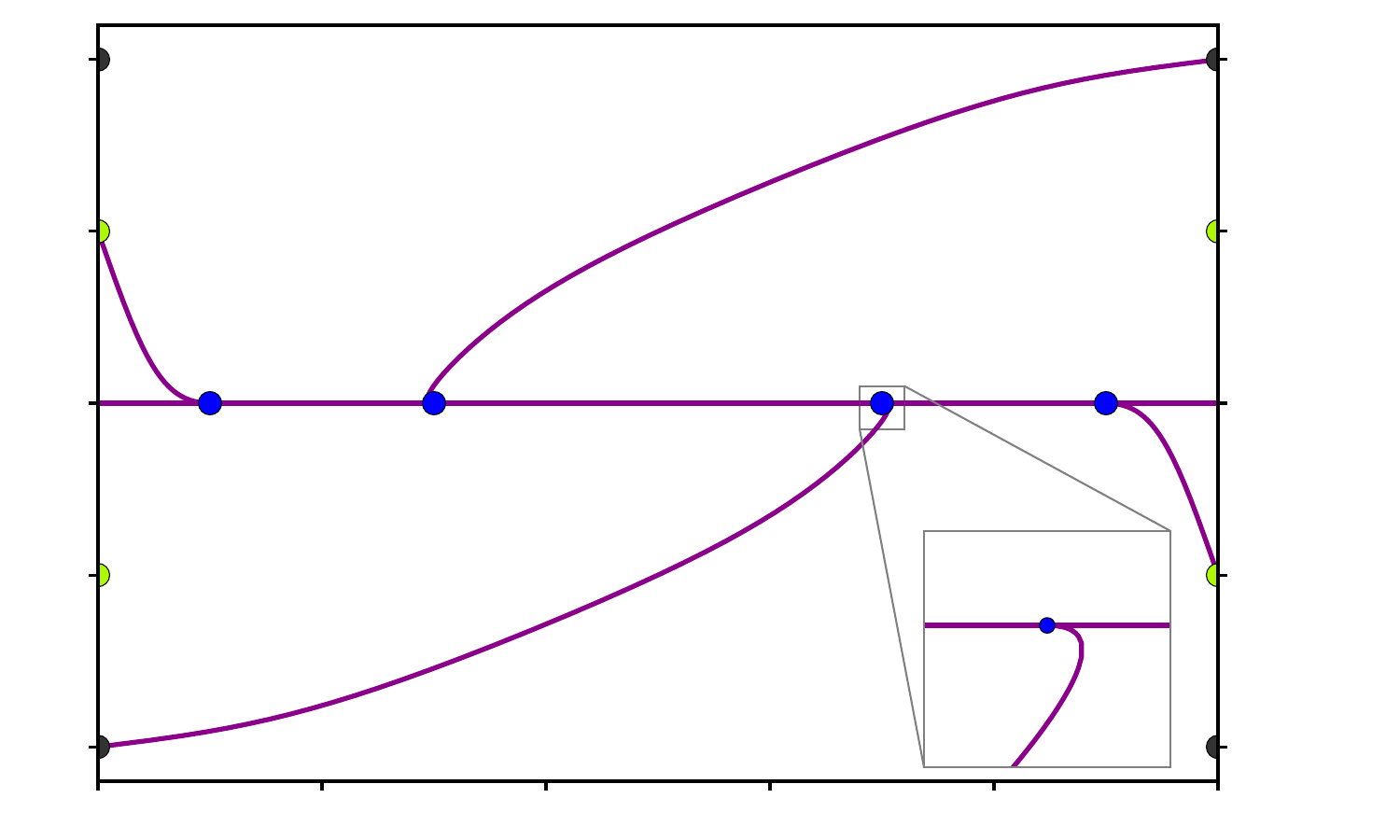}[\scriptsize]
  \draw (0.086000, 0.952000) node[below right] {\footnotesize $v1971$: genus = 4};
  \draw (0.070000, 0.049167) node[below] {$0.0$};
  \draw (0.230000, 0.049167) node[below] {$0.2$};
  \draw (0.390000, 0.049167) node[below] {$0.4$};
  \draw (0.550000, 0.049167) node[below] {$0.6$};
  \draw (0.710000, 0.049167) node[below] {$0.8$};
  \draw (0.870000, 0.049167) node[below] {$1.0$};
  \draw (0.057500, 0.110829) node[left] {$-2$};
  \draw (0.057500, 0.315415) node[left] {$-1$};
  \draw (0.057500, 0.520000) node[left] {$0$};
  \draw (0.057500, 0.724585) node[left] {$1$};
  \draw (0.057500, 0.929171) node[left] {$2$};
  \node[right] at (0.89, 0.07) {$\mu^*$};
  \node[left] at (0.026500, 0.96)  {$\lambda^*$};
\end{tikzoverlayabs}

%% file: plots/t12247.tex
\begin{tikzoverlayabs}[width=\matplotlibfigurewidth]{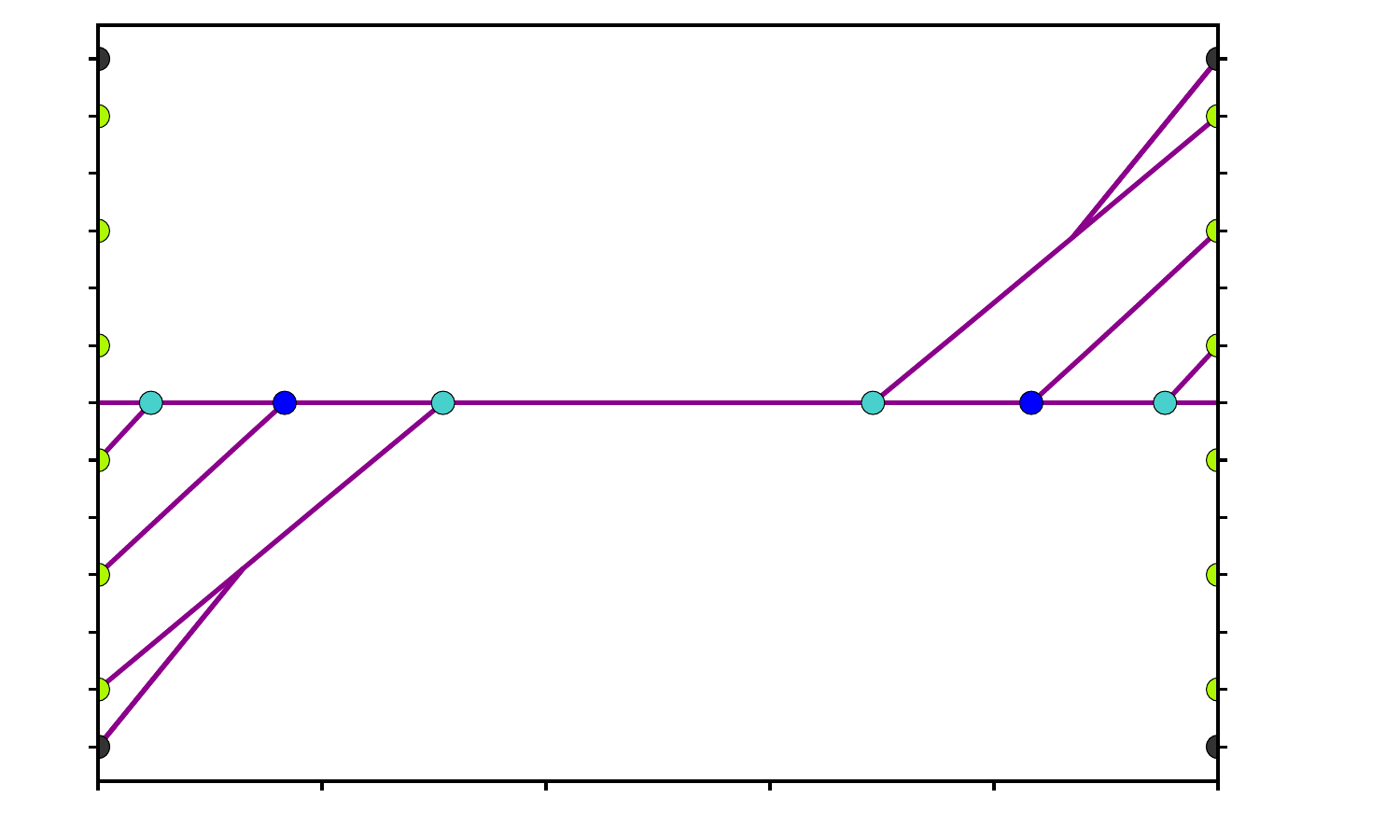}[\scriptsize]
  \draw (0.086000, 0.952000) node[below right] {\footnotesize $t12247$: genus = 5};
  \draw (0.070000, 0.049167) node[below] {$0.0$};
  \draw (0.230000, 0.049167) node[below] {$0.2$};
  \draw (0.390000, 0.049167) node[below] {$0.4$};
  \draw (0.550000, 0.049167) node[below] {$0.6$};
  \draw (0.710000, 0.049167) node[below] {$0.8$};
  \draw (0.870000, 0.049167) node[below] {$1.0$};
  \draw (0.057500, 0.110949) node[left] {$-6$};
  \draw (0.057500, 0.179198) node[left] {$-5$};
  \draw (0.057500, 0.247446) node[left] {$-4$};
  \draw (0.057500, 0.315694) node[left] {$-3$};
  \draw (0.057500, 0.383943) node[left] {$-2$};
  \draw (0.057500, 0.452191) node[left] {$-1$};
  \draw (0.057500, 0.520440) node[left] {$0$};
  \draw (0.057500, 0.588688) node[left] {$1$};
  \draw (0.057500, 0.656937) node[left] {$2$};
  \draw (0.057500, 0.725185) node[left] {$3$};
  \draw (0.057500, 0.793434) node[left] {$4$};
  \draw (0.057500, 0.861682) node[left] {$5$};
  \draw (0.057500, 0.929931) node[left] {$6$};
  \node[right] at (0.89, 0.07) {$\mu^*$};
  \node[left] at (0.026500, 0.96)  {$\lambda^*$};
\end{tikzoverlayabs}

%% file: plots/o9_30426.tex
\begin{tikzoverlayabs}[width=\matplotlibfigurewidth]{plots/images/o9_30426.pdf}[\matplotlibfigurefont]
  \draw (0.086000, 0.952000) node[below right] {\small $o9_{30426}$: genus 2};
  \draw (0.070000, 0.049167) node[below] {$0.0$};
  \draw (0.230000, 0.049167) node[below] {$0.2$};
  \draw (0.390000, 0.049167) node[below] {$0.4$};
  \draw (0.550000, 0.049167) node[below] {$0.6$};
  \draw (0.710000, 0.049167) node[below] {$0.8$};
  \draw (0.870000, 0.049167) node[below] {$1.0$};
  \draw (0.057500, 0.110748) node[left] {$-1$};
  \draw (0.057500, 0.520000) node[left] {$0$};
  \draw (0.057500, 0.929252) node[left] {$1$};
  \node[right] at (0.89, 0.07) {\small $\mu^*$};
  \node[left] at (0.026500, 0.96)  {\small $\lambda^*$};
\end{tikzoverlayabs}

%% file: plots/v0170.tex
\begin{tikzoverlayabs}[width=\matplotlibfigurewidth]{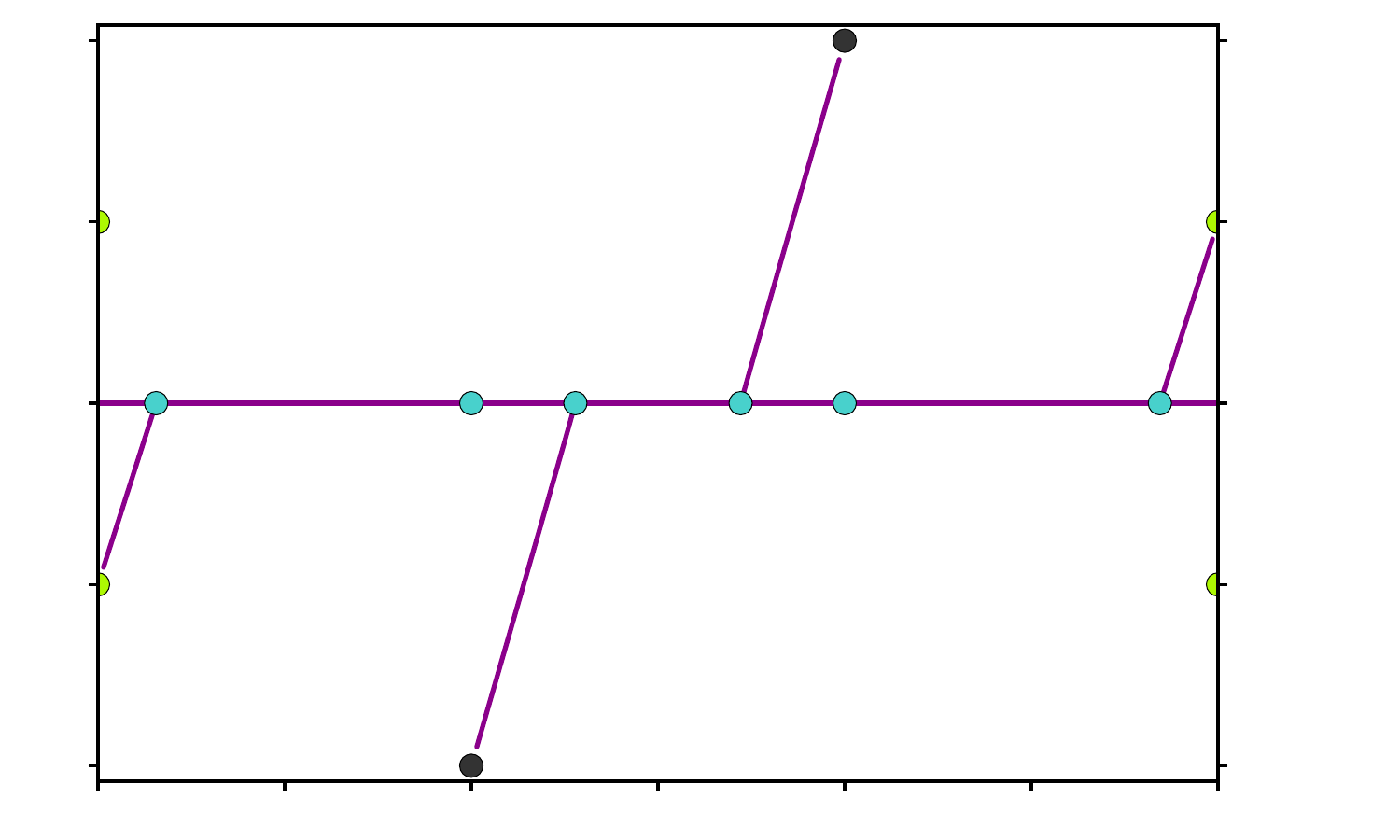}[\matplotlibfigurefont]
  \draw (0.086000, 0.952000) node[below right] {\small $v0170$};
  \draw (0.070000, 0.049167) node[below] {$0.0$};
  \draw (0.203333, 0.049167) node[below] {$0.5$};
  \draw (0.336667, 0.049167) node[below] {$1.0$};
  \draw (0.470000, 0.049167) node[below] {$1.5$};
  \draw (0.603333, 0.049167) node[below] {$2.0$};
  \draw (0.736667, 0.049167) node[below] {$2.5$};
  \draw (0.870000, 0.049167) node[below] {$3.0$};
  \draw (0.057500, 0.088447) node[left] {$-2$};
  \draw (0.057500, 0.304224) node[left] {$-1$};
  \draw (0.057500, 0.520000) node[left] {$0$};
  \draw (0.057500, 0.735776) node[left] {$1$};
  \draw (0.057500, 0.951553) node[left] {$2$};
  \node[right] at (0.89, 0.07) {\small $\mu^*$};
  \node[left] at (0.026500, 0.96)  {\small $\lambda^*$};
\end{tikzoverlayabs}

%% file: plots/v1108.tex
\begin{tikzoverlayabs}[width=\matplotlibfigurewidth]{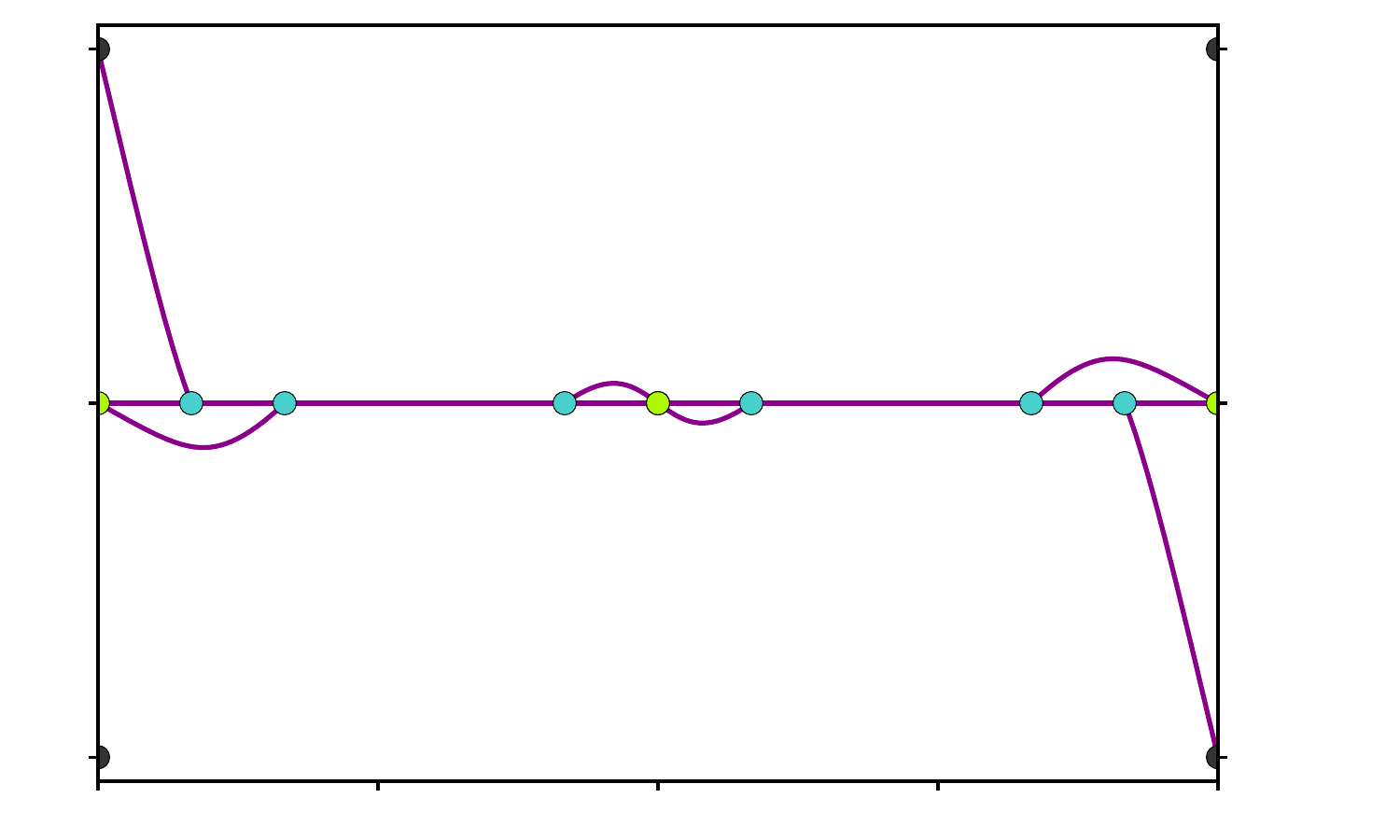}[\matplotlibfigurefont]
  \draw (0.086000, 0.952000) node[below right] {\small $v1108$};
  \draw (0.070000, 0.049167) node[below] {$0.0$};
  \draw (0.270000, 0.049167) node[below] {$0.5$};
  \draw (0.470000, 0.049167) node[below] {$1.0$};
  \draw (0.670000, 0.049167) node[below] {$1.5$};
  \draw (0.870000, 0.049167) node[below] {$2.0$};
  \draw (0.057500, 0.098563) node[left] {$-1$};
  \draw (0.057500, 0.520000) node[left] {$0$};
  \draw (0.057500, 0.941437) node[left] {$1$};
  \node[right] at (0.89, 0.07) {\small $\mu^*$};
  \node[left] at (0.026500, 0.96)  {\small $\lambda^*$};
\end{tikzoverlayabs}

%% file: plots/branched.tex
\newcommand{\yh}{1.32}
\newcommand{\telsubpath}{%
  \draw (0, 1) 
      .. controls (0.2, 0.95) and (0.43, 0.85) .. (0.43, 0.65) 
      .. controls (0.43, 0.47) and (0.2, 0.20) .. (0, 0);
}
\newcommand{\telpath}{%
  \telsubpath
  \begin{scope}[rotate=180]
    \telsubpath
  \end{scope}
}

\begin{tikzpicture}[scale=0.95, line cap=round, font=\small]
  % Picture downstairs 
  \begin{scope}
    % Vertical line from the lemma
    \draw[dashed, line width=1.5pt, color=blue!40] (0.3333334, \yh) -- (0.3333334, -1.8)
        node[below, color=black] {$\mu^* = \frac{1}{3}$}; 
    % Axes 
    \begin{scope}[line width=1.5pt, color=axesgray]
      \draw[->] (0, -\yh) -- (0, 1.5) node[left, color=black] {$\lambda^*$};
      \draw[->] (-1.6, 0) -- (1.6, 0) node[right, color=black] {$\mu^*$}; 
      \draw[dashed] (1, \yh) -- (1, -\yh) node[below right, color=black] {$\mu^* = 1$}; 
      \draw[dashed] (-1, \yh) -- (-1, -\yh) node[below left, color=black] {$\mu^* = -1$}; 
    \end{scope}
    
    % TEL
    \begin{scope}[color=locus, line width=1.5pt]
      \foreach \x in {-1, 0, 1}{
        \begin{scope}[shift={(\x, 0)}]
          \telpath
        \end{scope}
      }
    \end{scope}

    % Label
    \node at (0, -3) {$H^1(\partial M; \R)$};
  \end{scope}

  % Picture upstairs
  \begin{scope}[shift={(6.5, 0)}]
    
    % Axes
    \begin{scope}[line width=1.5pt, color=axesgray]
      \draw[->] (0, -\yh) -- (0, 1.5) node[left, color=black] {$\lambdatil^*$};
      \draw[->] (-3.5, 0) -- (3.9, 0) node[right, color=black] {$\mutil^*$}; 
      \foreach \x in {-1, 1}{
        \draw[dashed] (\x, \yh) -- (\x, -\yh) node[below, color=black] {$\mutil^* = \x$}; 
      }
      \foreach \x in {-3, -2, 2, 3}{
        \draw[dashed] (\x, \yh) -- (\x, -\yh);
      }
    \end{scope}

    % TEL: the clip box
    \begin{scope}
      \clip (-3.5, -\yh) rectangle (3.5, \yh);

      % TEL: the image of pi^*
      \begin{scope}[color=locus, line width=1.5pt]
        \foreach \x in {-3, 0, 3}{
          \begin{scope}[shift={(\x, 0)}, xscale=3]
            \telpath
          \end{scope}
        }
      \end{scope}

      % TEL: the new translates
      \begin{scope}[color=locus!30, line width=1.5pt]
        \foreach \x in {-5, -4, -2, -1, 1, 2, 4, 5}{
          \begin{scope}[shift={(\x, 0)}, xscale=3]
            \telpath
          \end{scope}
        }
      \end{scope}
    
    \end{scope} % TEL

    % Label
    \node at (0, -3) {$H^1(\partial \Mtil; \R)$};
  \end{scope} % Picture upstairs

\end{tikzpicture}

%% file: plots/arcs_to_cone.tex
\begin{tikzpicture}[line cap=round, font=\small]
  % The cone C.
  \coordinate (o) at (0,0);
  \coordinate (a) at (10:5);
  \coordinate (b) at (-10:5);
  \filldraw[color=blue!10] (o) -- (a) 
      arc [start angle=10, delta angle=-20, radius=5]
      -- cycle; 
  \begin{scope}[line width=1pt, color=blue!30]
    \draw (o) -- (a);
    \draw (o) -- (b);
  \end{scope}
  \node at (4.5, 0.4) {$\cC$};

  % The coordinate axis.
  \begin{scope}[line width=1.5pt, color=axesgray]
    \draw[->] (0, -1.5) -- (0, 2.0) node[left, color=black] {$\lambda^*$};
    \draw[->] (-1.0, 0) -- (6, 0) node[right, color=black] {$\mu^*$}; 
    \draw[dashed] (4, -1.52) -- (4, 2.0) 
        node[right, color=black] {$\mu^* = 1$}; 
  \end{scope}
  
  % The arc and its partner.  
  \draw[color=locus, line width=2pt] 
      (1, 0) .. controls (0.92, 0.4) and (0.8, 0.57) .. (0.7, 0.7);          
  \draw[color=locus, line width=2pt, shift={(4,0)}, rotate=180] 
      (1, 0) .. controls (0.92, 0.4) and (0.8, 0.57) .. (0.7, 0.7);    
  \node at (3.45, -0.95) {$B$};
  \node at (0.6, 0.92) {$A$};

  % The Alexander points.
  \begin{scope}[color=black, fill=galoisgeom, radius=2.2pt]
    \filldraw (1, 0)  circle;
    \filldraw (3, 0)  circle;
  \end{scope}
\end{tikzpicture}

%% file: plots/Aprime_arc.tex
\begin{tikzpicture}[scale=0.95, line cap=round, font=\small, radius=2.5pt]
  %\begin{scope}[line width=1.5pt, color=axesgray]
  %  \draw[->] (4, 0) -- (12, 0) node[right, color=black] {$\mu^*$}; 
  %\end{scope}

  \coordinate (A) at (0, -1);
  \coordinate (B) at (11, -1);
  \coordinate (C) at (0, -2.5);
  \coordinate (D) at (11, -2.5);
  \coordinate (E) at (0, -0.5);
  \coordinate (F) at (11, -3.0);
  \coordinate (G) at (3.9, -1);
  \coordinate (H) at (6, -2.5);

  \node[left] at (A) {$y=y_1$};
  \node[left] at (C) {$y=y_0$};
  \node[right] at (F) {$L_r$};

  \draw[name path=hor0, line width=0.9pt, color=black, loosely dashed] (A) -- (B); 
  \draw[name path=hor1, line width=0.9pt, color=black, loosely dashed] (C) -- (D); 
  \draw[color=blue, line width=1.5pt, name path=L] (E) -- (F);
  \path[name intersections={of=hor0 and L}];
  \coordinate (X) at (intersection-1);
  \path[name intersections={of=hor1 and L}];
  \coordinate (Y) at (intersection-1);

  \draw[color=locus, line width=2pt] (G) .. controls (3, -2.3) and (5, -1.7) .. (H);
  \filldraw (X)  circle node[above=4pt] {$(x_2, y_1)$};
  \filldraw (Y)  circle node[below=4pt] {$(x_3, y_0)$};
  \filldraw (G) circle node[above=4pt] {$(x_1', y_1)$};
  \filldraw (H) circle node[below=4pt] {$(x_0', y_0)$};
  \node at (3.47, -2.0) {$A'$};
  %\draw[color=locus, line width=2pt] 
  %    (0, 1.7) .. controls (0.25, 1.65) and (0.7, 1.4) .. (0.9, 1.1);
  %\node[above, color=black] at (0.7, 1.5) {$A$};
\end{tikzpicture}

%% file: plots/arc_from_parabolic.tex
\begin{tikzpicture}[line cap=round, font=\small]
  % left picture where k > 0.
  \begin{scope}[line width=1.5pt, color=axesgray]
    \draw[->] (0, -2.2) -- (0, 2.4) node[left, color=black] {$\lambda^*$};
    \draw[->] (-1.5, 0) -- (1.5, 0) node[right, color=black] {$\mu^*$}; 
  \end{scope}
  \draw[color=locus, line width=2pt] 
      (0, 1.7) .. controls (0.25, 1.65) and (0.7, 1.4) .. (0.9, 1.1);
  \node[above, color=black] at (0.7, 1.5) {$A$};
  \begin{scope}[color=black, fill=galoisgeom, radius=3pt]
    \filldraw (0, 1.7)  circle node[below left] {$(0, k)$};
  \end{scope}
  \node at (0, -2.8) {(i) When $k > 0$.};

  % right picture where k < 0

  \begin{scope}[shift={(7, 0)}]
    \begin{scope}[line width=1.5pt, color=axesgray]
      \draw[->] (0, -2.2) -- (0, 2.4) node[left, color=black] {$\lambda^*$};
      \draw[->] (-1.5, 0) -- (1.5, 0) node[right, color=black] {$\mu^*$}; 
    \end{scope}
    \draw[color=locus, line width=2pt] 
      (0, -1.7) .. controls (0.25, -1.65) and (0.7, -1.4) .. (0.9, -1.1);
      \node[above, color=black] at (0.9, -1.8) {$A$};
      \begin{scope}[color=black, fill=galoisgeom, radius=3pt]
        \filldraw (0, -1.7)  circle node[left=3pt] {$(0, k)$};
      \end{scope}
      \node at (0, -2.8) {(ii) When $k < 0$.};
  \end{scope}
\end{tikzpicture}